\PassOptionsToPackage{usenames, dvipsnames}{color}
\documentclass[a4paper,12pt,english,reqno]{amsart}
\usepackage[new]{old-arrows}
\usepackage{amsmath,amsfonts,amssymb,mathrsfs,stackrel}
\usepackage[english]{babel}
\usepackage[utf8]{inputenc}
\usepackage[T1]{fontenc}
\usepackage{graphicx}
\usepackage{tabularx}
\usepackage{mathtools}
\usepackage{hyperref}
\usepackage{cleveref}
\usepackage{comment}
\usepackage[all]{xy}
\usepackage{xurl}
\usepackage{soul}
\numberwithin{equation}{subsection}
\usepackage{blkarray}
\usepackage{enumitem}

\usepackage{extpfeil}

\usepackage[backend=bibtex,style=alphabetic, maxnames=99]{biblatex}

\hypersetup{colorlinks,citecolor=blue,linktocpage,hyperindex=true,backref=true}

\newcommand\cA{{\mathcal A}}

\newcommand\cC{{\mathcal C}}

\newcommand\cM{{\mathcal M}}

\newcommand\bA{{\mathbb A}}

\newcommand\C{{\mathbb C}}

\newcommand\bF{{\mathbb F}}

\newcommand\R{{\mathbb R}}

\newcommand\Z{{\mathbb Z}}

\newcommand\uR{{\underline R}}

\newcommand\Alb{{\rm Alb}}

\newcommand\ch{{\rm ch}}
\newcommand\A{{\mathbb A}}

\DeclareMathOperator{\Hom}{Hom}
\DeclareMathOperator{\im}{im}

\DeclareMathOperator{\ord}{ord}

\DeclareMathOperator{\Pic}{Pic}

\DeclareMathOperator{\rank}{rank}
\DeclareMathOperator{\rk}{rk}

\DeclareMathOperator{\pr}{pr}
\DeclareMathOperator{\alb}{alb}

\newcommand{\set}[1]{\left\{ #1 \right\}}

\newcommand{\ca}[1]{{\mathcal{#1}}}

\def\F{{\mathbb F}}

\usepackage{amssymb,tikz}

\usepackage{dsfont}

\newtheorem{theorem}{Theorem}[section]

\newtheorem{corollary}[theorem]{Corollary}
\newtheorem{lemma}[theorem]{Lemma}

\newtheorem{proposition}[theorem]{Proposition}

\theoremstyle{definition}

\newtheorem{definition}[theorem]{Definition}

\newtheorem{problem}[theorem]{Problem}

\newtheorem{remark}[theorem]{Remark}

\newtheorem{step}{Step}
\newtheorem{stepp}{Step}

\newtheorem{case}{Case} 

\newcommand{\bprimezero}{t}
\newcommand{\bzero}{t_0}
\newcommand{\ts}{h_s}
\newcommand{\nb}{{\rm nb}}
\newcommand{\red}{{\rm red}}

\usepackage[usenames,dvipsnames]{color}

\usepackage{listings}

\lstdefinelanguage{Magma}{
  morekeywords={for, in, do, end, if, then, else, elif, function, return, true, false, repeat, until, while, break, continue},
  sensitive=true,
  morecomment=[l]{//},
  morestring=[b]",
}

\title[Matroids and the integral Hodge conjecture]{Matroids and the integral Hodge conjecture for abelian varieties} 

\author[Engel]{Philip Engel}
\address{Department of Mathematics, Statistics, 
and Computer Science, University of Illinois in Chicago (UIC),
851 S Morgan St, Chicago, IL 60607, USA.}
\email{pengel@uic.edu}
\author[de Gaay Fortman]{Olivier de Gaay Fortman}
\address{Department of Mathematics, Utrecht University,  Budapestlaan 6, 3584 CD Utrecht, The Netherlands.}
\email{a.o.d.degaayfortman@uu.nl}
\author[Schreieder]{Stefan Schreieder}
\address{Leibniz University Hannover, Institute of Algebraic Geometry, Welfengarten 1, 30167 Hannover, Germany.}
\email{schreieder@math.uni-hannover.de}

\begin{document}

\date{\today}




\maketitle

\begin{abstract}  
We prove that the cohomology class of any curve on a very general principally polarized abelian variety of dimension at least 4 is an even multiple of the minimal class.
The same holds for the intermediate Jacobian of a very general cubic threefold.
This disproves the integral Hodge conjecture for abelian varieties and shows that very general cubic threefolds are not stably rational.
Our proof is motivated by tropical geometry; 
it relies on multivariable Mumford constructions, monodromy considerations, and the combinatorial theory of matroids.
\end{abstract}

\tableofcontents

\section{Introduction}

To a regular matroid $\uR$ of rank $g$ on an $n$ element
set, one may associate a degeneration $\pi\colon X\to B$ of 
principally polarized abelian varieties of dimension $g$ over a $n$-dimensional base $B$.
The goal of this paper is to connect combinatorial properties of $\uR$ to algebro-geometric properties of the very general fiber of $\pi$---primarily, the existence of algebraic curves representing multiples of the minimal class. 
This leads to solutions to two long-standing open problems, one in the field of algebraic cycles on abelian varieties and one concerning rationality questions.

\subsection{Integral Hodge conjecture for abelian varieties} 
\label{subsec:intro:1}
By \cite{atiyah-hirzebruch,kollar1992trento}, 
the integral Hodge conjecture fails in general. 
So its holding is a property of a variety.
Understanding whether or not this 
property holds for a particular class of varieties is crucial
for understanding their geometry. 

An important  case that has been open until now is
the integral Hodge conjecture for abelian varieties. 
The specific case of curve classes 
on principally polarized abelian varieties  
is of particular interest; 
this question goes back (at least) to the 
work of Barton and Clemens 
\cite[p.~66]{barton-clemens} 
and has been advertised for instance by
Voisin \cite{voisin-universalCHgroup,voisin2022cycle}, 
who found a relation to the stable rationality 
problem for rationally connected threefolds.

If $(X,\Theta)$ is a very general principally polarized abelian variety of dimension $g$, then the group of Hodge classes in $H^{2c}(X,\Z)$ is generated by $[\Theta]^c/c!$, which is well-known to be a primitive integral cohomology class.
Therefore, the integral Hodge conjecture for such $X$ reduces to the question which multiples of $[\Theta]^c/c!$  
are represented by a linear combination of closed subvarieties. 
Our result on this problem is as follows. 

\begin{theorem} \label{thm:main:IHC:intro}
Let $(X,\Theta)$ be a very general principally polarized abelian variety of dimension $g\geq 4$. 
Let $Z\in {\rm CH}^c(X)$ be an algebraic cycle of codimension $2\leq c\leq g-1$. 
Then $[Z]=m\cdot [\Theta]^{c}/c!\in H^{2c}(X,\Z)$ with $m$ even.
\end{theorem}

The above theorem is equivalent to saying that for any  cycle $Z\in {\rm CH}^c(X)$ of codimension $2\leq c\leq g-1$, the intersection number $Z\cdot \Theta^{g-c}$ is divisible by $2\cdot \frac{g!}{c!}$.
For instance, Theorem \ref{thm:main:IHC:intro} implies that for any curve $C \subset X$, the intersection number $C \cdot \Theta$ of $C$ with $\Theta$ is divisible by $2g$. 
The proof of Theorem \ref{thm:main:IHC:intro} quickly reduces to the case where $c=g-1$, i.e.\ to the case of curve classes that we will handle in this paper.

While the rational Hodge conjecture for abelian varieties of dimension four and five 
has recently been proved by Markman \cite{markman, markman2}, the image of the integral cycle class map remains mysterious. 
Theorem \ref{thm:main:IHC:intro} solves that problem for the very general principally polarized abelian four- and fivefold:

\begin{corollary} \label{cor:abelian-4-folds-IHC}
Let $(X,\Theta)$ be a very general principally polarized abelian of dimension $g\in\{4,5\}$. 
Then 
$$
\im({\rm cl}^c\colon {\rm CH}^c(X)\to H^{2c}(X,\Z))=\begin{cases}
2\Z[\Theta]^c/c!\quad &\text{if $2\leq c\leq g-1$ }\\
\Z[\Theta]^c/c! \quad &\text{otherwise}.
\end{cases}
$$
\end{corollary}

\subsection{Cubic threefolds} \label{subsec:intro:2} 

Our method is flexible and works for 
other families of abelian varieties, 
including the important case of 
intermediate Jacobians of cubic threefolds:

\begin{theorem} \label{thm:main:cubics:intro}
Let $Y\subset \mathbb P^4_{\C}$ be a very general cubic hypersurface.
Then the homology class of any curve $C\subset JY$ on its intermediate Jacobian $JY$ is an even multiple of the minimal class $[\Theta_Y]^4/4!\in H_2(JY,\Z)$.
\end{theorem}

By \cite{voisin-universalCHgroup}, Theorem \ref{thm:main:cubics:intro} implies:

\begin{corollary} \label{cor:cubic-intro}
    Very general cubic threefolds $Y\subset \mathbb P^4_{\C}$ do not admit a decomposition of the diagonal. Hence they are neither stably rational, nor retract rational, nor $\mathbb A^1$-connected.
\end{corollary}
 
The above result should be compared to the celebrated 
result of Clemens--Griffiths 
\cite{clemensgriffiths-cubicthreefold}, who showed 
that for a smooth cubic threefold $Y$, the intermediate
Jacobian $(JY,\Theta_Y)$ is not isomorphic to a 
product of Jacobians, which by Matsusaka's 
criterion is equivalent to saying that the 
minimal class $[\Theta_Y]^4/4!$ is not represented 
by the class of an effective algebraic curve on $JY$.
This implied that $Y$ is not rational, while the 
stronger conclusion that $[\Theta_Y]^4/4!$ is not 
algebraic for very general $Y$ implies that such 
cubics are not stably rational.
While every smooth cubic threefold is 
irrational, the smooth cubic
threefolds which admit a decomposition of 
the diagonal form a (non-empty) countable
union of subvarieties of their moduli, 
see \cite{voisin-universalCHgroup}.

We remark that in contrast to the case treated in Theorem \ref{thm:main:IHC:intro}, the minimal surface class $[\Theta_Y]^3/3!$ is algebraic and in fact effective on $JY$, represented by the Abel--Jacobi image of the Fano surface of lines on $Y$, see \cite{clemensgriffiths-cubicthreefold}.

In \cite{beckmann-degaayfortman}, it has been shown that the integral Hodge conjecture holds for products of Jacobians of curves,
and hence for abelian varieties that admit a split embedding into such a product. 
Voisin \cite{voisin2022cycle} proved the converse, showing that in fact the integral Hodge conjecture for curve classes on abelian varieties is equivalent to the statement that any abelian variety admits a split embedding into a product of Jacobians of curves.
Our results show that this fails in general:

\begin{corollary} \label{cor:isogeny-product}
Let $(X,\Theta)$ be a very general principally polarized abelian variety of dimension at least $4$  or the intermediate Jacobian of a very general cubic threefold.
Then, for any abelian variety $Y$, any isogeny $f\colon X\times Y\to \prod JC_i$ to a product of Jacobians of curves has even degree.
In particular, $X$ does not admit a split embedding into a product of Jacobians of curves.
\end{corollary}

Note that any abelian variety $X$ is, up to isogeny, a factor in the 
Jacobian of a curve, namely in the Jacobian of any curve $C\subset X$ which generates $X$ as an abelian variety.
The above corollary shows that in general, there are restrictions on the degree of 
the resulting isogeny $f \colon X \times Y \to JC$. 
Related results have been discussed in \cite{dGF-Sch}.

\subsection{Algebraic curves on matroidal families} \label{subsec:algcurves-matroidal}
In this section we explain a more precise result, which implies the results in Sections \ref{subsec:intro:1} and \ref{subsec:intro:2}.

Let $\underline R$ be a regular matroid of rank $g$ on a ground set $S$.
To this data we may associate a smooth projective family $ X^\star_{(\Delta^\star)^S}\to (\Delta^\star)^S$ of $g$-dimensional principally polarized abelian varieties  over the punctured polydisk $(\Delta^\star)^S$ with base point $t$, such that the associated vanishing cycles $y_s$, viewed via the polarization as linear forms on ${\rm gr}^W_0H_1(X_{t},\Z) $, realize the matroid $\underline R$. 
The monodromy about the $s$-th coordinate hyperplane is then encoded by some multiple of the quadratic form $y_s^2$, cf.\ Definition \ref{def:monodromy-bilinear-form}. 
We may further choose the family so that it 
extends to an algebraic family $\pi^\star\colon X^\star\to B^\star$ over a smooth quasi-projective base $B^\star$ with $(\Delta^\star)^S\subset B^\star$, see \cite[Propositions 4.10]{survey}.  
We call $\pi^\star$ a \emph{matroidal family} 
associated to $\underline R$, see Definition \ref{def:matroidal-family} below.
 
\begin{theorem} \label{thm:matroidal-intro}
Let $\uR$ be a regular matroid of rank $g$. 
Let $\pi^\star\colon X^\star\to B^\star$ be a matroidal family of $g$-dimensional principally polarized abelian 
varieties associated to $\uR$.
Let $\iota\colon C_t\to X_t$ be a morphism from a 
projective curve into a very 
general fiber of $\pi$, with
$$
\iota_\ast [C_t]=m\cdot 
[\Theta]^{g-1}/(g-1)!\in H_2(X_t,\Z).
$$ 
If $\uR$ is not cographic, 
then $m$ is even.
\end{theorem}

A matroid $M^*(G)$ is {\it cographic} if it can be 
realized by the natural map $E\to H_1(G,\Z)^\ast$,
where $G$ is some oriented graph with edge set $E=E(G)$. 
On the other hand, the map $E\to \Z^E/H_1(G,\Z)$ 
defines the {\it graphic} matroid $M(G)$.
The graphic matroid of $G$
is isomorphic to a cographic matroid
if and only if $G$ is planar.
In addition to graphic and cographic matroids, 
there are regular matroids that are not 
related to graphs at all, such as the 
$\uR_{10}$ matroid or any matroid 
with $\uR_{10}$ as a minor, 
cf.\ \cite[Corollary 13.2.5]{oxley-matroids}. 

Our obstruction is topological in the sense that the matroidal information is encoded in the monodromy, which in turn depends only on the topological type of the family $ X^\star_{(\Delta^\star)^S}\to (\Delta^\star)^S$ over the punctured polydisc, that is, only on the homeomorphism
type of the corresponding real torus bundle over $(S^1)^S$. 
Cographic matroids are precisely those that appear as the vanishing cycles of a degeneration of curves.  
In fact,  $\underline R$ is cographic if and only if $ X^\star_{(\Delta^\star)^S}\to (\Delta^\star)^S$ deforms to a family of Jacobians of curves, see \cite[Remark 2.31]{survey} and Corollary \ref{cor:matroidal-topological-body} below.  
The above theorem thus shows that for a matroidal family of principally polarized abelian varieties, the integral Hodge conjecture fails 
for curve classes on very general fibers, unless the family is deformation equivalent over 
$(\Delta^\star)^S$ to a family of Jacobians, see Lemma \ref{lem:alg-min->l-prime-multiple-is-effective} and Corollary \ref{cor:matroidal-topological-body} below.

Theorem \ref{thm:main:IHC:intro} for curve classes on abelian fourfolds will follow by applying Theorem \ref{thm:matroidal-intro} to the graphic matroid $M(K_5)$ associated to the complete graph $K_5$, which is not planar, hence the graphic matroid $M(K_5)$ of rank 4 is not cographic.
Similarly, Theorem \ref{thm:main:cubics:intro} follows by applying Theorem \ref{thm:matroidal-intro} to a matroidal family of intermediate Jacobians of cubic threefolds associated to $\underline R_{10}$ due to Gwena \cite{gwena}, see Section \ref{sec:application} below.

\subsection{Obstruction and combinatorial results} \label{subsec:obstruction}
We sketch some ideas behind Theorem 
\ref{thm:matroidal-intro}. 
Replacing $C_t$ by its normalization, 
we may assume that $C_t$ is smooth projective.
We then get a morphism $f\colon JC_t\to X_t$.
Using the principal polarizations on both sides, 
the dual of $f$ defines a map $f^\vee:X_t\to JC_t$.
Since $f_\ast[C_t]=m\cdot [\Theta]^{g-1}/(g-1)!$, 
we find that $(f^\vee)^\ast [\Theta_{C_t}]=m\cdot [\Theta]$ 
and the composition $f\circ f^\vee\colon X_t\to X_t$ is 
multiplication by $m$.

Let $\Lambda$ be a ring in which $m$ is invertible.
Then the above observations show that 
$\frac{1}{m}f_\ast^\vee$ splits the map 
$f_\ast$ and we get a canonical decomposition
\begin{align} \label{eq:decomposition-intro}
H_1(C_t,\Lambda)\cong  H_1(X_t,\Lambda) \oplus 
\ker(f_\ast\colon H_1(C_t,\Lambda)\to H_1(X_t,\Lambda)).
\end{align}
One checks that this decomposition is orthogonal 
with respect to $\Theta_{C_t}$ 
(Lemma \ref{lem:orthogonal-complement}).
Since $t\in B^\star$ is very general, there exists a generically finite cover ${B'}^\star \to B^\star$ such that we may spread out the curve $C_t$ over an open subset of ${B'}^\star$. 
To explain the mechanism of our proof, we will 
ignore this (important) technical difficulty
and assume that ${B'}^\star=B^\star$, i.e.\ that the above decomposition holds 
over a Zariski open subset of $B^\star$,
and hence is respected by all monodromy operators. 

Let us now fix $t\in (\Delta^\star)^S\subset B^\star$.
Then the decomposition \eqref{eq:decomposition-intro} 
is respected by the monodromy operator $T_s$ 
about the $s$-th coordinate hyperplane 
of $\Delta^S$.
Assume moreover that $T_s$ is unipotent with 
nilpotent logarithm $N_s$ and note that the decomposition 
\eqref{eq:decomposition-intro} is compatible with the monodromy weight filtration.  
Taking ${\rm gr}^W_0$,
we obtain a decomposition
$$
{\rm gr}^W_0H_1(C_t)\cong 
{\rm gr}^W_0H_1(X_t) \oplus \ker(f_\ast \colon 
{\rm gr}^W_0H_1(C_t)\to 
{\rm gr}^W_0H_1(X_t)) ,
$$
where homology is taken with $\Lambda$ coefficients, and where we assume $\Lambda\subset \R$ for simplicity.
We may then define the monodromy bilinear form $Q_s$ on 
${\rm gr}^W_0H_1(C_t)$ by the formula $Q_s(u,v)\coloneqq \Theta_{C_t}(N_su,v)$, see Definition 
\ref{def:monodromy-bilinear-form} below.
The above decomposition is orthogonal 
with respect to $Q_s$ and the restriction of $Q_s$ 
to ${\rm gr}^W_0H_1(X_t,\Lambda)$ agrees with $m$ times
the monodromy bilinear form $B_s(u,v)=\Theta(N_su,v)$.
By the assumption that $\pi^\star \colon X^\star \to B^\star$ 
is matroidal, these monodromies are given by a multiple 
of the bilinear form $u\otimes v\mapsto y_s(u)y_s(v)$, 
which we denote by $y_s^2$.

Let us, in addition, assume that the curve $C_t$ has a unique nodal limit $C_0$ over the origin of $\Delta^S$.
Then ${\rm gr}^W_0H_1(C_t,\Z)\cong H_1(\Gamma(C_0),\Z)$ identifies to the homology of the dual complex of $C_0$, cf.~\cite[Proposition 5.12]{survey}. 
Under all of the above assumptions, we then see that $(\underline R,S)$ admits a $\Lambda$-splitting of level $d=m$ in the cographic matroid $M^*(G)$ of some graph $G$, in the following sense (cf.\ Section \ref{subsec:degenerations-abelian-varieties} below):

\begin{definition} \label{def:d-Lambda-splitting-general}
Let $(\underline R,S)$ and 
$(\underline M,E)$ be regular matroids 
with integral realizations $S\to U^\ast$, $s\mapsto y_s$, 
and $E\to V^\ast$, $e\mapsto x_e$, respectively.  
Let $\Lambda$ be a ring and let $d$ be a positive integer.
    A {\it quadratic $\Lambda$-splitting of level $d$} 
    of $(\underline R,S)$ in $(\underline M,E)$ 
    consists of an embedding $U_\Lambda\hookrightarrow V_\Lambda$ and a decomposition
        \begin{align} \label{eq:def-H_1(G)=U+U'-intro}
        V_{\Lambda}=U_\Lambda\oplus U'
        \end{align} 
        for some $U'\subset V_\Lambda$, 
        together with a matrix $(a_{se})\in \Z_{\geq 0}^{S\times E}$ of non-negative integers, such that, for all $s\in S$, the bilinear form  $Q_s\coloneqq \sum_{e}a_{se}x_e^2$ has the following properties:
        \begin{enumerate}
            \item the decomposition \eqref{eq:def-H_1(G)=U+U'-intro} is orthogonal with respect to $Q_s$;
            \item the restriction of $Q_s$ to $U_\Lambda$ agrees with $d\cdot y^2_s$.
        \end{enumerate} 
\end{definition}

The above definition is key to our proof of Theorem \ref{thm:matroidal-intro}, which falls naturally into two parts: the reduction to combinatorics,  and the proof of the resulting combinatorial statement.
The reduction to combinatorics is achieved by the following algebro-geometric result, which is proven without any of the aforementioned simplifying assumptions.

\begin{theorem} \label{thm:reduction-to-combinatorics-intro}
Let $\pi^\star \colon X^\star\to B^\star$ be a matroidal family
of principally polarized abelian varieties associated to a regular matroid $(\underline R,S)$. 
Let $\iota\colon C_t\to X_t$ be a morphism from a projective curve into a very general fiber of $\pi^\star$, with 
$$
\iota_\ast [C_t]=m\cdot [\Theta]^{g-1}/(g-1)!\in H_2(X_t,\Z).
$$ 
Let $\ell$ be a prime that is coprime to $m$.
Then there is a positive integer $d$ such that $(\underline R,S)$ admits a quadratic $\Z_{(\ell)}$-splitting of level $d$ in a cographic matroid.
\end{theorem}

The main combinatorial result that we prove and which allows us to finish the proof of Theorem \ref{thm:matroidal-intro} is then the following:
 
\begin{theorem}\label{thm:splitting-in-cographic=cographic-intro}
Let $(\underline R,S)$ be a regular matroid with integral realization $S\to U^\ast$.
Then $(\underline R,S)$ admits a quadratic $\Z_{(2)}$-splitting of some level $d \geq 1$ in a cographic matroid if and only if $(\underline R,S)$ is cographic.
\end{theorem}

In contrast to the above result, there are many non-cographic matroids, such as 
$M(K_5)$, $M(K_{3,3})$, and $\uR_{10}$, 
that admit quadratic $\Z_{(\ell)}$-splittings in cographic matroids for all odd primes $\ell$, see Remark \ref{rem:K5,K33,R10-ell-odd} below.
Moreover, we have the following general result.

\begin{theorem}\label{thm:splitting-in-cographic=cographic-intro-positive-result}
Let $(\underline R,S)$ be a regular matroid of rank $g$ and with integral realization $S\to U^\ast$. 
If $\ell\geq g$ is a prime, then $(\underline R,S)$ admits a  quadratic $\Z_{(\ell)}$-splitting of some level $d$ in a cographic matroid.
\end{theorem}

\subsection{Outline} \label{subsec:outline} 
As aforementioned, our main results follow from Theorem \ref{thm:matroidal-intro}, which itself reduces to 
Theorems \ref{thm:reduction-to-combinatorics-intro} and \ref{thm:splitting-in-cographic=cographic-intro}:~To prove Theorem \ref{thm:matroidal-intro}, assume contrapositively
that $m$ is odd.  
Then Theorem 
\ref{thm:reduction-to-combinatorics-intro} 
implies that there is a $\Z_{(2)}$-splitting of $\uR$
of some level $d$ in a cographic matroid. Hence, by Theorem \ref{thm:splitting-in-cographic=cographic-intro}, $\underline R$ is cographic. 

\subsubsection{Outline of the proof of Theorem \ref{thm:reduction-to-combinatorics-intro}}
In the set-up of Theorem \ref{thm:reduction-to-combinatorics-intro}, we choose a smooth partial compactification $B^\star\subset B$ such that $H=B\setminus B^\star$ is an snc divisor on $B$, $(\Delta^\star)^S\subset B^\star$ extends to an embedding $\Delta^S\subset B$, and $H \cap \Delta^S$ is given by the vanishing of the coordinate functions.
We further perform the necessary 
base change $\tau\colon B'\to B$ 
over which the aforementioned curve $C_t$ spreads out.
Here we want $B'$ to be regular and the preimage of the 
coordinate axes to be snc 
(approaches with 
singular $B'$ and finite $\tau$ were unsuccessful). 
In particular, $\tau$ will not be finite in general.
In fact, $\tau^{-1}(0)$ can be assumed to be an snc 
divisor, which implies that we do not have a single 
limit curve $C_0$ at our disposal anymore.
This poses serious technical difficulties. 
For instance, using the universal family over $\overline{\mathcal M}_g$, we may after further blow-up assume that $C_t$ extends to a nodal family over $B'$. But this only allows one to deduce a ``patch'' version of the notion in Definition \ref{def:d-Lambda-splitting-general}---that is, a weaker version
corresponding to patching together the monodromy information we obtain at the deepest strata of $\tau^{-1}(0)$. 
However, there are examples that show that this weaker 
combinatorial property is not enough to conclude 
what we want and that, moreover, the regularity of $\underline R$ has to play an important role. 
The crucial idea to circumvent this issue is to first control the limits of the abelian varieties and then to produce a family of curves 
in such a way that the limits of the curves 
are closely related to 
the limits of the abelian varieties. 
This provides enough control of the 
combinatorics of the 
singular curves, allowing 
us to deduce 
Theorem \ref{thm:reduction-to-combinatorics-intro}. 

The key starting input, 
exploiting the regularity of $\underline R$,
is \cite[Theorem 7.1]{survey},
which produces a regular, flat 
extension $\pi\colon X\to B$ of the 
 matroidal family $\pi^\star\colon 
X^\star\to B^\star$, such that 
$\pi$ is {\it $H$-nodal}, see Definition
\ref{def:D-nodal}. 
We perform a detailed analysis of a 
specific resolution 
$X''\to X'=X\times_BB'$ which is 
semistable over $B'$
in Section \ref{sec:admissible-resolutions}.  
In Section \ref{sec:algebraicity->d-QS}, we replace $C_t$ by 
a particularly nice representative of some $\ell$-prime multiple of the minimal class which is a complete intersection in some projective bundle over $X''$.
We then use the resolution $X'' \to X'$ to construct 
a graph $G$ via a  careful gluing procedure of pieces of the dual graphs
of the limits of the curve $C_t$ over certain snc strata of $\tau^{-1}(0)$. 

The construction of $G$ begins by  
relating the dual graph of each
limiting nodal curve to the $1$-skeleton of the dual complex of the corresponding 
fiber of $X''\to B'$.
By the explicit resolution $X''\to X'$, 
we relate that $1$-skeleton to the $1$-skeleton of 
the dual complex of the respective fiber of $X'\to B'$; 
the latter is constant along $\tau^{-1}(0)$, because $X'=X\times_BB'$ is a fiber product.
This allows us to construct $G$ via 
``edge multiplication'' from the $1$-skeleton $\Gamma^1(X_0)$ of the dual complex of $X_0$. 
Our analysis shows that $\uR$ admits a quadratic $\Z_{(\ell)}$-splitting of some level $d$ into the cographic matroid $M^*(G)$. 

\subsubsection{Outline of the proof of Theorem \ref{thm:splitting-in-cographic=cographic-intro}} 
In Section \ref{sec:splitting->solution}, we show that a quadratic $\Z_{(\ell)}$-splitting of level $d$ of $(\uR,S)$ in the cographic matroid $M^*(G)$ 
yields {\it $\Z_{(\ell)}$-solutions} in 
$G$; 
these are certain linear combinations of edges of $G$
satisfying linear conditions dictated by $\uR$.
We then define  
Albanese graphs $\Alb_{\ell ^r,\ell ^j}\coloneq \Alb_{\ell ^r,\ell ^j}(\underline R)$  associated to $(\underline R,S)$. 
These Albanese graphs satisfy 
a universal property,  
yielding a combinatorial Albanese map 
$\alb \colon G\to \Alb_{\ell ^r,\ell ^j}$ 
along which we can push forward a given solution. 
In this way, we also obtain $\Z_{(\ell)}$-solutions in the graphs $\Alb_{\ell ^r,\ell ^j}$.
Applying $\otimes_{\Z_{(\ell)}}\Lambda$, we obtain such solutions for any $\Z_{(\ell)}$-algebra $\Lambda$.

In Section \ref{sec:l^j-indivisible->l-indivisible}, 
we study divisibility of solutions 
to go from 
$\ell^i$-indivisible $\Lambda$-solutions in 
$\Alb_{\ell ^r,\ell ^j}$ to $\ell ^{i-j}$-indivisible 
$\Lambda$-solutions in $\Alb_{\ell^{r-j},1}$. 
Applying this to the case $r=j+1$ 
and $\ell=2$, it then suffices to show that the existence of $\Lambda$-solutions in $\Alb_{2,1}$ implies that $\underline R$ is cographic.
In Section \ref{sec:2-dindivisible->cographic}, we show 
that the property of having an 
$\ell^{i}$-indivisible $\Lambda$-solution in 
$\Alb_{\ell ^r,\ell ^j}$ is closed 
under taking minors.
By a theorem of Tutte, the class of cographic matroids 
can be described as those regular matroids that 
do not have $M(K_5)$ and $M(K_{3,3})$ as minors.
We will use this to reduce Theorem \ref{thm:splitting-in-cographic=cographic-intro}
to showing that those two 
excluded minors do not have $2$-indivisible 
$\Z/2$-solutions in $\Alb_{2,1}$.
This in turn reduces to an explicit rank 
computation of certain matrices, see Proposition \ref{prop:K_5-K_3,3}.

\begin{remark}
Recall that the integers $a_{se}$ in Definition \ref{def:d-Lambda-splitting-general} are non-negative; Theorem \ref{thm:splitting-in-cographic=cographic-intro} fails if one drops that assumption.
This explains the choice of coefficients $\Z_{(\ell)}$ (opposed to $\Z_\ell$, say), because $\Z_{(\ell)}\subset  \R$ and we will exploit this in the proof, see Theorem \ref{thm:d-QE->solutions} and Lemma \ref{lem:elementary} below.  
\end{remark}

\subsection{Acknowledgements}
We thank Olivier Benoist, Yano Casalaina--Martin, and \linebreak
Evgeny Shinder for useful conversations and comments. 
PE was partially supported by NSF grant DMS-2401104.
OdGF and StS have received funding from the European Research Council (ERC) under the
European Union’s Horizon 2020 research and innovation programme under grant agreement
N\textsuperscript{\underline{o}} 948066 (ERC-StG RationAlgic). 
OdGF has also received funding from the ERC Consolidator Grant FourSurf N\textsuperscript{\underline{o}} 101087365. 
The research was partly conducted in the framework of the DFG-funded research training group RTG 2965: From Geometry to Numbers, Project number 512730679. 

\section{Preliminaries}

\subsection{Conventions}
\subsubsection{Schemes and analytic spaces}
We work, if not mentioned otherwise, 
over the field of complex numbers. 
We frequently identify a complex algebraic scheme or variety with the corresponding complex analytic space.
If $Y\to B$ and $B'\to B$ are morphisms of 
complex analytic spaces, 
we denote the corresponding base change
by $Y_{B'}\coloneq Y\times_BB'$.
We denote by 
$\Delta\subset \C$ the open unit disc.

An snc divisor $D$ on a smooth complex analytic space $X$ is a divisor which is locally isomorphic to the union of some coordinate hyperplanes in $\Delta^n$.
If $D_i$ with $i\in I$ denote the components of $D$, then the strata of $D$ are the intersections $D_J\coloneq \bigcap_{j\in J}D_j$ for $J\subset I$; the respective open stratum is given by $D_J^\circ=D_J\setminus \bigcup_{i\in I\setminus J}D_i$.

\subsubsection{Linear algebra}
If $U$ is a free $\Z$-module and $\Lambda$ is a ring, 
then $U_\Lambda\coloneqq U\otimes_{\Z} \Lambda$ denotes the tensor product.

Let $\ell$ be a prime and let $x$ 
be an element of an abelian group.
An {\it $\ell$-prime multiple} of $x$ is a 
multiple $m\cdot x$ such that $m\in \Z$ is 
coprime to $\ell$.
For instance, $2$-prime multiples are 
nothing but odd multiples of $x$.  

Let $V$ be a free $\Z$-module of finite rank and let $l_i\in V^\ast=\Hom(V,\Z)$ be linear forms. 
Then we will often
denote the bilinear form 
$(x,y)\mapsto \sum a_i l_i(x)l_i(y)$ on $V$ 
by $Q=\sum a_i l_i^2$ and identify it with 
the attached quadratic form. 

\subsubsection{Graphs} \label{subsec:convention:graphs}
All graphs are finite.
An {\it orientation} of a graph $G$ 
is an orientation of its edges.
If not mentioned otherwise, then the 
chain complex $C_\ast(G,\Lambda)$ and the homology
$H_\ast(G,\Lambda)$  of a graph $G$ with coefficients 
in some ring $\Lambda$ refers to the simplicial 
chain complex, resp.\ simplicial homology.

Let $S$ be a set.
A {\it partial $S$-coloring} of a graph $G$ is a partition 
$\sqcup_{s\in S}E_s\subset E$ of a subset of the edges of $E$.
Edges in $E_s$ are called edges of \emph{color} $s$; edges in 
$E\setminus \sqcup_{s\in S}E_S$ are \emph{colorless} edges.
An {\it $S$-coloring} of a graph $G$ is a 
partial $S$-coloring without colorless edges, 
i.e.\ with $\sqcup_{s\in S}E_S=E$.
A \emph{morphism} of (partially) $S$-colored graphs is a map 
of graphs that respects colors, i.e.\ it maps 
edges of color $s$ to edges of color $s$ and 
it maps colorless edges to colorless edges.  
An $(I,S)$-\emph{bicolored} graph is a graph $G$ which 
carries colorings with respect to two sets 
$I$ and $S$. A \emph{morphism} of $(I,S)$-bicolored 
graphs is a morphism which respects both colorings.

Let $G$ be an $S$-colored graph. 
We say that a $1$-chain $\alpha\in C_1(G,\Lambda)$ 
has \emph{color} $s$ if it is a linear combination 
of (oriented) edges of color $s$.
Any $1$-chain $\alpha\in C_1(G,\Lambda)$ 
can uniquely be written as a sum 
$\alpha=\sum_{s\in S}\alpha_s$ of $1$-chains 
$\alpha_s$ of color $s$.
The $1$-chain $\alpha_s$ will be 
referred to as the {\it $s$-colored part} of $\alpha$. 

\subsubsection{Cell complexes} \label{subsubsec:cell-complexes}
We frequently identify a (finite) cell or 
polyhedral complex
of dimension $1$ with the associated graph and 
vice versa. 
A {\it refinement} of a graph $G$ is a 
refinement of the associated cell complex, i.e.\ it 
is the graph obtained from $G$ by replacing certain 
edges by chains of edges. 

Let $X\coloneq Y_p$ be a fiber of $D$-semistable or a 
$D$-quasi-nodal morphism $Y\to B$, 
see Definitions \ref{def:D-nodal} and 
\ref{def:D-quasi-nodal} below.
In the former case, $X$ analytically-locally
has, by the local form \eqref{d-snc-form} in Definition \ref{def:D-nodal}, 
product-of-snc
singularities $\prod_{i\in I'} 
\{x_i^{(1)}\cdots x_i^{(m_i)}=0\}$.  
In the latter case, $X$ has singularities which
are a product of nodal singularities 
$\{x_sy_s=0\}$ 
with a smooth variety.

Let $X^\nu\coloneq X^{[0]}$ be the normalization,
and more generally, $X^{[k]}$ be the normalization
of all codimension $k$ Whitney strata of $X$.
 If $X$ is a fiber of a strict
 $D$-semistable morphism (cf.\ Remark \ref{rem:strict}), we may write
 $X=\bigcup_{j\in J} X_j$ as a union 
 of smooth components, with smooth 
 intersections $X_{J'}\coloneq \bigcap_{j\in J'} X_j$
 for $J'\subset J$. We refer to the irreducible
 components of these 
 intersections as the {\it strata} of $X$
 and they form the irreducible components
 of $X^{[k]}$, though note that in general,
 $k\neq |J|$. 

The {\it dual complex} $\Gamma(X)$
is a polyhedral complex
whose $k$-dimensional polyhedral cells
are in bijection with the components of $X^{[k]}$,
and which are glued according to the incidences
of strata. For example, for the snc variety
$x^{(1)}\cdots x^{(m)}=0$, the dual complex
is an $(m-1)$-simplex, and for the product of snc
varieties $\prod_{i\in I'} 
\{x_i^{(1)}\cdots x_i^{(m_i)}=0\}$, the dual
complex is a product of $(m_i-1)$-simplices, 
with its natural
product polyhedral structure.
Thus, in the case of a $D$-quasi-nodal morphism,
the dual complex is a gluing of {\it cubes}, 
i.e.~products of $1$-simplices, along various faces.

We denote the {\it $k$-skeleton} $\Gamma^{k}(X)\subset \Gamma(X)$ of $\Gamma(X)$ as the union of all polyhedral faces of dimension $\leq k$.
In particular, 
$\Gamma^{1}(X)\subset \Gamma(X)$ is a graph,
encoding the incidences of the irreducible
components and double loci of $X$. 
We observe that the above constructions work
more generally, for any toroidal morphism $Y\to B$;
but, in general, the polyhedral cells are 
more complicated (and we will not need this).

\subsubsection{Matroids} \label{subsubsection:matroids}
A matroid $\underline R$ is a pair consisting of a finite set $S$, called the ground set of $\underline R$, and a subset $\mathcal I\subset \mathcal P(S)$ of its power set, whose elements are called the independent sets of $\underline R$, satisfying $\emptyset\in \mathcal I$, the downward-closed property, and the independent set exchange property, see \cite{oxley-matroids}.
If we want to highlight the ground set, we will also write $(\underline R,S)$ instead of $\underline R$.  
An {\it integral realization} is a map $S\to V$ 
to a free $\Z$-module $V$, which induces a realization of  $\uR$ over any field $k$, i.e.~over $V_k$ 
for $k$ arbitrary.
Equivalently, any subset of $S$ generates 
a saturated sublattice of $V$.  
In particular,
a matroid with an integral realization is 
{\it regular}---it
admits a realization over any field.
Conversely, every regular matroid admits an integral
realization by a {\it totally unimodular matrix}, 
i.e.~a matrix all of whose minors have 
determinant $\{0, \pm 1\}$,
see \cite[Theorem 6.6.3]{oxley-matroids}.
 
We always assume that 
the rank of $\uR$ agrees with the rank of $V$. In other words,
the image of $S$ generates $V$ as a $\Z$-module; 
this implies that the dual map $V^\ast \to \Z^S$ is injective.  
In this paper, the free $\Z$-module $V$ 
frequently occurs as the dual of another free $\Z$-module.

\subsection{Integral realizations and loopless matroids}

\begin{lemma} \label{lem:matrix=TU}
Let $(\underline R,S)$ be a regular matroid of rank $g$ with integral realization $S\to V$.
Let $B\subset S$ be a basis of $\underline R$.
Then the matrix $(\mathds 1_g|D)$ that represents the linear map $\Z^S\to V$ in the basis of $V$ defined by $B$ is totally unimodular.
\end{lemma}
\begin{proof}
Since $S\to V$ is an integral realization, any subset of $S$ generates a saturated sublattice of $V$.
It follows that all maximal minors are contained in $\{0,\pm 1\}$.
Let now $M$ be a non-maximal square 
submatrix of $(\mathds{1}_g|D)$.
We may assume that $M$ does not contain a zero column.
Then, we may extend $M$ to a maximal square
submatrix, by adding rows and columns 
corresponding to an identity submatrix
of the $\mathds{1}_g$ block.
This leaves the determinant of $M$ the same
up to sign, and so reduces to the earlier case.
\end{proof}

We discuss now uniqueness of integral realizations.

\begin{lemma} \label{lem:unique-integer-realization}
Let $(\underline R,S)$ be a regular matroid with  integral realizations $\varphi_1:S\to V_1$ and $\varphi_2:S\to V_2$.
Let $g=\rk V_i$ be the rank of $\underline R$.
Then there is a commutative diagram  
$$
\xymatrix{
\Z^S\ar[r] \ar[d]^\cong &V_1\ar[d]^{\cong}\\
\Z^S\ar[r] & V_2
}
$$
where the vertical maps are isomorphisms of $\Z$-modules, and the horizontal maps are the linear maps induced by $\varphi_1$ and $\varphi_2$, respectively.
More precisely, if we choose a basis $B\subset S$ of $\underline R$ and consider the associated matrix $(\mathds{1}_g|D_i) \in \Z^{g\times S }$ that represents $\Z^S\to V_i$ in the given basis, then $(\mathds{1}_g|D_2)$ can be obtained from $(\mathds{1}_g|D_1)$ by multiplying some rows and columns by $-1$.
\end{lemma}
\begin{proof} 
By Lemma \ref{lem:matrix=TU}, the choice of a basis $B\subset S$ of $\underline R$ allows us to represent the linear maps $\Z^S\to V_1$ and $\Z^S\to V_2$, induced by $\varphi_1$ and $\varphi_2$, by totally unimodular matrices $(\mathds{1}_g|D_1)$ and $(\mathds{1}_g|D_2)$ for $D_i\in \Z^{g\times (S\setminus B)}$, respectively. 
By \cite[Proposition 6.4.1]{oxley-matroids}, $D_1^{\sharp}=D_2^{\sharp}$, where $D_i^{\sharp}$ is the matrix whose entries are the absolute values of the respective entries of $D_i$.
By \cite[Lemma 13.1.6]{oxley-matroids}, $D_2$ can thus be obtained from $D_1$ by multiplying some rows and columns by $-1$.
This implies that we can obtain $(\mathds{1}_g|D_2)$ from $(\mathds{1}_g|D_1)$ by multiplying some rows and columns by $-1$.
\end{proof}

We say that $(\underline R,S)$ is
{\it loopless} if every 
singleton is an independent set.
If $(\underline R,S)$ is regular with integral 
realization $S\to V$, then the associated 
totally unimodular matrix $M$ has no 
zero column if and only if 
$(\underline R,S)$ is loopless. 

\begin{lemma} \label{lem:integral-realization}
If $S\to V$ is an integral realization of a loopless
regular matroid $(\underline R, S)$, then for each 
$s\in S$, the composition 
$V^\ast \to \Z^S\stackrel{\pr_s}\to \Z$ is surjective, 
where $\pr_s\colon \Z^S\to \Z$ denotes the projection 
to the $s$-th coordinate. 
\end{lemma}
\begin{proof}
By Lemma \ref{lem:matrix=TU}, $\Z^S\to V$ may be represented by a totally unimodular matrix $M\in \Z^{g\times S}$.
Since $(\underline R,S)$ is loopless, $M$ has no 
column which is identically zero.
Hence, the transpose $M^t$, which represents $V^\ast \to \Z^S$, has the property that each row contains an element $\pm 1$.
The lemma follows from this property.
\end{proof}

\subsection{$D$-nodal, nearly $D$-nodal, and $D$-semistable morphisms}
\label{sec:S-nodal}
We  use the following terminology, see also \cite[Definition 5.1]{survey}.

\begin{definition}\label{def:D-nodal} 
Let $(B,D)$ be a pair of a smooth complex analytic space $B$ 
and an snc divisor $D$ with components $D_i$, $i\in I$.
We say that a morphism $f\colon Y\to B$ of complex analytic spaces is 
\begin{enumerate}
\item  {\it $D$-nodal} if locally in the analytic topology, $f$  is 
of the form 
$$\textstyle 
\prod_{i\in I'} \{u_i=x_iy_i\}\times \Delta^{j+k}\to \prod_{i\in I'}\Delta_{u_i} \times \Delta^j,
$$ 
where $I'\subset I$, $u_i$ is a local equation for
$D_i$, and $\Delta^{j+k}\to \Delta^j$ 
is the projection to the first $j$ coordinates;
\item {\it nearly $D$-nodal} if we rather have a normal form
of shape 
$$\textstyle
\prod_{i\in I'} \{u_i=x_i^{(1)}y_i^{(1)}=\cdots 
= x_i^{(m_i)}y_i^{(m_i)}\}\times \Delta^{j+k}\to 
\prod_{i\in I'}\Delta_{u_i} \times \Delta^j;
$$
\item\label{d-snc-form}
{\it $D$-semistable} if we have
$$\textstyle
\prod_{i\in I'} \{u_i=x_i^{(1)}x_i^{(2)}\cdots x_i^{(m_i)} \}
\times  \Delta^{j+k}\to \prod_{i\in I'}\Delta_{u_i} 
\times \Delta^j.
$$ 
\end{enumerate} 
If moreover the generic fiber of $Y\times_B{D_i}\to D_i$ 
has smooth components for all $i\in I$, then $f$ is 
called \emph{strict} $D$-nodal, \emph{strict} 
nearly $D$-nodal, \emph{strict} $D$-semistable, 
respectively.
\end{definition}

\begin{remark} \label{rem:strict}
The strictness condition, together with the given 
local normal forms, ensure that the components of 
$Y\times_B D_i$ are  regular and have, in particular, 
no self-intersections.
\end{remark}

Definition \ref{def:D-nodal} works equally well in the algebraic category, by requiring the given local forms, in \'etale charts.
Next, we prove 
a Bertini-type theorem for 
hyperplane sections of strict $D$-semistable morphisms.
We work in the algebraic category, as it makes the existence of a certain Zariski open subset slightly more transparent.

\begin{lemma}\label{lem:Bertini} 
Let $B$ be a smooth quasi-projective variety over an infinite field 
with snc divisor $D\subset B$ and let $Y\to B$ be a projective 
strict $D$-semistable morphism of relative dimension $g\geq 1$.  
Let  $\mathcal E=L_1\oplus\dots \oplus L_{c}$ 
be a sum of very ample line bundles $L_i$ on $Y$,
$c\leq g$, and let $P\subset B$ be a finite set of points.
Let $Z\subset Y$ be the zero locus 
of a general section of $\mathcal E$.

Then there is a Zariski open subset 
$B^\circ\subset B$ which contains $P$, such that:
\begin{enumerate}
    \item \label{item:Bertini:1}
    The base change $Z_{B^\circ}\to B^\circ$ is strict 
    $D^\circ$-semistable, where $D^\circ\coloneqq B^\circ \cap D$.
\item \label{item:Bertini:3} 
The local normal form of $Z_{B^\circ}\to B^\circ$
at a point $z\in Z_{B^\circ}$ is the same
as that of $Y_{B^\circ}\to B^\circ$ at $z$, up to
reducing the value of $k$ in item \eqref{d-snc-form} of Definition \ref{def:D-nodal}. 
\item \label{item:Bertini:2}
For $p\in B^\circ$, the positive-dimensional strata of $Z_p$ are 
in bijective correspondence with the 
strata of dimension $\geq c+1$ of $Y_p$. 
\end{enumerate} 
\end{lemma}
\begin{proof}
Let $Y_p$ be the fiber of $Y\to B$ at a point $p\in P$.
As the ground field is infinite, 
Bertini's theorem implies that we may ensure that
\begin{enumerate}
    \item $Z$ is smooth of codimension $c$, and
    \item\label{item:lem:Bertini:transversality} 
    for each $p\in P$, $Z$ intersects
    all strata of $Y_p$ transversely. 
\end{enumerate}

In particular, $Z$ is disjoint from all
 strata of $Y_p$ of dimension 
$\leq c-1$, intersects each stratum
of dimension $c$ at a finite set
of points, and intersects each stratum
of dimension $k \geq c+1$ in a smooth irreducible
$(k-c)$-dimensional subvariety of that stratum.

Let $z\in Z_p$ and suppose that the \'etale
local form of the morphism
$Y\to B$ at $z$ is given by 
$W\times \bA^{j+k}\to \prod_{i\in I'} \bA_{u_i}\times \bA^j$, where $W\coloneq \prod_{i\in I'} 
\{u_i=x_i^{(1)}\cdots x_i^{(m_i)}\}$, cf.\ item \eqref{d-snc-form} in Definition \ref{def:D-nodal}. 
We may assume that $z$ corresponds to the origin in this chart.
That is, there is a Zariski open subset $U\subset Y$ with $z\in U$ and an \'etale map $U\to W\times \bA^{j+k}$ which sends $z$ to $(0,0)$ and which is compatible with the respective projections to $B$.

The transversality of the intersection, see \eqref{item:lem:Bertini:transversality} above, then implies $k\geq c$.
We may thus consider the composition
\begin{align} \label{eq:Bertini-composition}
Z\cap U\longhookrightarrow U\longrightarrow W\times \bA^{j}\times \bA^k\longrightarrow W\times \bA^{j} \times \bA^{k-c},
\end{align}
where the second arrow is induced by the identity on the first two factors and by the projection $\bA^k\to \bA^{k-c}$ to the first $k-c$ variables on the last factor. 
The above map is compatible with the respective projections to $B$.
The lowest-dimensional strata of $U_p$ corresponds in the above chart to $\{0\}\times \{0\}\times \bA^k$.
The fact that $Z$ intersects this stratum transversely thus implies that, up to a linear change of coordinates on $\bA^k$, we may assume that the composition in \eqref{eq:Bertini-composition} induces an isomorphism on tangent spaces at $z$.
Hence, up to shrinking $U$, we may assume that  \eqref{eq:Bertini-composition} is \'etale.
We have thus produced a Zariski open subset $U\subset Y$ with $z\in U$ such that items \eqref{item:Bertini:1}--\eqref{item:Bertini:2} hold on $U$.

Since $P$ is finite, we are able to find a finite number of \'etale charts as above that cover $Y_p$ for all $p\in P$.
From this one easily derives the existence of a Zariski open subset $B^\circ\subset B$ that has the properties in the lemma.
This concludes the proof.
\end{proof}

\begin{remark}
In this paper, we will apply the above lemma in the case $c=g-1$, so that
$Z_{B^\circ}\to B^\circ$ is a family of nodal curves, because the only local normal form in (\ref{d-snc-form}) giving a codimension one singular stratum of a fiber is $u_i=x_iy_i$.
\end{remark}

\subsection{Curves whose cohomology class is a multiple of the minimal class}
In the following lemma, we reduce the task of proving that an $\ell$-prime multiple of the minimal class is not algebraic to showing that no $\ell$-prime multiple is represented by an effective curve, which may in fact be assumed to be smooth and projective. 

\begin{lemma} \label{lem:alg-min->l-prime-multiple-is-effective}
Let $(X,\Theta)$ be a principally polarized abelian variety of dimension $g\geq 3$.
Let $\ell$ be a prime and assume that an $\ell$-prime multiple of $[\Theta]^{g-1}/(g-1)!$ is represented by a $\Z$-linear combination of classes of algebraic curves.
Then some (possibly different) $\ell$-prime multiple of $[\Theta]^{g-1}/(g-1)!$ is represented by a smooth projective connected curve.
\end{lemma}
\begin{proof}
By assumption, there is an integer $m$ that is coprime to $\ell$ such that 
$$
m[\Theta]^{g-1}/(g-1)!=\sum_i a_i[C_i]\in  H_2(X,\Z),
$$ 
for some curves $C_i\subset X$ and integers $a_i\in \Z$.
By a result of Hironaka \cite{hironaka-AJM}, we may assume that $C_i$ is smooth for all $i$.
Replacing $C_i$ by a generic translation, we may further assume that the union $C\coloneq \bigcup_i C_i$ is smooth (because $g\geq 3$).
A general complete intersection surface $Y\subset X$ of elements in $|k\Theta|$ for $k\gg 0$ which contain $C$ is then smooth by Bertini's theorem.
For a sufficiently large integer $b$, the line bundle
$$
\sum a_iC_i+\ell^b \Theta|_Y\in \Pic Y
$$
on $Y$ is globally generated by Serre's theorem, and hence represented by a smooth projective connected curve $C'\subset Y$.
When viewed as a curve on $X$, we then have
$$
[C']=\sum a_i[C_i]+\ell^b [\Theta]\cdot [Y]=m [\Theta]^{g-1}/(g-1)!+\ell ^b \cdot k^{g-2} [\Theta]^{g-1} \in H_2(X,\Z),
$$
which is an $\ell$-prime multiple of the minimal class, as we want.
\end{proof}

In the next lemma we work over an arbitrary field $k$ and consider the $\ell$-adic \'etale cohomology of varieties over the algebraic closure $\bar k$; this will later be applied to the generic fiber of a family of complex projective varieties.

\begin{lemma}\label{lem:minimal-class-algebraic}
    Let $X$ be an abelian variety of dimension $g$ over a field $k$, 
    principally polarized by 
    $\Theta \in \textnormal{NS}(X_{\bar k})$. 
    Let $\ell$ be a prime number different from the 
    characteristic of $k$. 
     Let $C$ be a smooth projective curve, 
    possibly disconnected, together with a 
    morphism $f \colon JC\to X$. Let $m$ be 
    a positive integer and consider the 
    following properties.  
    \begin{enumerate}
    \item \label{item:lem:minimal-class-algebraic:0} 
    We have $f_\ast[C]=m\cdot \Theta^{g-1}/(g-1)! 
    \in H^{2g-2}(X_{\bar k}, \Z_\ell(g-1))$. 
        \item \label{item:lem:minimal-class-algebraic:1} 
        The dual map $f^\vee \colon X\to JC$ has the property 
        that $(f^\vee)^\ast \Theta_C=m\cdot \Theta$.
        \item \label{item:lem:minimal-class-algebraic:2} 
        The composition $ 
X \xlongrightarrow{f^\vee} JC \xlongrightarrow{f}X
        $
        is multiplication by $m$.
        \item \label{item:lem:minimal-class-algebraic:3} 
        There is a vector bundle $\mathcal E$ on $X$ such that 
        $s_i(\ca E) = m^i \cdot \Theta^i/i!$ for each $i\geq 0$, 
        where $s_i(\ca E)$ denotes the $i$-th Segre class.
    \end{enumerate}
    Then \eqref{item:lem:minimal-class-algebraic:0} and 
    \eqref{item:lem:minimal-class-algebraic:1} are equivalent, 
    and imply \eqref{item:lem:minimal-class-algebraic:2} and 
    \eqref{item:lem:minimal-class-algebraic:3}. 
\end{lemma}
\begin{proof}
Let $\lambda_X \colon X \xrightarrow{\sim} X^\vee$ (resp.\ $\lambda_C \colon JC \xrightarrow{\sim} JC^\vee$) be the symmetric isomorphism attached to the principal polarization $\Theta$ (resp.\ $\Theta_C$).  
Let $f^\vee\colon X^\vee \to JC^\vee$ be the dual of $f$, and also denote by $f^\vee \colon X \to JC$ the composition $X \to X^\vee \to JC^\vee \to JC$ given by $\lambda_C^{-1} \circ f^\vee \circ \lambda_X$. 

The equivalence of \eqref{item:lem:minimal-class-algebraic:0} and \eqref{item:lem:minimal-class-algebraic:1} is well-known (sometimes referred to as Welters' criterion); it follows from the fact that the chain of isomorphisms
$
H^2(X_{\bar k},\Z_{\ell}(1))\cong H^{2g-2}(X^\vee_{\bar k},\Z_{\ell}(g-1))\cong H^{2g-2}(X_{\bar k},\Z_{\ell}(g-1))$ 
maps $\Theta$ to $\Theta^{g-1}/(g-1)!$. The first isomorphism is given by Poincar\'e duality and the second isomorphism is induced by $\lambda_X$.  

To prove \eqref{item:lem:minimal-class-algebraic:1}$\Rightarrow$\eqref{item:lem:minimal-class-algebraic:2}, 
recall that the pull back $(f^\vee)^\ast(\lambda_C)$ is 
by definition the symmetric isogeny $X \to X^\vee$ 
that makes the diagram 
\[
\xymatrix{
X \ar[d]_-{(f^\vee)^\ast(\lambda_C)} 
\ar[r]^-{f^\vee} & 
JC \ar[d]_-{\lambda_C} \\
X^\vee & JC^\vee \ar[l]_-f
}
\]
commute. 
As $(f^\vee)^\ast(\Theta_C) = m \cdot \Theta$ by \eqref{item:lem:minimal-class-algebraic:1}, we get $(f^\vee)^\ast(\lambda_C)
= m \cdot \lambda_X$. Hence the composition 
\[
X \xlongrightarrow{f^\vee} JC 
\xlongrightarrow{\lambda_C} JC^\vee 
\xlongrightarrow{f} X^\vee 
\xlongrightarrow{\lambda_X^{-1}} X
\]
is multiplication by $m$.
Hence, \eqref{item:lem:minimal-class-algebraic:1}$\Rightarrow$\eqref{item:lem:minimal-class-algebraic:2}, as claimed. 

It remains to prove 
\eqref{item:lem:minimal-class-algebraic:1}$\Rightarrow$\eqref{item:lem:minimal-class-algebraic:3}.
For an integer $n \geq 2g(C)-1$, where $g(C)$ denotes the genus of the curve $C$, let $P_n$ be a Poincar\'e 
bundle on $C \times J^nC$, where $J^nC$ is the $JC$-torsor 
of isomorphism classes of line bundles of degree $n$. 
Push $P_n$ forward to get a sheaf on $J^nC$, 
which is a vector bundle because $n \geq 2g(C)-1$
(which implies that the higher cohomology of 
$P_n$ on the fibers of $C \times J^nC \to J^nC$ vanishes). 
Pull this vector bundle back along the isomorphism 
$JC \cong J^nC$, defined by some degree $n$ divisor on $C$, 
to get a vector bundle $F$ on $JC$ that satisfies 
$\ch(F) = \rk(F) - \Theta$ and $c_i(F) = (-1)^i \Theta_C^i/i!$ 
for all $i\geq 0$, see \cite[p.~336]{arbarello-et-al-1985}.
 In other words, the total Chern class is given by 
 $c(F)=e^{-\Theta_C}$ and hence the total 
 Segre class is given by $s(F)=e^{\Theta_C}$.
In view of \eqref{item:lem:minimal-class-algebraic:1}, 
the pullback of $F$ along $f^\vee \colon X\to JC$ yields a 
vector bundle $\mathcal E \coloneqq (f^\vee)^\ast(F)$ on $X$ 
which satisfies \eqref{item:lem:minimal-class-algebraic:3} 
in the lemma. This concludes the proof. 
\end{proof}

\begin{lemma} \label{lem:orthogonal-complement}
Let $(X,\Theta)$ be a principally polarized abelian variety (over $\C$).
Let $f\colon JC\to X$ be a morphism from a Jacobian of a curve and let $f^\vee:X\to JC$ be the dual with respect to the given principal polarizations.
Assume that $f_\ast [C]=m\cdot [\Theta]^{g-1}/(g-1)!\in H_2(X,\Z)$ for some positive integer $m$.
Let $\Lambda$ be a ring in which $m$ is invertible.
Then,
$$
(f^\vee_\ast H_1(X,\Lambda))^{\perp_{\Theta_C}}=
\ker(f_\ast\colon H_1(JC,\Lambda)\to H_1(X,\Lambda)),
$$
where $\perp_{\Theta_C}$ denotes the orthogonal complement with respect to $\Theta_C$.
\end{lemma}
\begin{proof}
By Lemma \ref{lem:minimal-class-algebraic}, $\frac{1}{m}f_\ast^\vee$ splits the map $f_\ast\colon H_1(JC,\Lambda)\to H_1(X,\Lambda)$.
Hence,
\begin{align} \label{eq:orthogonal-complement:1}
H_1(JC,\Lambda)= f^\vee_\ast H_1(X,\Lambda)\oplus 
\ker(f_\ast\colon H_1(JC,\Lambda)\to H_1(X,\Lambda)).
\end{align}
We will show that this decomposition is orthogonal with respect to $\Theta_C$.
To this end, let $\beta\in H_1(X,\Z)$.  
The associated element of $H_1(X^\vee,\Z)=H_1(X,\Z)^\vee$ is given by $\Theta(\beta,-)$, and its image via $f^\vee_\ast$ is given by
$$
\Theta(\beta,f_\ast (-))\in 
H_1(JC^\vee,\Lambda)=H_1(JC,\Lambda)^\vee.
$$
It is clear that any class $\alpha\in \ker(f_\ast)$ lies in the kernel of the above linear form.
This proves the inclusion
$$
(f^\vee_\ast H_1(X,\Lambda))^{\perp_{\Theta_C}}\supset 
\ker(f_\ast\colon H_1(JC,\Lambda)\to H_1(X,\Lambda)).
$$
Equality follows from \eqref{eq:orthogonal-complement:1} together with the fact that the restriction of $\Theta_C$ to the subspace $f^\vee_\ast H_1(X,\Lambda)$ is non-degenerate and unimodular (because $m$ is invertible in $\Lambda$), see item \eqref{item:lem:minimal-class-algebraic:1} in Lemma \ref{lem:minimal-class-algebraic}.
\end{proof}
 
\begin{remark}
In the notation of Lemma \ref{lem:orthogonal-complement}, if $Y \subset JC$ denotes the identity component of the kernel of $f \colon JC \to X$ endowed with its induced polarization, then the natural morphism $X \times Y \to JC$ is an isogeny of polarized abelian varieties whose degree is invertible in $\Lambda$.
See \cite[Proposition 4.5]{beckmann-degaayfortman} and \cite{voisin2022cycle} for related results.
\end{remark}

\subsection{Monodromy of abelian varieties and algebraicity of the minimal class in families}
\label{subsec:degenerations-abelian-varieties}

We recall here some standard facts on the monodromy of families of
abelian varieties and deduce some consequences on algebraicity of the minimal class in families that will be crucial for our approach; for further references and details, 
see for instance \cite[Section 2.3]{survey}. 

\subsubsection{Monodromy bilinear forms} 
\label{subsec:mondoromy-bilinear-forms}
Let $S$ be a finite set and let 
$\pi\colon X^\star\to (\Delta^\star)^S$ be a smooth 
projective family of principally polarized abelian varieties.
We denote its fiber over $b\in (\Delta^\star)^S$ by $(X_b,\Theta)$.
Fix a base point $t\in (\Delta^\star)^S$ and assume that for all $s\in S$, the monodromy operator $T_s$ on $H_1(X_t,\Z)$ about the $s$-th coordinate hyperplane of $\Delta^S$ is unipotent with nilpotent logarithm $N_s\coloneq T_s-{\rm id}$.
We may then define a weight filtration
$$
0 \subset W_{-2}\subset W_{-1}\subset W_0=H_1(X_t,\Z) ,
$$
where $W_{-2}$ denotes the saturation of 
$\sum_{s\in S} {\rm im}\, N_s$ and 
$W_{-1}\coloneqq \bigcap_{s\in S} \ker N_s$.  
 
We will freely identify the theta divisor $\Theta$ on $X_t$ with the associated bilinear form on $H_1(X_t,\Z)$.
This form induces an isomorphism 
$$
W_{-2}H_1(X_t,\Z)\stackrel{\sim}\longrightarrow ({\rm gr}^W_0H_1(X_t,\Z))^\ast, \ \ \alpha \mapsto \Theta(\alpha,-) ,
$$
which allows us, as is well-known, to translate the nilpotent operator $N_s$ into a bilinear form, as follows, see e.g.\ \cite[Definition 2.6]{survey}:

\begin{definition} \label{def:monodromy-bilinear-form}
The $s$-th monodromy bilinear form $B_s$ associated to the family $\pi\colon X^\star\to (\Delta^\star)^S$ is the bilinear form on ${\rm gr}^W_0H_1(X_t,\Z)$, given by 
$$
x\otimes y\mapsto \Theta(N_s x, y) .
$$
\end{definition}

Let us now assume in addition that $\pi$ extends to a strict $H$-nodal morphism $X\to \Delta^S$, where $H\subset \Delta^S$ denotes the union of the coordinate hyperplanes. 
By a theorem of Clemens \cite[Theorem 7.36]{clemens}, the monodromy operator $T_s$ from above is then automatically unipotent and so the aforementioned assumption is satisfied. 
Moreover, the above weight filtration is the weight filtration of the associated limit mixed Hodge structure.
Let furthermore $\Gamma(X_0)$ be the dual complex of the central fiber $X_0$ of $\pi$. 
Then there is a canonical isomorphism
$$
{\rm gr}^W_0H_1(X_t,\Z)\cong H_1(\Gamma(X_0),\Z) .
$$ 
The result is well-known for rational coefficients, but holds in the case of $H_1$ also for integral coefficients, see e.g.\ \cite[Proposition 5.12]{survey}.
It follows that the monodromy bilinear forms $B_s$ from above may be viewed, in a canonical manner, as bilinear forms on $H_1(\Gamma(X_0),\Z)$, which we will denote with the same letter.

\subsubsection{Consequences of algebraicity of the minimal class in families} 
Let $C\to \Delta^S$ be a strict semistable degeneration of curves, smooth over  $(\Delta^\star)^S$.
As above, we get for each $s\in S$ 
a monodromy bilinear form $Q_s$ on
$$
{\rm gr}^W_0H_1(C_t,\Z)\cong H_1(\Gamma(C_0),\Z).
$$
From the hypothesis that $C$ is smooth,
each node of $C_0$ deforms to a node over exactly
one coordinate hyperplane of $\Delta^S$.
Thus, the
graph $G\coloneq \Gamma(C_0)$ is naturally $S$-colored, 
where we say that an edge has color $s$ if the 
corresponding node deforms to a node of the general 
fiber of $C\to \Delta^{S}$ over the $s$-th coordinate 
hyperplane. Since $C\to \Delta^S$ is semistable, 
the Picard--Lefschetz formula shows that the above 
monodromy bilinear form $Q_s$ is given by
$$
Q_s=\sum_{e\in E_s}x_e^2,
$$
where $E_s$ denotes the set of edges of color $s$, 
and where $x_e$ denotes the linear form on 
$H_1(\Gamma(C_0),\Z)$ which is, up to a 
sign, uniquely determined by the edge $e$.

Assume in the set-up of Section 
\ref{subsec:mondoromy-bilinear-forms} above, 
that we have a morphism $\iota \colon C\to X$ over $\Delta^S$, 
such that 
$$
\iota_\ast[C_t]=m\cdot [\Theta]^{g-1}/(g-1)! \in H_2(X_t,\Z).
$$
Note that $\iota$ induces a morphism of abelian schemes 
$f\colon JC|_{(\Delta^\star)^S}\to X|_{(\Delta^\star)^S}$.

\begin{lemma} \label{lem:orthogonal-Qs-preliminaries}
Let $\ell$ be a prime with $\ell \nmid m$. 
There is a canonical direct sum decomposition
\begin{align} \label{eq:orthogonal-complement:2}
{\rm gr}^W_0H_1(C_t ) \cong {\rm gr}^W_0 H_1(X_t )\oplus 
\ker(f_\ast \colon {\rm gr}^W_0H_1(C_t )\to {\rm gr}^W_0H_1(X_t)),
\end{align}
where homology is taken with $\Z_{(\ell)}$-coefficients.  
For each $s\in S$, the monodromy bilinear form $Q_s$ 
on ${\rm gr}^W_0H_1(C_t)$ satisfies:
\begin{enumerate}
    \item the restriction of $Q_s$ to 
    ${\rm gr}^W_0 H_1(X_t)$ agrees with  $m$ times the monodromy 
    bilinear form on ${\rm gr}^W_0 H_1(X_t)$ induced by 
    the family $X^\star\to (\Delta^\star)^S$;
    \item the decomposition \eqref{eq:orthogonal-complement:2} 
    is orthogonal with respect to $Q_s$.
\end{enumerate}
\end{lemma}
\begin{proof} 
By Lemma \ref{lem:orthogonal-complement}, 
there is a direct sum decomposition
\begin{align} \label{eq:orthogonal-complement:3}
H_1(JC_t,\Z_{(\ell)})\cong H_1(X_t,\Z_{(\ell)}) \oplus \ker(f_\ast\colon H_1(JC_t,\Z_{(\ell)}) \to H_1(X_t,\Z_{(\ell)}))
\end{align}
which is orthogonal with respect to $\Theta_{C_t}$.
Here we used that $(f^\vee)^\ast \Theta_{C_t}=m\Theta$ 
(see Lemma \ref{lem:minimal-class-algebraic}), and 
$m$ is invertible in $\Z_{(\ell)}$, so that 
$H_1(X_t,\Z_{(\ell)})\cong f^\vee_\ast H_1(X_t,\Z_{(\ell)})$.

The above decomposition respects the respective 
monodromy operators about the coordinate axis of 
$\Delta^S$ and hence in particular the respective 
weight filtrations. Taking ${\rm gr}^W_0$, we thus 
get the decomposition \eqref{eq:orthogonal-complement:2}.
The fact that this decomposition is $Q_s$-orthogonal 
follows from the fact that the monodromy operator 
$T_s$, and hence the nilpotent operator $N_s$, 
respects the decomposition \eqref{eq:orthogonal-complement:3} together with the fact that the latter is 
orthogonal with respect to the principal 
polarization $\Theta_{C_t}$ 
(see Lemma \ref{lem:orthogonal-complement}).
\end{proof}

\section{Admissible resolutions of $S$-colored morphisms} \label{sec:admissible-resolutions}

\subsection{$S$-colored $D$-quasi-nodal morphisms}

We aim to study base changes of strict $D$-nodal degenerations along morphisms $\tau\colon B'\to B$ such that the reduction of $\tau^{-1}(D)$ is snc. If $S$ is the indexing set of the components of $D$, then this base change will turn out to satisfy the following auxiliary definitions, see Lemma \ref{lem:base-change-of-D-nodal-degenerations} below.
 
\begin{definition} \label{def:D-quasi-nodal}
    Let $(B,D)$ be a pair of a smooth variety $B$ and 
    an snc divisor $D$ with components $D_i$, $i\in I$.
    A morphism $f\colon Y\to B$ is {\it $D$-quasi-nodal} 
    if there is a set $\mathcal C$ of analytic charts that 
    cover $Y$, such that for each $c\in \mathcal C$, 
    $Y\to B$ is given by the normal form
    \begin{align} \label{eq:normal-form-S-nodal}
\left\{\textstyle \prod_{i\in I} u_i^{a_{si}}=x_sy_s \mid 
 s\in S'\right\} \subset U\times \Delta^{2|S'|} ,
    \end{align}
    where $u_i$ are regular functions on some 
    analytic open $U\subset B$ with 
    $D_i\cap U=\{u_i=0\}$, $a_{si}\geq 0$ are
    non-negative integers, and $S'$ is a finite set. 
\end{definition}
 
The integers $a_{si}$ may all vanish for a given $s\in S'$; this corresponds to a smooth factor 
in \eqref{eq:normal-form-S-nodal}.
Note also that $U$, $a_{si}$, and $S'$ depend on the chart $c\in \mathcal C$; we will write $U_c$, $a_{si}(c)$, and $S'_c$ if we want to emphasize the dependence on $c$. 
Let us finally mention that the local normal form \eqref{eq:normal-form-S-nodal} is flat, as it is the 
pullback of a nodal morphism (where miracle flatness applies). 

\begin{definition}  \label{def:S-coloring} 
Let $S$ be a finite set.
An \emph{$S$-coloring} of a $D$-quasi-nodal morphism $f\colon Y\to B$ consists of a finite collection of effective Weil divisors $E_\alpha\subset Y$, indexed by $\alpha\in \Omega$, an $(S\times I)$-partition $ \Omega=\bigsqcup_{(s,i)\in S\times I} \Omega_{s,i}$, an atlas $\mathcal C$ as in Definition~\ref{def:D-quasi-nodal}, and an inclusion $S'_c \hookrightarrow S$ for every $c \in \mathcal C$, such that:
    \begin{enumerate} 
    \item   
    If $\alpha\in \Omega_{s,i}$, then $E_\alpha\cap c$ is empty, or cut out by $(u_i)$,   $(u_i,x_s)$, or $(u_i,y_s)$.
    Moreover, the divisors $E_\alpha\cap c$ with $\alpha\in \Omega_{s,i}$ have pairwise distinct components.\label{item:S-coloring:1}
    \item 
    If $s\in S'$ and $a_{si}\neq 0$, then the (local) divisors cut out by $(u_i,x_s)$ and $(u_i,y_s)$ lift to (global) Weil divisors $E_{\alpha}$ and $E_{\alpha'}$ with $\alpha,\alpha'\in \Omega_{s,i}$.
    \label{item:S-coloring:2} 
    \end{enumerate}
\end{definition}
\begin{remark} \label{rem:2-distinct-E_alpha-on-c}
In item \eqref{item:S-coloring:1}, if the divisor $E_\alpha\cap c$ is cut out by $(u_i,x_s)$ or $(u_i,y_s)$, then the condition that $s\in S'$ is implicit, as otherwise $x_s,y_s$ are not defined on the chart $c$.
Moreover, $a_{si}$ must be nonzero in this case, as otherwise $(u_i,x_s)$ and $(u_i,y_s)$ do not cut out a divisor.
\end{remark}
It is sometimes useful to have the following associated index sets at our disposal:
\begin{align} \label{eq:def:Omega_s-Omega_i}
\Omega_i\coloneq \bigsqcup_{s\in S } \Omega_{s,i}\quad \text{and}\quad \Omega_s\coloneq \bigsqcup_{i\in I} \Omega_{s,i}.
\end{align}
We then also have the partitions $\Omega=\bigsqcup_{i\in I}\Omega_i$ and $\Omega=\bigsqcup_{s\in S}\Omega_s$. 

By item \eqref{item:S-coloring:1}, $\alpha\in \Omega_{s,i}$ implies $E_\alpha\subset f^{-1}(D_i)$.
Since $f$ is flat, 
the  $(S\times I)$-partition of $\Omega$ can therefore be deduced from the $S$-partition $\Omega=\bigsqcup_{s\in S}\Omega_s$.

\begin{remark}
The conditions in Definition \ref{def:S-coloring} are stable under localization. 
In particular, if they are satisfied for one atlas $\mathcal C$ as in Definition~\ref{def:D-quasi-nodal} then in fact for any.
\end{remark}

\begin{lemma} \label{lem:components-of-E_alpha-smooth} 
Let $\alpha\in \Omega_{s,i}$ and let $c \in \ca C$.
Then the (local) components of $E_\alpha\cap c$ extend to (global) components of $E_\alpha$.  
\end{lemma}

\begin{proof}
Since $\alpha\in \Omega_{s,i}$, $E_\alpha\subset f^{-1}(D_i)$. 
Since $Y\to B$ is flat,
we thus do not loose any component of  $E_{\alpha}\cap c$ when we localize at a general point of $u_i=0$.
We may therefore assume that $u_j\neq 0$ on $c$ for all $j\neq i$.
We may further assume that $E_\alpha\cap c\neq \emptyset$. By item \eqref{item:S-coloring:1}, $E_\alpha\cap c$ is cut out by $(u_i)$, $(u_i,x_s)$, or $(u_i,y_s)$.

We first treat the case where it is cut out by $(u_i,x_s)$. 
Then $E_{\alpha}\cap c$ is given by (a smooth base change of)
\begin{align} \label{eq:rem:components-of-E_alpha-smooth}
\{u_i=x_s=0,\ x_ty_t=0\mid t\in S'\setminus \{s\},\ a_{ti}\neq 0\} 
\end{align}
and the components of this divisor can be read off from this description. 
Since $(u_i,x_t)$ and $(u_i,y_t)$ correspond for all $t\in S'$ with $a_{ti}\neq 0$ to global divisors as well 
(see item \eqref{item:S-coloring:2}), we find that the components of \eqref{eq:rem:components-of-E_alpha-smooth} can be written as intersections of some $E_\beta \cap c$. In particular, the components of $E_{\alpha}\cap c$ extend to components of $E_\alpha$, as claimed. 
A similar argument applies if  $E_\alpha\cap c$ is cut out by $(u_i,y_s)$.

It remains to deal with the case where  $E_\alpha\cap c$ is cut out by $(u_i)$.
Then $E_{\alpha}\cap c$ is given by (a smooth base change of)
\begin{align} \label{eq:rem:components-of-E_alpha-smooth}
\{u_i=0,\ x_ty_t=0\mid t\in S',\ a_{ti}\neq 0\} 
\end{align}
and we argue similarly as above.
This concludes the proof of the lemma.
\end{proof} 
 
\begin{remark} \label{rem:strict}
The above argument shows that for $\alpha \in \Omega_{s,i}$ the components of the generic fiber of  $E_\alpha\to D_i$ are regular.
In particular, $D$-quasi-nodal morphisms are automatically \emph{strict} $D$-quasi-nodal, where strictness is defined analogously to Definition \ref{def:D-nodal}. 
\end{remark}

\begin{remark} \label{rem:components-c_p-globalize}   
Let $p\in B$. 
In a local chart $c$ as in \eqref{eq:normal-form-S-nodal}, the fiber $c_p$ is given by the product of a smooth variety with
\begin{align} \label{eq:c_p}
\left\{x_sy_s=0\mid s\in S'' \right\}\subset \Delta^{2|S''|} ,
\end{align}
where $S''=\{s\in S'\mid \prod_i u_i^{a_{si}}(p)=0\}.$ 
It follows that the components of $c_p$ are, by item \eqref{item:S-coloring:2} in Definition  \ref{def:S-coloring}, restrictions of intersections $\bigcap_{s\in S''}E_{\alpha_s}$ for suitable $\alpha_s\in \Omega_s$.
In particular, any component of $c_p$ globalizes to a component of $Y_p$, cf.\ Lemma \ref{lem:components-of-E_alpha-smooth}.  
Taking suitable intersections, we see that, more generally, any intersection of components of $c_p$ extends to a component of some closed codimension $d$ stratum $Y_p^{[d]}$ of $Y_p$.  
Since the components of $c_p$ as well as their intersections are smooth, we find that the components of the codimension $d$ stratum $Y_p^{[d]}$ of $Y_p$ are smooth for all $d$.
In particular, the components of $Y_p^{[d]}$ have no self-intersections and the natural morphism of polyhedral complexes $\Gamma(c_p) \to \Gamma(Y_p)$ is injective. 
\end{remark} 

The conditions in Definition \ref{def:S-coloring} yield canonical $S$-colorings on the $1$-skeleta $\Gamma^1(Y_p)$ of the fibers of $f$ as follows; this explains the phrase ``$S$-coloring'' in Definition \ref{def:S-coloring}.

\begin{definition}\label{def:def-degeneration->S-coloring} 
    In the notation of Definition \ref{def:S-coloring}, let $p\in B$.
    For $s \in S$, an edge $e$ of $\Gamma^1(Y_p)$ has \emph{color $s$} if, for some $i\in I$, there are distinct indices $\alpha_1,\alpha_2\in \Omega_{s,i}$  such that $e$ is given by the intersection of a component of $Y_p\cap E_{\alpha_1}$ with a component of $Y_p\cap E_{\alpha_2}$.  
\end{definition} 

Note that for each $p\in B$, the natural inclusion of graphs $\Gamma^1(c_p)\to \Gamma^1(Y_p)$ does not contract any edge.  
We then say that an edge of $\Gamma^1(c_p)$ has color $s$ if its image in $\Gamma^1(Y_p)$ has that property.

\begin{lemma}\label{lem:def-degeneration->S-coloring} 
Definition \ref{def:def-degeneration->S-coloring} yields a canonical $S$-coloring of $\Gamma^1(Y_p)$ for all $p\in B$.  
In a local chart $c$ as in \eqref{eq:normal-form-S-nodal} with a fiber $c_p$ as in \eqref{eq:c_p}, an edge $e$ of $\Gamma^1(c_p)$ has color $s$ if and only if $s\in S''$ (via the given  
embedding $S''\subset S'\hookrightarrow S$ from Definition \ref{def:S-coloring}) and $e$ corresponds to the intersection of a component of $x_s=0$ with a component of $y_s=0$.
\end{lemma}

\begin{proof}
In order to show that Definition \ref{def:def-degeneration->S-coloring} yields an $S$-coloring, we have to show that each edge has exactly one color.
This will follow once the local description claimed in the lemma is proven.
To this end, let $e$ be an edge of $\Gamma^1(c_p)$ of color $s$, corresponding to the intersection of two components of $c_p$ that globalize to components of $Y_p \cap E_{\alpha_1}$ and $Y_p \cap E_{\alpha_2}$ with distinct $\alpha_1,\alpha_2\in \Omega_{s,i}$.  
Since $\alpha_1\neq \alpha_2$, condition \eqref{item:S-coloring:1} in Definition \ref{def:S-coloring} implies that $E_{\alpha_1}\cap c$ and  $E_{\alpha_2}\cap c$ have no components in common and that they are cut out by the ideals $(u_i,x_s)$ and $(u_i,y_s)$. 
In particular, $s \in S''$ (see Remark \ref{rem:2-distinct-E_alpha-on-c}) and $e$ corresponds to the intersection of a component of $x_s=0$ with a component of $y_s=0$. 
Conversely, let $e$ be an edge of $\Gamma^1(c_p)$ corresponding to the intersection of a component of $x_s=0$ with a component of $y_s=0$ with $s\in S''$.
Since $s\in S''$, there is some index $i\in I$ with  $a_{si}\neq 0$ and $u_i(p)=0$.
To prove that $e$ has color $s$, we then use that the local Weil divisors cut out by $(u_i,x_s)$ and $(u_i,y_s)$ correspond to some $E_\alpha$ and $E_\beta$ with $\alpha,\beta\in \Omega_{s,i}$ by item \eqref{item:S-coloring:2} in Definition \ref{def:S-coloring}. This concludes the proof. 
\end{proof}

\subsection{Dual complexes of the fibers}
Let $Y\to B$ be $S$-colored $D$-quasi-nodal.
Consider a chart $c$ as in \eqref{eq:normal-form-S-nodal}. 
Let $p\in B$ and consider the chart $c_p$ of the fiber $Y_p$ above $p$.
This chart is empty if $c$ does not meet $Y_p$.
Otherwise, $c_p$ is the product of a smooth variety with  $  \left\{  0=x_sy_s \mid \ s\in S''\right\}  \subset \Delta^{2|S''|}$, where $S''=\{s\in S'\mid \prod_{i\in I} u_i^{a_{si}}(p)=0\}$, see \eqref{eq:c_p}. 

\begin{lemma} \label{lem:S-nodal-Gamma(Y_p)-charts}
Let $Y\to B$ be an $S$-colored $D$-quasi-nodal morphism. For each $c\in \mathcal C$, let $\Gamma(c_p)$ be the dual complex of the chart $c_p$ given by \eqref{eq:c_p}.
Then $\Gamma(c_p)$ is a cube of dimension $|S''|$ and the polyhedral cell structure of $\Gamma(Y_p)$ is induced by a quotient map
\begin{align} \label{eq:chart-Gamma(Y_p)}
\textstyle\bigsqcup_{c\in \mathcal C} 
\Gamma(c_p) \xtwoheadrightarrow{}\Gamma(Y_p),
\end{align}
which is injective on each cube $\Gamma(c_p)$ and identifies two 
cubes $ \Gamma(c_p)$ and $ \Gamma(c'_p)$ along $d$-dimensional
faces $F$ and $F'$ via an isomorphism of polyhedral complexes $F\stackrel{\sim} \to F'$ if and only if $F$ and $F'$ correspond to the same component of the codimension $d$ stratum $Y_p^{[d]}$ of $Y_p$.  
\end{lemma}
\begin{proof}
Note first that any cell of the dual complex $\Gamma(Y_p)$ can be detected in some  local chart \eqref{eq:c_p}. 
Hence, there is a natural surjection as in \eqref{eq:chart-Gamma(Y_p)}.
The restriction of this map to each cube $\Gamma(c_p)$ is injective because of Remark \ref{rem:components-c_p-globalize}.
The description of the gluing follows from the definition of the dual complex. 
\end{proof}

\begin{remark}
A priori the topological space $\Gamma(Y_p)$ may 
not be homeomorphic to a manifold. 
This behavior will not arise in our applications, 
where the charts \eqref{eq:chart-Gamma(Y_p)} 
will define a cubical tiling of a real torus,  
cf.\ \cite[Figure 15]{survey}. 
\end{remark}

\subsection{Specialization maps} \label{subsec:specialization}

\begin{definition} \label{def:specialization}
Let $D\subset B$ be an snc divisor with components $D_i$, $i\in I$. Let $p\in D^\circ _J$ and $q\in D^\circ_{J'}$ be points of open strata of $D$ such that $J \subset J'$. We say that \emph{$p$ specializes to $q$} if we are given a path $\gamma \colon [0,1]\to D_{J}$ with $\gamma(0)=p$ and $\gamma(1)=q$, such that $\gamma([0,1))\subset D_J^{\circ}=D_J\setminus \bigcup_{i\in I\setminus J}D_i$ is contained in the open stratum $D_J^\circ\subset D_J$.  
\end{definition}

Let $D\subset B$ be an snc divisor and let $f\colon Y\to B$ be a $D$-quasi-nodal or  $D$-semistable morphism. 
By the Thom--Mather theorem \cite[Proposition 11.1 and Corollary 10.3]{mather}, $f|_{D_J^\circ } \colon Y|_{D_J^\circ } \to D_J^\circ $ admits locally on the target a topological trivialization which is compatible with the stratification of $Y|_{D_J^\circ }$ given by the singularity type of the fibers. As a consequence, components, (components of) double intersections as well as lower dimension strata of the fibres of $f|_{D_J^\circ }$ form local systems over $D_J^\circ $. 
Let $p\in D_J^\circ $ and $q\in D_{J'}^\circ $ be points such that $p$ specializes to $q$ via a continuous path $\gamma$, see Definition \ref{def:specialization}.
Pick a sufficiently small open ball $U\subset D_J$ centered at $q$ and let $U^\circ=U\cap D_J^\circ$.
The local normal forms imply that, after possibly shrinking $U$, the local systems formed by the strata of the fibers of $Y_{U^\circ}\to U^\circ$ are in fact trivial. 
We may thus extend a stratum of a fiber over a point of $U^\circ$ to a family of strata over $U^\circ$ whose closure gives rise to a union of strata of $Y_q$ of the same dimension. Since the path $\gamma$ lands eventually in $U$, this allows us to specialize components and double loci of $Y_p$ to unions of such in $Y_q$. 
In this situation, we may define a continuous specialization map on graphs
$$
{\rm sp}\colon \Gamma^1(Y_q)\longrightarrow \Gamma^1(Y_p)
$$
as follows:
\begin{enumerate}
    \item if $v$ is a component of $Y_q$, then ${\rm sp}(v)$ is the unique component of $Y_p$ whose specialization along $\gamma$ to $Y_q$ contains the component $v$.  
    \item if $e$ is an edge of $\Gamma^1(Y_q)$, corresponding to a double locus where components $v$ and $v'$ meet, then consider the components ${\rm sp}(v)$ and ${\rm sp}(v')$ of $Y_p$.
    If ${\rm sp}(v)={\rm sp}(v')$, then $e$ is contracted to the vertex ${\rm sp}(v)$. 
    Otherwise, ${\rm sp}(v)$ and ${\rm sp}(v')$ meet along a double locus and we let ${\rm sp}(e)$ be the corresponding edge of $Y_q$.
\end{enumerate}

\subsection{Admissible modifications}
Recall from Section \ref{subsubsec:cell-complexes} that
a refinement $\hat G$ of a graph 
$G$ is a refinement of the corresponding cell complex.
If $G$ is $S$-colored, then 
$\hat G$ carries a canonical $S$-coloring.

\begin{definition}\label{def:admissible}
    Let $B$ be a smooth variety with an snc divisor $D$ 
    with strata $D_J$, $J\subset I$.
    Let $Y\to B$ be an $S$-colored $D$-quasi-nodal morphism.
    We define an {\it admissible modification} $(g\colon Y'\to Y, \varphi)$ to be a pair, consisting of
    \begin{enumerate}
        \item[(a)] a projective modification 
        $g\colon Y'\to Y$ 
        for which $Y'\to B$ is $D$-quasi-nodal or $D$-semistable, and $g$ induces an isomorphism over the complement of $D$, and 
        \item[(b)] for each $p\in B$, 
        a refinement $\hat \Gamma^1(Y'_p)$ of the $1$-skeleton $\Gamma^1(Y'_p)$, together
        with a morphism of $S$-colored graphs
    $$
    \varphi_p\colon \hat \Gamma^1(Y'_p)\longrightarrow \Gamma^1(Y_p) ,
    $$
    \end{enumerate}
    such that the following conditions hold:  
    \begin{enumerate}  
            \item  
    \label{item:def-admissible:1}
            The dual complex  $\Gamma(Y'_p)$ is given by a refinement of the cell structure of $\Gamma(Y_p)$ and the induced  homeomorphism of topological spaces $\Gamma(Y_p)\stackrel{\approx}\to \Gamma(Y'_p)$ has the property that the following diagram is commutative up to homotopy:
        $$
        \xymatrix{
        \Gamma^1(Y'_p) \approx \hat\Gamma^1(Y'_{p}) \ar@{^{(}->}[d] \ar[r]^-{\varphi_p} & \Gamma^1(Y_{p})\ar@{^{(}->}[d]\\ 
        \Gamma(Y'_p)  & \ar[l]_-{\approx}\Gamma(Y_{p}),
        }
        $$
        where the vertical maps are the canonical inclusions, and ``$\approx$'' indicates the canonical homeomorphism induced by the given refinement.
        \item \label{item:def-admissible:2} 
        If $p\in D_J^\circ $ and $q\in D_{J'}^\circ $ are points of open strata of $D$ with $J\subset J'$, such that $p$ specializes to $q$ along a path $\gamma$ (see Definition \ref{def:specialization}), then the following diagram of topological spaces commutes: 
        $$
        \xymatrix{
\hat\Gamma^1(Y'_{p}) \approx \Gamma^1(Y'_p) \ar[d]^{\varphi_p} & \ar[l]_-{\operatorname{sp}}  \Gamma^1(Y'_q)\approx \hat\Gamma^1(Y'_q) \ar[d]^{\varphi_q} \\
\Gamma^1(Y_p)  & \ar[l]_-{\operatorname{sp}}  \Gamma^1(Y_q) ,
        }
        $$
        where  the horizontal maps are induced by the specialization maps from Section \ref{subsec:specialization}. 
    \end{enumerate} 
\end{definition}

An admissible modification $(g\colon Y'\to Y, \varphi)$ such that $Y'$ is $D$-semistable (see Definition \ref{def:D-nodal}) will also be called an {\it admissible resolution}.

\begin{remark}
Note that any $D$-quasi-nodal morphism is locally on the source a base change of a $D$-nodal morphism; by miracle flatness and stability of flatness under fiber products, $D$-quasi-nodal and $D$-semistable morphisms are therefore automatically flat 
and hence so is the morphism $Y' \to B$ in the above definition.  
\end{remark}

\begin{remark}
There is a natural map $\Gamma^1(Y_p)\hookrightarrow \hat \Gamma^1(Y'_p)$ of topological spaces, because the entire dual complex $\Gamma(Y'_p)$ is a refinement of $\Gamma(Y_p)$ and $ \hat \Gamma^1(Y'_p)$ is a refinement 
of the $1$-skeleton $\Gamma^1(Y'_p)$.
In this paper, the collection $\varphi = (\varphi_p)_{p\in B}$ will always be constructed in such a way that $\varphi_p$ is, up to homotopy, a retraction of 
this map, i.e.\ the composition $\Gamma^1(Y_p)\hookrightarrow \hat \Gamma^1(Y'_p)\stackrel{\varphi_p}\to \Gamma^1(Y_p)$ is homotopic to the identity, cf.\ item \eqref{item:def-admissible:1} in Definition \ref{def:admissible} and Figures \ref{fig:retract-cube} and \ref{fig:diagonal-subdivide} below. 
\end{remark}

The main result of this section is then the following theorem.

\begin{theorem} \label{thm:admissible-modification}
    Let $D\subset B$ be an snc divisor with components $D_i$, $i\in I$.
    Let $Y\to B$ be an $S$-colored $D$-quasi-nodal morphism. Fix a total ordering on the set $\Omega$ from Definition \ref{def:S-coloring}.  
    Then there is a canonical admissible modification $(g\colon Y'\to Y,\varphi)$ such that $Y'\to B$ is strict $D$-semistable; in particular, $Y'$ is regular, and flat over $B$.  
\end{theorem}

\begin{remark} Our resolution algorithm should be compared to the more general result in \cite[Theorem 2.7]{alt}.
For our applications though, the collection $\varphi=(\varphi_p)_{p\in B}$ of maps of $1$-skeleta is also crucial data.
\end{remark}

The remainder of this section is devoted to a proof of the above theorem.
We will first blow-up certain Weil divisors $E_{\alpha}$ until the $D$-quasi-nodal morphism turns into a nearly $D$-nodal morphism, see Proposition \ref{prop:blow-up-of-S-nodal-degenerations} below.
We will then choose certain common components of the Weil divisors $E_{\alpha}$ whose blow-ups in a carefully chosen order will produce an admissible resolution, see Proposition \ref{prop:small-blow-up-of-S-nodal-degenerations} below.
Roughly speaking, the strategy of proof is to construct admissible resolutions in local charts and to exploit the global Weil divisors $E_\alpha$ from Definition \ref{def:S-coloring} to show that our construction glues to give what we want globally.

Up to a slightly different ordering of the blow-ups, the same resolution (without the discussion of admissibility) is constructed in toric language in the survey \cite[\S 5.3]{survey}. 
(Strictly speaking, \cite[\S 5.3]{survey} only considers the case of a monomial base change of a strict $\bar D$-nodal morphism as in Lemma \ref{lem:base-change-of-D-nodal-degenerations} below, which is however sufficient for the applications in this paper.)  

\subsubsection{Reduction of the $\alpha$-order}\label{subsec:reduction-alpha-weight}

\begin{definition} \label{def:alpha-weight}
Let $Y\to B$ be an $S$-colored $D$-quasi-nodal morphism.  
For $\alpha\in \Omega_{s,i}$ and $p\in B$, we define the $\alpha$-order $o_{\alpha,p}(c)$ of a chart $c\in \mathcal C$ at $p\in B$ as follows.  
Assume that $c$ has normal form \eqref{eq:normal-form-S-nodal}.
If $E_\alpha \cap c_p=\emptyset$, then $o_{\alpha,p}(c)=0$. If $E_\alpha\cap c=\{u_i=0\}$ and $u_i(p)=0$, then $o_{\alpha,p}(c)=1$. 
Otherwise, by item \eqref{item:S-coloring:1} in Definition \ref{def:S-coloring}, $E_\alpha\cap c$ is cut out by $(u_i,x_s)$ or $(u_i,y_s)$ and we have $s\in S'$ and $a_{si}\geq 1$ (cf.~Remark \ref{rem:2-distinct-E_alpha-on-c}); in that case, we define $o_{\alpha,p}(c)$ as the vanishing order of the function $\prod_{j}u_j^{a_{sj}}$ at $p$:
$$
o_{\alpha,p}(c) \coloneqq 
\ord_p\left(\textstyle \prod_{j}u_j^{a_{sj}}\right) .
$$  
The $\alpha$-order of $Y$ is then defined as
$$ 
o_{\alpha}(Y)\coloneqq 
\max_{c\in \mathcal C,\, p\in B}o_{\alpha,p}(c) .
$$ 
\end{definition}

\begin{remark} \label{rem:alpha-weight} 
It follows that $E_\alpha$ is empty in a neighborhood of $c_p$ if and only if $o_{\alpha,p}(c)=0$, and it is non-empty and Cartier in such a neighborhood if and only if $o_{\alpha,p}(c)=1$. 
\end{remark}

Note that $Y\to B$ is nearly $D$-nodal (see Definition \ref{def:D-nodal}) if and only if $o_{\alpha}(Y) \leq  1$ for all $\alpha\in \Omega$.  

\begin{proposition} 
\label{prop:blow-up-of-S-nodal-degenerations}
Let $f \colon Y\to B$ be an $S$-colored $D$-quasi-nodal morphism.  
Assume there is some $\alpha\in \Omega$ with $o_{\alpha}(Y)\geq 2$.
Then the blow-up $g \colon Y'\to Y$ of the Weil divisor $E_\alpha$ is, for the trivial refinements $\hat \Gamma^1(Y'_p) = \Gamma^1(Y'_p)$ and a canonically defined collection of maps $\varphi_p\colon \hat \Gamma^1(Y'_p)\to \Gamma^1(Y_p)$, $p\in B$, an admissible modification.
Moreover, $Y'\to B$ is $S$-colored $D$-quasi-nodal, with indexing set $\Omega' =\Omega\sqcup \{\beta\}$ for some $\beta\in \Omega'$, such that 
$$
o_{\alpha}(Y')=o_{\beta}(Y')= o_{\alpha}(Y)-1 
$$ 
and $o_{\gamma}(Y')=o_{\gamma}(Y)$ for $\gamma\in \Omega\setminus \{\alpha\}$.
\end{proposition}

\begin{proof} 
In view of the $(S\times I)$-partition of $\Omega$ from Definition \ref{def:S-coloring}, we have $\alpha\in \Omega_{s,i}$ for some $(s,i) \in S\times I$. 
We proceed in several steps.

\begin{step} \label{step:blow-up-E_alpha-big:1}
Blow-up of local charts.
\end{step}

Let $c$ be a chart with normal form
\begin{align} \label{eq:S-nodal-blow-up-chart} 
\left\{\textstyle \prod_{j\in I} u_j^{a_{tj}}=x_ty_t 
\mid \ t\in S'\right\} \subset U\times \Delta^{2|S'|} .
\end{align}
We may assume that locally on $c$, $E_\alpha$ is non-empty and not Cartier on $c$. 
By Definition \ref{def:S-coloring}, we may regard $S'\subset S$ as a subset.
Moreover, since $\alpha\in \Omega_{s,i}$ and $E_\alpha\cap c$ is not Cartier, it is cut out by $(u_i,x_s)$ or by $(u_i,y_s)$.  
To describe the local equations of the blow-up, up to replacing $x_s$ by $y_s$, we can without loss of generality assume that it is cut out by $(u_i,x_s)$.
The blow-up $c'\to c$ of the ideal $(u_i,x_s)$ is covered by two charts, denoted by $(c,x_s)$ and $(c,u_i)$ respectively, as follows. The $(c,x_s)$-chart is explicitly given by
\begin{align} \label{eq:S-nodal-blow-up-x_s-chart} 
\left\{\textstyle u_i=x_su'_i,\ \ 
\prod_{j\in I}u_j^{a_{tj}}=
x_ty_t\mid t\in S' \setminus \{s\} \right\} ,
\end{align} 
while the $(c,u_i)$-chart is given by
\begin{align} \label{eq:S-nodal-blow-up-u_i-chart} 
\left\{\textstyle
u_i^{a_{si}-1}\prod_{j\in I\setminus 
\{i\}}u_j^{a_{sj}}=x'_sy_s,\ \  
\prod_{j \in I}u_j^{a_{tj}}=
x_ty_t\mid t\in S' \setminus \set{s} \right\}.
\end{align}   
Altogether this shows that $Y'\to B$ is a $D$-quasi-nodal morphism. 

\begin{step} \label{step:blow-up-E_alpha-big:1.5}
$Y'\to B$ is  canonically an $S$-colored $D$-quasi-nodal morphism. 
\end{step}

We define $\Omega'_{s,i}\coloneqq \Omega_{s,i}\sqcup\{\beta\}$ 
for a new element $\beta$, and $\Omega'_{t,j}\coloneqq \Omega_{t,j}$ for $(t,j)\in (S \times I) \setminus \{(s,i)\}$. 
We further let $\Omega'\coloneq \bigsqcup_{(t,j)\in S\times I} \Omega'_{t,j}$.
For $\gamma\in \Omega'$, we then define the divisors $E'_\gamma$ on $Y'$ as follows.
We let 
$$
E'_\gamma \coloneq \begin{cases}g^\ast E_\alpha - g^{-1}_\ast E_\alpha\quad &\text{if $\gamma=\beta$};\\
    g^{-1}_\ast E_\gamma\quad &\text{if $\gamma\in \Omega'_{s,i}\setminus\{\beta\}$};\\
    g^\ast E_\gamma \quad &\text{if $\gamma\notin \Omega'_{s,i}$}.
\end{cases} 
$$
Here, $g^\ast E_\gamma$ and $g^{-1}_\ast E_\gamma$ refer to the total and proper transform of $E_\gamma$, respectively.
In particular, $E'_\beta$ is the reduced divisor contracted by $g\colon Y'\to Y$; it is empty if and only if $g$ is a small map.
We have thus defined a collection of Weil divisors on $Y'$, indexed by the sets $\Omega'_{t,j}$, $(t,j) \in S \times I$. 
It remains to check that this satisfies item \eqref{item:S-coloring:1} and item \eqref{item:S-coloring:2} in Definition \ref{def:S-coloring}.  
 
Let $c$ be a chart with normal form \eqref{eq:S-nodal-blow-up-chart}. We may without loss of generality assume that $E_\alpha$ is not Cartier on $c$, hence $o_{\alpha,p}(c)\geq 2$ for some $p \in B$.
In particular, $s \in S'$ and $a_{si} \neq 0$, hence by item \eqref{item:S-coloring:2} in Definition \ref{def:S-coloring} there exists an index $\alpha'\in \Omega_{s,i}$ such that $E_{\alpha'}$ is cut out by $(u_i,y_s)$ in the chart \eqref{eq:S-nodal-blow-up-chart}. 
Note that $\alpha$ and $\alpha'$ are the only  two indices $\gamma \in \Omega_{s,i}$ such that $E_\gamma \cap c \neq \emptyset$, see item \eqref{item:S-coloring:1} in Definition \ref{def:S-coloring}.

First assume $a_{si} \geq 2$. 
Then in the $(c,x_s)$-chart \eqref{eq:S-nodal-blow-up-x_s-chart}, we have $y_s=u_i' u_i^{a_{si}-1} \prod_{j \neq i} u_j^{a_{sj}}$ and $u_i=x_su'_i$, and we deduce
\begin{align} \label{eq:E'_alpha-x_s-chart}
E'_\alpha=\emptyset, \quad \quad 
E'_{\alpha'}=\{u_i=u'_i=0\}
\quad \quad \text{and}\quad \quad E'_\beta=\{u_i=x_s=0\}.
\end{align}
In the $(c,u_i)$-chart \eqref{eq:S-nodal-blow-up-u_i-chart},  we have $x_s = u_ix_s'$ and so
\begin{align} \label{eq:E'_alpha-u_i-chart}
E'_\alpha=\{u_i=x'_s=0\}, \quad \quad E'_{\alpha'} = \emptyset
\quad \quad \text{and}\quad \quad E'_\beta=\{u_i=y_s=0\} .
\end{align}  
Next, assume $a_{si} = 1$. Then in the $(c,x_s)$-chart \eqref{eq:S-nodal-blow-up-x_s-chart}, we have $y_s=u_i' \prod_{j \neq i} u_j^{a_{sj}}$.
Hence
\begin{align} \label{eq:E'_alpha-x_s-chart:a_si=1}
E'_\alpha= \{u_i = x_s = 0\}, \quad \quad E'_{\alpha'}= \set{u_i = u_i' = 0} \quad \quad \text{and}\quad \quad E'_\beta= \emptyset.
\end{align}
In the $(c,u_i)$-chart \eqref{eq:S-nodal-blow-up-u_i-chart},  we have $x_s = u_ix_s'$ and $\prod_{j\neq i}u_j^{a_{sj}}=x'_sy_s$ and so
\begin{align} \label{eq:E'_alpha-u_i-chart:a_si=1}
E'_\alpha= \{u_i = 0\}, \quad \quad E'_{\alpha'} = \emptyset 
\quad \quad \text{and}\quad \quad E'_\beta=\emptyset .
\end{align} 
Moreover, for any $\gamma \in \Omega'$ with $\gamma \notin \Omega'_{s,i}$, we know by assumptions that $E_\gamma$ is either trivial on $c$ or it agrees with the vanishing locus of some ideal of the form $(u_j)$, $(u_j,x_t)$, resp.\ $(u_j,y_t)$ with $(j,t)\neq (s,i)$.
The same description holds for the total transform $E'_{\gamma}=g^\ast E_\gamma$ of $E_\gamma$.

The above descriptions show that the divisors $E'_{\gamma}$ with $\gamma\in \Omega'$,  
together with the atlas of $Y'$ induced by the given atlas of $Y$ via \eqref{eq:S-nodal-blow-up-x_s-chart} and \eqref{eq:S-nodal-blow-up-u_i-chart}, satisfy the conditions of Definition \ref{def:S-coloring}.

\begin{step} \label{step:blow-up-E_alpha-big:3}
    Cell structure of $\Gamma(Y'_p)$.
\end{step}

Let $p\in B$.  
If $p\notin D_i$, then $Y'_p\to Y_p$ is an isomorphism and hence $\Gamma(Y'_p)=\Gamma(Y_p)$ as cell complexes.
Let now $p\in D_i$.
We have a presentation 
$$
\bigsqcup_{c}\Gamma(c_p)\xtwoheadrightarrow{} \Gamma(Y_p)
$$
of $\Gamma(Y_p)$ by the cuboids $\Gamma(c_p)$, where $c$ runs through all charts that meet $Y_p$, see Lemma \ref{lem:S-nodal-Gamma(Y_p)-charts}.
If $c\in \mathcal C$ is such a chart, then we consider the blow-up $c'\to c$ induced by $Y'\to Y$. 
Since $c'\to c$ is an isomorphism in a neighborhood of $c_p$ if $o_{\alpha,p}(c)\leq 1$, to describe the cell structure of $\Gamma(Y'_p)$ we can assume that $o_{\alpha,p}(c)\geq 2$. 

The chart $c$ has normal form \eqref{eq:S-nodal-blow-up-chart}, and the charts \eqref{eq:S-nodal-blow-up-x_s-chart} and \eqref{eq:S-nodal-blow-up-u_i-chart} cover $c'$.
We thus obtain a surjection 
\begin{align} \label{eq:chart-Gamma(c'_p)}
\Gamma((c,x_s)_p)\sqcup \Gamma((c,u_i)_p) \xtwoheadrightarrow{} \Gamma(c'_p)
\end{align}
of cell complexes. 
Composing the disjoint union of the maps \eqref{eq:chart-Gamma(c'_p)} for $c \in \mathcal C$ with the natural surjection $\sqcup_c \Gamma(c'_p) \twoheadrightarrow \Gamma(Y'_p)$ yields the cell structure of $\Gamma(Y'_p)$.

The map \eqref{eq:chart-Gamma(c'_p)} identifies $\Gamma(c'_p)$ with the cell complex given by gluing the cube $\Gamma((c,x_s)_p)$ to $\Gamma((c,u_i)_p)$ along their common codimension one face $F$, and can be described as follows. The chart $c_p$ is isomorphic to a product of a smooth variety with
$$
\{x_ty_t=0\mid t\in S''\}\subset \Delta^{2|S''|}
$$
for some subset $S''\subset S'$.
(Here $s\in S''$ because $o_{\alpha,p}(c)\geq 2$ 
and $\alpha\in \Omega_s$.) 
By Remark \ref{rem:components-c_p-globalize}, each component of $c_p$ extends to a component of $Y_p$.
The proper transform of this component yields components of $Y'_p$ and $c'_p$, respectively.
The vertices of $\Gamma^0(c'_p)$ that are not contained in $\Gamma^0(c_p)$ are precisely those of the faces of $\Gamma((c,x_s)_p)$ and $\Gamma((c,u_i)_p)$ that are glued in \eqref{eq:chart-Gamma(c'_p)}.
They can intrinsically be defined as those components of $c'_p$ that are contracted via $c'_p\to c_p$.

In terms of the coloring of the $1$-skeleton of $\Gamma(c_p)$, the refinement in \eqref{eq:chart-Gamma(c'_p)} corresponds to adding a vertex at the middle of each edge of color $s$ and to subdivide the cuboid $\Gamma(c_p)\approx \Gamma(c'_p)$ along the wall that is generated by those vertices. 
See Figure \ref{fig:retract-cube}.

\begin{figure}[h]
    \centering
    \includegraphics[width=0.5\linewidth]{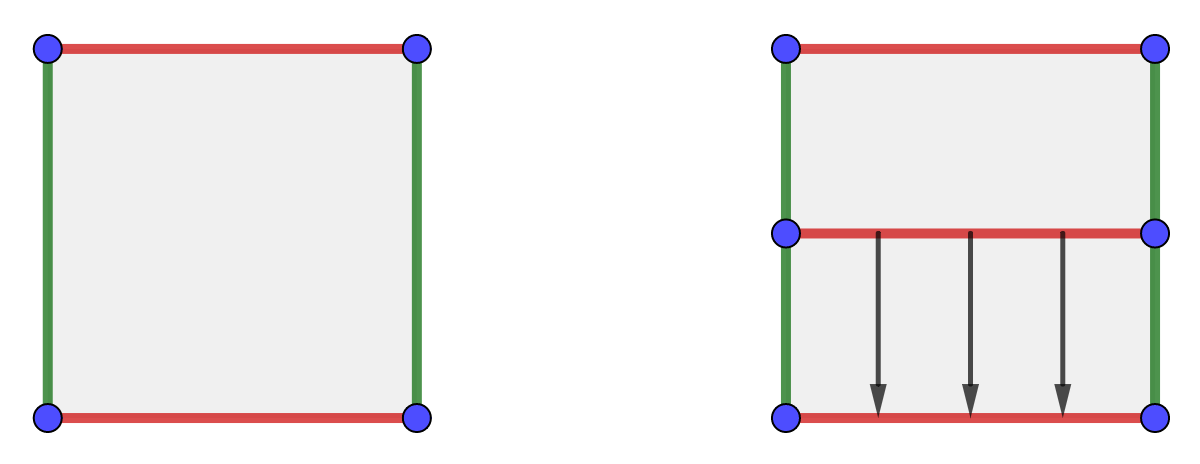}
    \caption{Left: Cube $\Gamma(c_p)$ of the dual complex $\Gamma(Y_p)$.
    Right: Cuboid refinement $\Gamma(c_p')$
    of this cube, corresponding to the fiber of the blowup $Y_p'\to Y_p$.
    Map $\varphi_p\colon \Gamma(c_p')\to \Gamma(c_p)$ is depicted by
    arrows.}
    \label{fig:retract-cube}
\end{figure}

\begin{step} \label{step:blow-up-E_alpha-big:4}
$g:Y'\to Y$ is admissible for a canonically defined
collection $\varphi=\{\varphi_p\}_{p\in B}$
of maps  $\varphi_p\colon
\Gamma^1(Y_p')\to \Gamma^1(Y_p)$. 
\end{step}

We have already seen that there is a canonical way in which we can, for any $p\in B$, identify $\Gamma(Y'_p)$ with a refinement of the cell structure of $\Gamma(Y_p)$.
This gives a homeomorphism $\Gamma(Y_p)\stackrel{\approx}\to \Gamma(Y'_p)$, which restricts to a continuous map on $1$-skeleta $\Gamma^1(Y_p)\to \Gamma^1(Y'_p)$.
We define $\hat \Gamma^1(Y'_p)\coloneqq \Gamma^1(Y'_p)$ with the $S$-coloring induced from the one of $\Gamma^1(Y'_p)$ from Step \ref{step:blow-up-E_alpha-big:1.5} and Definition \ref{def:def-degeneration->S-coloring}. We aim to construct a map of $S$-colored graphs $\varphi_p\colon \Gamma^1(Y'_p)\to \Gamma^1(Y_p)$ which is homotopic to a retraction of the continuous map  $\Gamma^1(Y_p)\to \Gamma^1(Y'_p) \subset \Gamma(Y'_p)$. 

We start with the local problem in the chart $c$ from Step \ref{step:blow-up-E_alpha-big:1}. We may without loss of generality assume that $o_{\alpha,p}(c)\geq 2$. We aim to define a canonical map
$$
\varphi_p\colon \Gamma^1(c'_p)\longrightarrow \Gamma^1(c_p)
$$
of colored graphs with the desired properties.
In the present case, it turns out that one can define $\varphi_p$ easily on the entire dual complexes to obtain maps $\Gamma(c'_p)\to \Gamma(c_p)$.
A priori, there are two obvious choices: Via the refinement \eqref{eq:chart-Gamma(c'_p)}, we could contract either  $\Gamma((c,x_s)_p)$ or $\Gamma((c,u_i)_p)$, by contracting its $s$-colored edges and identifying the result linearly with $\Gamma(c_p)$.  
The fact that the divisor $E_\alpha$ that is blown-up is part of the data in the proposition allows us to pick one of the choices in a canonical way, as follows. The components of $E_\alpha\cap c_p$ determine a codimension one face of $\Gamma(c_p)$; the opposite face is given by $E_{\alpha'}\cap c_p$ and the edges that join the two faces are precisely the edges of color $s$.
Since $\Gamma(c'_p)$ is a refinement of $\Gamma(c_p)$ given by introducing a wall, spanned by the midpoints of all edges of color $s$, we can thus define $\varphi_p$ by contracting the chamber of $\Gamma(c'_p)$ that contains those vertices in $\Gamma^0(c_p)\subset \Gamma^0(c'_p)$ which correspond to the components of $E_\alpha\cap c_p$. 
In other words, we contract $\Gamma((c,u_i)_p)$ and not $\Gamma((c,x_s)_p)$, and this choice is well-defined globally.
We thus get a canonical map of cell complexes
$$
\psi_p\colon \Gamma(c'_p)\longrightarrow \Gamma(c_p) .
$$ 
See Figure \ref{fig:retract-cube}.
The above map of cell complexes induces a canonical map on $1$-skeleta: $\varphi_p\colon \Gamma^1(c'_p)\to \Gamma^1(c_p)$.
Our construction is canonical and in particular compatible with localization.
Hence, it glues to give a unique map of $S$-colored graphs
$$
\varphi_p\colon \Gamma^1(Y'_p)\longrightarrow \Gamma^1(Y_p) ,
$$
that we denote by the same letter.

Each cuboid $\Gamma(c_p)$ is embedded into $\Gamma(Y_p)$ (see Lemma \ref{lem:S-nodal-Gamma(Y_p)-charts}), and $\varphi_p$ maps $\Gamma^1(c'_p)$ to $\Gamma^1(c_p)$.
Moreover, there exists a canonical homotopy between the continuous self-map of the cuboid $\Gamma(c_p)$ induced by $\psi_p$, and the identity on $\Gamma(c_p)$. These homotopies glue to define the homotopy required in the diagram in item \eqref{item:def-admissible:1} of Definition \ref{def:admissible}.

Let $p\in D_J^\circ $ and $q\in D_{J'}^\circ $ be points as in item \eqref{item:def-admissible:2} of Definition \ref{def:admissible}, so that $p$ specializes to $q$ via a continuous path $\gamma$.
To check the compatibility under specialization stated in item \eqref{item:def-admissible:2} of Definition \ref{def:admissible}, cover $Y_q$ by finitely many charts $c$. Up to replacing $p$ by a point sufficiently nearby $q$, we may assume that the path $\gamma$ lies in the image of the map $c \to U_c \subset B$ for each of these charts $c$.
It then suffices to prove that the following diagram commutes: 
\begin{align}\label{align:specialization-local-diagram}
\begin{split}
\xymatrix{
\Gamma^1(c'_{p}) \ar[d]^{\varphi_p} & \ar[l]_-{\operatorname{sp}} \Gamma^1(c'_q) \ar[d]^{\varphi_q} \\
\Gamma^1(c_p)  & \ar[l]_-{\operatorname{sp}}  \Gamma^1(c_q).
}
\end{split}
\end{align}  
Let $S'(p)\coloneq\{t\in S'\mid \prod_{j\in I} u_j^{a_{tj}}(p)=0\}$ and  $S'(q)\coloneq\{t\in S'\mid \prod_{j\in I} u_j^{a_{tj}}(q)=0\}$.
Since $p$ specializes to $q$, we have $S'(p)\subset S'(q)$.
The dual complex $\Gamma(c_q)$ (resp.\ $\Gamma(c_p)$) is a cuboid of dimension $|S'(q)|$ (resp.\ $|S'(p)|$), given by the product of intervals of color $ t\in S'(q)$ (resp.\ $ t\in S'(p)$). 
The dual complexes $\Gamma(c'_{p})$ and $\Gamma(c'_{q})$ are obtained from $\Gamma(c_{p})$ and $\Gamma(c_{q})$ by inserting a vertex in the middle of each edge of color $s$ and connecting these vertices to a wall that is parallel to the two codimension one faces that do not contain any edge of color $s$.  
The specialization map $\operatorname{sp}$ can be identified to the projection map that contracts all edges of color $t\in S'(q)\setminus S'(p)$.
Moreover, $\varphi_p$ and $\varphi_q$ contract for each chain of consecutive edges of color $s$ exactly that edge of color $s$ that is visible in the chart $(c,u_i)$; i.e.~they contract exactly one of the two chambers that have been introduced by the aforementioned wall.
From this description, the commutativity of the diagram is easily seen. 

Altogether, we have proven the compatibility stated in item \eqref{item:def-admissible:2} of Definition \ref{def:admissible}.
This concludes the proof of Proposition \ref{prop:blow-up-of-S-nodal-degenerations}.
\end{proof}

\subsubsection{Small resolution of $S$-colored nearly $D$-nodal degenerations}

\begin{proposition} \label{prop:small-blow-up-of-S-nodal-degenerations}
Let $Y\to B$ be an $S$-colored $D$-quasi-nodal morphism.
Assume $o_{\alpha}(Y)\leq 1$ for all $\alpha\in \Omega$, i.e.\ $Y\to B$ is nearly $D$-nodal.  
Then a total order on $\Omega$ induces in a canonical manner an admissible modification $(g\colon Y'\to Y,\varphi)$, 
such that $Y'\to B$ is strict $D$-semistable.
\end{proposition}
\begin{proof}  
We proceed in several steps.

\begin{stepp} \label{step:small-blow-up:1}
Local resolution in a chart $c$.
\end{stepp}

Let $c\in \mathcal C$ be a chart with normal form \eqref{eq:normal-form-S-nodal}.
By assumption, $o_{\alpha}(Y)\leq 1$ for all $\alpha\in \Omega$. 
Hence, $\sum_{i\in I}a_{si}$ is either $0$ or $1$ and $c$ is (a smooth base change of):
\begin{align} \label{eq:chart-nearly-nodal-small-resolution}
\left\{ u_i=x_sy_s  \mid i\in I',\ s\in S_i \right\} \subset U\times \Delta^{2|S''|} ,
\end{align}
where $I'\subset I$ is the subset of $i\in I$ with $a_{si}=1$ for some $s$, $S_i=\{s\in S\mid a_{si}=1\}$
and $S''\subset S'$ is the subset of $s\in S'$ with $a_{si}=1$ for some $i \in I'$.
Note that the subsets $S_i$ are pairwise disjoint, hence $S''=\cup_{i\in I'}S_i$ is a partition.
Let $p\in U$. 
To describe the local resolution of the chart $c$, up to shrinking $U$, we can assume that $u_i$ with $i\in I'$ can be extended to analytic coordinates on $U$ and that $u_i(p)=0$ for all $i\in I'$.
Then $c$ is locally analytically isomorphic to the  product of a polydisc with the product  
$$
\prod_{i\in I'} \left\{ u_i=x_sy_s\mid s\in S_i \right\} \subset \prod_{i\in I'}\Delta_{u_i}\times \Delta^{2|S_i|} . 
$$ 
In other words, $Y\to B$ is nearly $D$-nodal, see Definition \ref{def:D-nodal}.

The above product is resolved by resolving each factor independently.
For $i\in I$, we denote by $c_i$ the chart
\begin{align} \label{eq:chart-c_i}
\left\{ u_i=x_sy_s\mid s\in S_i \right\}
\end{align}
of the $i$-th factor.
Note that this is a family over a curve whose central fiber is a union of $2^{|S_i|}$-many components; the dual complex $\Gamma(c_{i,p})$ of the central fiber is isomorphic to the cube $[0,1]^{|S_i|}$ of dimension $|S_i|$, given as the product of edges of color $s\in S_i$ which correspond to the node $x_s=y_s=0$.
The singularities of $c_i$ admit a small resolution $c_i'\to c_i$, given by blowing-up the non-Cartier components of the central fiber of $c_i$ repeatedly (unique up to the choice of an order of the components to be blown up).
A local computation shows that the result is semistable over $\Delta_{u_i}$, i.e.\ it will be given by local equations of the form $u_i=z_1z_2\cdots z_m$ for some $m$, cf.\ \cite[\S 5.3]{survey}.
Moreover, the small resolution $c'_i\to c_i$ induces  a bijection 
on the set of components of each fiber. 
See Figure \ref{fig:diagonal-subdivide}. 

\begin{stepp} \label{step:small-blow-up::1.5}
For each $p\in B$, there is a canonical partial ordering on the vertices $\Gamma^0(Y_p)$ which for each chart $c\in \mathcal C$ is a total order on $\Gamma^0(c_p)$.
\end{stepp} 

Up to shrinking $I$ and $S$ by removing those $i \in I$ and $s \in S$ for which $\Omega_i =\sqcup_{s\in S}\Omega_{s,i}$ and $\Omega_s= \sqcup_{i\in I}\Omega_{s,i}$ are empty, we can assume that $\Omega_i$ and $\Omega_s$ are non-empty for all $i\in I$ and all $s\in S$. 
The total order on $\Omega=\bigsqcup_{(s,i)\in S\times I} \Omega_{s,i}$ then defines a total ordering on $S$, by declaring $s<t$ if $\Omega_{s}$ contains an element that is smaller than any element of $\Omega_{t}$.
Similarly, we get a total ordering on $I$ by saying that $i<j$ if $\Omega_i$ contains an element that is smaller than any element of $\Omega_j$.

The total ordering on $S$ induces a total ordering on its power set, with the property that $A<B$ whenever $A,B\subset S$ are subsets with $|A|<|B|$; subsets of the same cardinality are ordered via the lexicographic ordering.
Moreover, the total ordering on $\Omega$ induces a total ordering on each of its subsets and, via the lexicographic order, on  products  of these sets over totally ordered index sets. Using this,  we get for each $i\in I$ a total ordering on the following set:
\begin{align*}
    \mathfrak S_i&\coloneq \coprod_{A \in \ca P(S)} \prod_{s \in A} \Omega_{s,i} \\
    &=\left \{ (\alpha_s)_{s\in A}\mid A\in \mathcal P(S),\ \ \alpha_s\in \Omega_{s,i}\ \text{for all $s\in A$} \right \}  .
\end{align*}
Note that this ordering has the property that $(\alpha_s)_{s\in A}< (\beta_s)_{s\in A'}$ whenever $|A|<|A'|$. 
For $(\alpha_s)_{s\in A}\in \mathfrak S_i$, we define the divisor 
\begin{align} \label{eq:E_alpha_s-common-component}
E_{(\alpha_s)_{s\in A}}\quad \text{as the common components of $E_{\alpha_s}$ with $s\in A$. }
\end{align}  
We then define
$$
\mathfrak S\coloneq \prod_{i\in I}  (\mathfrak S_i\cup \{\ast\} )
$$
for a formal element $\ast$ that is smaller then any element of $\mathfrak S_i$.
Since $I$ is totally ordered, the lexicographic order yields a total order on $\mathfrak S$.
For $(\alpha_{i})_{i\in I}\in \mathfrak S$, we define
\begin{align} \label{eq:E_alpha_i-intersection}
Z_{(\alpha_i)_{i\in I}}\coloneq \bigcap_{i\in I, \ \alpha_i\neq \ast} E_{\alpha_i}
\end{align}   
as the intersection of all divisors $E_{\alpha_i}$ with $i\in I$ (as defined in \eqref{eq:E_alpha_s-common-component}) such that $\alpha_i\neq \ast$. Moreover, $Z_{(\alpha_i)_{i\in I}}=Y$ if $\alpha_i=\ast$ for all $i\in I$. 

Let now $p\in D$ and consider the fiber $Y_p$.
We define 
$$
\nu\colon \Gamma^0(Y_p)\longrightarrow \mathfrak S
$$
by mapping a vertex $v\in \Gamma^0(Y_p)$ with corresponding component $V\subset Y_p$ to the element
 $$
 \nu(v)\coloneq\max\{(\alpha_i)_{i\in I}\in \mathfrak S \mid V\subset Z_{(\alpha_i)_{i\in I}}\} .
 $$
This is well-defined, because $\mathfrak S$ is totally ordered and any component is contained in  $Z_{(\ast)_{i\in I}}=Y$, where $(\ast)_{i\in I}\in \mathfrak S$ is the minimal element. The map $\nu$ defines a partial order on $\Gamma^0(Y_p)$ by declaring that $v,w\in \Gamma^0(Y_p)$ with $v\neq w$ are incomparable if $\nu(v)=\nu (w)$ and otherwise we have $v>w$ if and only if $\nu(v)>\nu(w)$.
Since the total ordering on $\mathfrak S$ depends only on the total ordering on $\Omega$, the same holds true for the induced partial ordering on $\Gamma^0(Y_p)$.

We aim to describe the partial order on vertices in the local chart $c$, given by a smooth base change of \eqref{eq:chart-nearly-nodal-small-resolution}. 
Let $i\in I'$, i.e.\ $c\to U_c\subset B$ is not smooth over general points of $u_i=0$. 
Then $u_i=0$ cuts out the components of $c\times_BD_i$. 
Each of these components can,  by item \eqref{item:S-coloring:2} in Definition \ref{def:S-coloring},  
be described as intersections of a collection of divisors
$$
\bigcap_{s\in S_i}E_{\alpha_s}\quad \quad \text{for some}\quad \quad  (\alpha_s)_{s\in S_i}\in \prod_{s\in S_i}\Omega_{s,i},
$$
cf.~Lemma \ref{lem:components-of-E_alpha-smooth}. The set $(\alpha_s)_{s\in S_i}$ is unique by the disjointness condition in item \eqref{item:S-coloring:1} of Definition \ref{def:S-coloring}.
(Note that the above intersection is not necessarily maximal, because there might be colors $s\in S\setminus S_i$ and indices 
$\alpha\in \Omega_{s,i}$ with $E_\alpha\cap c=\{u_i=0\}$; in fact, if such an index $\alpha$ exists for some $s\in S\setminus S_i$ then it is unique by the disjointness condition in item \eqref{item:S-coloring:1}.
To relate the above description with the one used in the definition of $\nu$, we thus have to add for each such color $s\in S\setminus S_i$ the element $\alpha\in \Omega_{s,i}$ with $E_{\alpha}\cap c=\{u_i=0\}$.) 
The total ordering on $\mathfrak S_i$ thus induces for all $p\in B$ a total ordering on the vertices $\Gamma^0(c_{i,p})$, where $c_i$ is the chart in \eqref{eq:chart-c_i}.
Without loss of generality, we can assume $u_i(p)=0$ for all $i\in I'$.
Then we have $\Gamma^0(c_p)=\prod_{i\in I'}\Gamma^0(c_{i,p})$, because $c=\prod_{i\in I'}c_i$ (up to the product with a smooth morphism). Therefore, the total ordering of $I$, together with the lexicographic order on tuples, defines a total ordering on $\Gamma^0(c_p)=\prod_{i\in I'}\Gamma^0(c_{i,p})$.  
This total ordering agrees with the restriction of the above defined partial order on $\Gamma^0(Y_p)$ to the vertices $\Gamma^0(c_p) \subset \Gamma^0(Y_p)$ of the cuboid $\Gamma(c_p)$.

\begin{stepp} \label{step:small-blow-up:2}
Global resolution $g:Y'\to Y$.
\end{stepp}

For $(\alpha_s)_{s\in A}\in \mathfrak S_i$, we consider the divisor $E_{(\alpha_s)_{s\in A}}$ from \eqref{eq:E_alpha_s-common-component}. We construct $g$ by repeatedly blowing-up divisors of this form. We determine the order as follows.
We use the total ordering of $I$ induced by the one of $\Omega$ (see Step \ref{step:small-blow-up::1.5}) and go from the largest to the smallest element.
For given $i\in I$, we consider the totally ordered set $\mathfrak S_i$ from Step \ref{step:small-blow-up::1.5} and go again from the largest to the smallest element.
For each $(\alpha_s)_{s\in A}\in \mathfrak S_i$ in the given ordering, we then blow up the divisor $E_{(\alpha_s)_{s\in A}}$.

In the local chart \eqref{eq:chart-nearly-nodal-small-resolution}, the blow-up of $E_{(\alpha_s)_{s\in A}}$ only affects the $i$-th factor of  $c=\prod_{j\in I'}c_j$.
To analyse how this blow-up affects the chart $c_i$, let $T_i\subset S\setminus S_i$ be the subset of colors $t\in S$ such that there is some $E_{\alpha}$ with $\alpha\in \Omega_{t,i}$ which restricts to $\{u_i=0\}$ on $c_i$.
If $|A|>|S_i|+ |T_i|$ or $|A|=|S_i|+ |T_i|$ but $A\neq S_i\sqcup T_i$, then $E_{(\alpha_s)_{s\in A}}$ is empty in $c$ and the chart is not affected.
If $A=S_i\cup T_i$ and $(\alpha_s)_{s\in S_i\cup T_i}$ runs through all respective tuples, then the chart $c_i$ will be resolved as explained in Step \ref{step:small-blow-up:1}. 
Afterwards, $c'_i$ is regular and further blow-ups of Weil divisors as in \eqref{eq:E_alpha_s-common-component} with $(\alpha_s)_{s\in A}\in \mathfrak S_i$ will not alter the chart anymore. 

Altogether we have thus seen that $g$ is a small resolution which restricts to the resolution described in Step \ref{step:small-blow-up:1} in charts.
In particular, for all $p\in B$, $Y'_p\to Y_p$ induces 
a bijection on irreducible components.
Moreover, the local analysis in charts shows that $Y'\to B$ is $D$-semistable and that the natural map $Y'\times_BD_i\to Y\times _BD_i$ 
also induces a bijection on the set of irreducible components.
Since $Y\to B$ is $S$-colored $D$-quasi-nodal,  
the components of the generic fiber of $Y\times _BD_i \to D_i$ 
are smooth, see Remark \ref{rem:strict}, and so the local analysis shows that the same holds for the components of the generic fiber of $Y'\times_BD_i \to D_i$.  
It follows that $Y'\to B$ is strict $D$-semistable.

\begin{stepp} \label{step:small-blow-up:3}
First properties needed in Definition \ref{def:admissible}. 
\end{stepp}

Note first that $Y'\to B$ is (strict) $D$-semistable by Step \ref{step:small-blow-up:2} and $Y'\to Y$ is an isomorphism over $B\setminus D$ because we have blown up  divisors of the form \eqref{eq:E_alpha_s-common-component} 
and these divisors are supported on $Y\times_BD$ by item \eqref{item:S-coloring:1} in Definition \ref{def:S-coloring}.

Next, we aim to describe $\Gamma^1(Y'_p)$ as a refinement of $\Gamma^1(Y_p)$. For $p\in B$, consider the presentation
$$
\bigsqcup_{c\in \mathcal C} \Gamma(c_p) \xtwoheadrightarrow{} \Gamma(Y_p) 
$$
of the cell structure of $\Gamma(Y_p)$, given by \eqref{eq:chart-Gamma(Y_p)}.
Then the dual complex $\Gamma(Y'_p)$ admits a presentation
\begin{align} \label{eq:chart-Gamma(Y'_p)-small}
\bigsqcup_{c\in \mathcal C} \Gamma(c'_p) \xtwoheadrightarrow{} \Gamma(Y'_p) ,
\end{align}
where $c'_p$ is the fiber at $p$ of the small resolution $c'$ of the chart $c$, induced by $g:Y'\to Y$.
We have seen in Step \ref{step:small-blow-up:1} that $\Gamma(c'_p)=\prod_{i\in I'} \Gamma(c'_{i,p})$ is a refinement of the cube $\Gamma(c_p)=\prod_{i\in I'} \Gamma(c_{i,p})$, given by the fact that each factor $\Gamma(c'_{i,p})$ is a simplicial refinement of the cuboid $\Gamma(c_{i,p})$, given by tiling the cuboid $\Gamma(c_{i,p})$ into standard simplices (determined by the order of the blow-up).
These refinements of the cubes $\Gamma(c_p)$ are compatible for different charts $c$, because the blow-ups performed are defined globally. 
This identifies the cell structure of $\Gamma(Y'_p)$ with a refinement of the cell structure on $\Gamma(Y_p)$.
We further note that this refinement is an isomorphism on the $0$-skeleta because $Y'_p\to Y_p$ induces a bijection on the sets of irreducible components, see Step \ref{step:small-blow-up:1}.

\begin{figure}[h]
    \centering
    \includegraphics[width=0.75\linewidth]{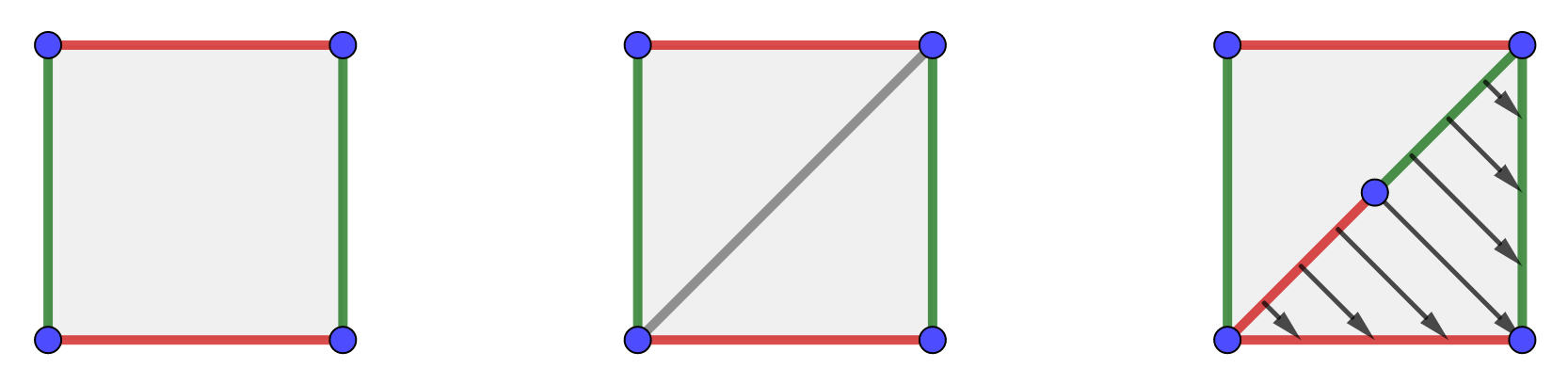}
    \caption{Left: Cube $\Gamma(c_p)$ of
    $\Gamma(Y_p)$. Middle: Dual complex
    $\Gamma(c_p')$ in the fiber $Y_p'\to Y_p$ of the resolution with colorless (grey)
    diagonal edge. Right: Subdivision
    $\hat{\Gamma}^1(c_p')$ of $\Gamma^1(c_p')$ and its coloring. 
    Map 
    $\varphi_p\colon \hat{\Gamma}^1(c_p')\to \Gamma^1(c_p)$
    of colored graphs is depicted by arrows.}
    \label{fig:diagonal-subdivide}
\end{figure}

\begin{stepp} \label{step:small-blow-up:4}
Definition of $\hat \Gamma^1(Y'_p)$ and $\varphi_p$.
\end{stepp}

By \eqref{eq:chart-Gamma(Y'_p)-small}, the $1$-skeleton $\Gamma^1(Y'_p)$ is covered by the $1$-skeleta $\Gamma^1(c'_p)$ with $c\in \mathcal C$.
Since $Y\to B$ is $S$-colored $D$-quasi-nodal, $\Gamma^1(c_p)$ is an $S$-colored graph.
Moreover, two edges of $\Gamma^1(c_p)$ have the same color if and only if they are parallel edges of the cuboid $\Gamma(c_p)$.  
In the above notation, 
$\Gamma(c_p)=\prod_{i\in I'}\Gamma(c_{i,p})$ and $\Gamma(c'_p)=\prod_{i\in I'}\Gamma(c'_{i,p})$. 
Here, $\Gamma(c'_{i,p})$ is a refinement of the cuboid $\Gamma(c_{i,p})$ given by standard simplices spanned by certain subsets of vertices of $\Gamma(c_{i,p})$.
The $1$-skeleton $\Gamma^1(c'_{i,p})$ has the same vertices as $\Gamma^1(c_{i,p})$ but contains additional edges, which we refer to as \emph{diagonal} edges. Any two such diagonal edges will meet only at vertices of $\Gamma^1(c_{i,p})$. See Figure \ref{fig:diagonal-subdivide}.

By the above description, $\Gamma^1(c'_p)$ identifies to the $1$-skeleton of the product $\prod_{i\in I'}\Gamma^1(c'_{i,p})$.
Hence, any edge $e$ of $\Gamma^1(c'_p)$ is given by a product of an edge $e_{i_0}$ of $\Gamma^1(c'_{{i_0},p})$ for some $i_0\in I'$ with vertices of $\Gamma^1(c'_{i,p})$ with $i\neq i_0$.
We will say that $e$ is a diagonal edge if $e_{i_0}$ is a diagonal edge of $\Gamma^1(c'_{{i_0},p})$.
Note moreover that any edge which is not diagonal is an edge that is contained in $\Gamma^1(c_p)\subset \Gamma^1(c'_p)$ (where we identify $\Gamma(c'_p)$ canonically with a refinement of $\Gamma(c_p)$) and hence it carries a unique color $s\in S$.
Hence,  $\Gamma^1(c'_p)$ is naturally a partially $S$-colored graph by declaring that diagonal edges are colorless.
These partial $S$-colorings are compatible for different charts (because the $S$-coloring of $\Gamma^1(c_p)$ comes from a global $S$-coloring of $\Gamma^1(Y_p)$) and hence induce the structure of a partially $S$-colored graph on $\Gamma^1(Y'_p)$.

Note that  $\Gamma^1(Y'_p)$ and $\Gamma^1(Y_p)$ have the same vertices and they agree on colored edges. 
Let now $e$ be a colorless edge of $\Gamma^1(Y'_p)$.
By \eqref{eq:chart-Gamma(Y'_p)-small}, there is a chart $c$ of $Y$ such that $e$ is contained in $\Gamma^1(c'_p)$. 
By Step \ref{step:small-blow-up::1.5}, we have a partial ordering on the vertices $\Gamma^0(Y'_p)$, which induces a total order on  $\Gamma^0(c'_p)=\Gamma^0(c_p)$.  
Using this, we get a canonical orientation of $e$, directing from the smaller to the larger vertex.  
Next, we consider $1$-chains $\gamma\in C_1(\Gamma^1(c_p),\Z)$ with $\partial \gamma=\partial e$.
The length of $\gamma$ is the minimal number of edges needed to write $\gamma$ as a linear combination of oriented edges.
We define the taxicab-norm $l(e)$ of $e$ as the smallest number $m\geq 1$ such that a $1$-chain $\gamma$ as above of length $m$ exists. 
The choice of such a $1$-chain $\gamma$ may not be unique, but there will be a unique one, that we call a {\it geodesic path}, with the property that the vertices that are covered by the path $\gamma$ are minimal in the lexicographic order induced by the total order on $\Gamma^0(c_p)$ from Step \ref{step:small-blow-up::1.5}.  
The pairwise compatibility, across charts $c \in \mathcal C$, of the tiling $\Gamma(c'_p)$ of $\Gamma(c_p)$ and of the ordering of its vertices shows that the above definition of $l(e)$ and the geodesic path are independent of the chart.

We then divide each diagonal edge $e$ of $\Gamma^1(Y'_p)$ (which comes with a natural orientation) into a chain $\hat e$ of $l(e)$ consecutive oriented edges. 
The geodesic path $\gamma$ with $\partial \gamma = \partial e$ has the same begin- and endpoint as the $1$-chain $\hat e$, and there is a unique isomorphism of oriented graphs between $\hat e$ and $\gamma$ that is the identity on these points. 
Since each edge of $\gamma$ is colored, this induces a canonical coloring of the edges of $\hat e$.
Replacing each diagonal edge $e$ of $\Gamma^1(Y'_p)$ by $\hat e$ then defines a refinement $\hat \Gamma^1(Y'_p)$ of $\Gamma^1(Y'_p)$ together with a canonical $S$-coloring.
Moreover, there is a natural morphism of $S$-colored graphs
$$
\varphi_p\colon \hat \Gamma^1(Y'_p)\longrightarrow  \Gamma^1(Y_p)
$$
which is the identity on the subgraph $\Gamma^1(Y'_p)\subset  \hat \Gamma^1(Y'_p)$ induced by the refinement \eqref{eq:chart-Gamma(Y'_p)-small}, and which maps for each diagonal edge $e$ of $\Gamma^1(Y_p')$ the chain of edges $\hat e$ to the unique geodesic path $\gamma$ with $\partial \gamma = \partial e$. See Figure \ref{fig:diagonal-subdivide}. This concludes Step \ref{step:small-blow-up:4}.

\begin{stepp} \label{step:small-blow-up:4.5}
$\varphi_p$ respects the ``product-structure''.
\end{stepp}

For further reference below, we note that if $c=\prod_{i\in I'}c_i$,  and $e$ is a diagonal edge of $\Gamma^1(c'_p)$, then $c'=\prod_{i\in I'}c'_i$ and $e$ is the product of a diagonal edge of $\Gamma^1(c'_{j})$ for some $j\in I'$ with vertices in the remaining factors.
The given product of vertices in the factors $\Gamma^1(c'_i)$ with $i\neq j$ define a face $\Gamma^1(c'_j)\hookrightarrow \prod_{i\in I'} \Gamma^1(c'_i)$ and any shortest length path $\gamma$ with $\partial \gamma=\partial e$ will be contained in that face.
We will refer to this property as to saying that
\begin{align} \label{eq:varphi_p-respects-product-structure}
\varphi_p\colon \hat \Gamma^1(c'_p)\longrightarrow \Gamma^1(c_p)
\end{align}
respects the natural ``product-structure'' induced by $\Gamma(c'_p)=\prod_{i\in I'} \Gamma(c'_{p,i})$ and $\Gamma(c_p)=\prod_{i\in I'} \Gamma(c_{p,i})$.

\begin{stepp} \label{step:small-blow-up:5}
    Items \eqref{item:def-admissible:1} and \eqref{item:def-admissible:2} in Definition \ref{def:admissible}.
\end{stepp}

Recall that the construction of $\varphi_p$ is local in the charts of $Y$.
In particular, $\varphi_p$ maps $\hat \Gamma^1(c'_p)$ to $\Gamma^1(c_p)$
and (in contrast to Proposition \ref{prop:blow-up-of-S-nodal-degenerations}) this map is a retract of the natural inclusion $\Gamma^1(c_p)\hookrightarrow \hat \Gamma^1(c'_p) $.
Hence, $\varphi_p\colon \hat \Gamma^1(Y'_p)\to \Gamma^1(Y_p)$ is a retract and so the diagram in item \eqref{item:def-admissible:1} of Definition \ref{def:admissible} commutes in the present case on the nose (not only up to homotopy).
 
Let $p\in D_J^\circ$ and $q\in D_{J'}^\circ$ be points of open strata with $J\subset J'$, such that $p$ specializes to $q$ via a continuous path $\gamma$, see Definition \ref{def:specialization}.
To check item \eqref{item:def-admissible:2} of Definition \ref{def:admissible}, cover $Y_q$ by finitely many charts $c$ and replace $p$ by a point sufficiently nearby $q$ so that $\gamma$ lies in the image of $c \to U = U_c$ for each of these charts $c$. 
Recall from  Step \ref{step:small-blow-up:1} that $\Gamma(c_q)=\prod_{i\in I'} \Gamma(c_{i,q})$, where $c_i$ is as in \eqref{eq:chart-c_i} and where $I' \subset I$ is the subset of $i \in I$ with $a_{si} = 1$ for some $s$; up to shrinking $U$ we have $u_i(q) = 0$ for all $i \in I'$.
Moreover, $\Gamma(c_p)=\prod_{i\in I''} \Gamma(c_{i,p})$, where $I''\subset I'$ is the subset of indices $i$ with $u_i(p)=0$ (equivalently, with $p\in D_i$). 
The specialization maps
$$
\operatorname{sp}\colon \Gamma^1(c_q)\longrightarrow \Gamma^1(c_p)\quad \quad \text{and}\quad \quad \operatorname{sp}\colon \hat \Gamma^1(c'_q)\longrightarrow \hat \Gamma^1(c'_p)
$$
are induced by the natural projection maps 
\begin{align*} 
&\pr_{I''}\colon \prod_{i\in I'} \Gamma^1(c_{i,q})\longrightarrow  \prod_{i\in I''} \Gamma^1(c_{i,q})\cong \prod_{i\in I''} \Gamma^1(c_{i,p})\quad{\rm and} \\
&\pr_{I''}\colon \prod_{i\in I'} \hat \Gamma^1(c'_{i,q})\longrightarrow \prod_{i\in I''} \hat \Gamma^1(c'_{i,q})\cong \prod_{i\in I''} \hat \Gamma^1(c'_{i,p}).
\end{align*}
Recall that an edge of $\hat \Gamma^1(c'_p)$ is given by the product of an edge $e_j$ of some $\hat\Gamma(c'_{j,p})$, $j \in I''$, 
with a product of vertices $v_i$ in the remaining factors. 
By Step \ref{step:small-blow-up:4.5}, $\varphi_p$ restricts to a map of the form 
$$
\varphi_p\colon \hat \Gamma^1(c_{j,p}')\times \prod_{i\in  I'' \setminus \{j\}} v_i  \longrightarrow  \Gamma^1(c_{j,p})\times \prod_{i\in I'' \setminus \{j\}} v_i .
$$
The analogous result holds for $\varphi_q$, which shows that these maps are compatible with the projection map $\pr_{I''}$ and hence with the specialization map.
This proves that the diagram in item \eqref{item:def-admissible:2} of Definition \ref{def:admissible} commutes, as we want.  

The proposition follows from Steps \ref{step:small-blow-up:1}--\ref{step:small-blow-up:5} above.
\end{proof}

\subsubsection{Proof of Theorem \ref{thm:admissible-modification}} \label{subsec:proof-Thm-admissible}

\begin{proof}[Proof of Theorem \ref{thm:admissible-modification}]
Let $Y\to S$ be an $S$-colored $D$-quasi-nodal morphism.
Fix a total ordering on the set $\Omega$ from Definition \ref{def:S-coloring}.  
We construct $g\colon Y'\to Y$ as a sequence of blow-ups as follows.  
We start with the largest element $\alpha\in \Omega$ and we blow-up (the proper transform of) $E_\alpha$ repeatedly as long as $o_{\alpha}(Y)\geq 2$.   
After each such blow-up, we replace $\Omega$ by $\Omega'=\Omega\sqcup \{\beta\}$, where $\beta$ is a new element, contained in $\Omega'_{s,i}$ if $\alpha\in \Omega_{s,i}$, that we define to be smaller than all elements of $\Omega$.
By Proposition \ref{prop:blow-up-of-S-nodal-degenerations}, 
$o_\alpha(Y')=o_\beta(Y')=o_\alpha(Y)-1$, while $o_\gamma(Y')=o_{\gamma}(Y)$ for all $\gamma\in \Omega\setminus \{\alpha\}$.
We repeat this process until $o_{\alpha}(Y)\leq 1$, at which point we replace $\alpha$ with the next element in $\Omega$.
Note that the maximum of $o_\alpha(Y')$ with $\alpha\in \Omega'$ drops by 1 via the above process after we have run through all elements in the original set $\Omega\subset \Omega'$.
Hence, the algorithm terminates and produces an $S$-colored $D$-quasi-nodal morphism $Y'\to S$ with $o_\alpha(Y')\leq 1$ for all $\alpha\in \Omega'$. 
So $Y'\to S$ is nearly $D$-nodal, see Definition \ref{def:D-nodal}. 

At each step of the algorithm, we have
constructed a collection of maps 
$\varphi =\{\varphi_p\}_{p\in B}$ of $1$-skeleta
satisfying hypotheses \eqref{item:def-admissible:1} and 
\eqref{item:def-admissible:2} of Definition 
\ref{def:admissible}.
The composition of those
admissible modifications that appear in Proposition \ref{prop:blow-up-of-S-nodal-degenerations} are again admissible (because $Y'\to B$ will be $S$-colored and $D$-quasi-nodal, and $\hat \Gamma^1(Y'_p)=\Gamma^1(Y'_p)$ for all $p\in B$).
Proposition \ref{prop:blow-up-of-S-nodal-degenerations} thus reduces the theorem to the case treated in Proposition \ref{prop:small-blow-up-of-S-nodal-degenerations}. 
\end{proof}

\subsubsection{Base change of $D$-nodal morphisms}

The following lemma is our original motivation for Definitions \ref{def:D-quasi-nodal} and \ref{def:S-coloring}.

\begin{lemma} \label{lem:base-change-of-D-nodal-degenerations}
Let $(\bar B,\bar D)$ be a pair of a regular variety $\bar B$ and an snc divisor $\bar D$ with components $\bar D_s$, $s\in S$.
Let  $\bar f \colon \bar Y\to \bar B$ be a strict $\bar D$-nodal morphism.
Let $B$ be a regular variety and let $\tau\colon B\to \bar B$ be a morphism such that $D=\tau^{-1}(\bar D)_{\rm{red}}$ is an snc divisor on $B$. Then $Y \coloneqq \bar Y\times_{\bar B} B$ is canonically $S$-colored $D$-quasi-nodal over $(B,D)$.
\end{lemma}

\begin{proof}
Let $\bar \Omega_s$ be the indexing set such that $\bar E_\alpha$ with $\alpha\in \bar \Omega_s$ are the irreducible components of $\bar Y\times_{\bar B}\bar D_s$.
Since $\bar f\colon \bar Y\to \bar B$ is $\bar D$-nodal, the local charts $c\in \mathcal C$ of $\bar Y$ are of the form $\{\ts =x_sy_s\mid s\in S'\}$ for some subset $S'\subset S$, where $\ts =0$ is a local equation of $\bar D_s$.
Since $\bar f$ is strict, there are no self-intersections of the divisors $\bar E_\alpha$ in these charts.
Hence, for $\alpha\in \bar \Omega_s$, $\bar E_\alpha \cap c$ is either empty or cut out by $(h_s)$ if $s\notin S'$, and it is cut out by $(\ts ,x_s)$ or $(\ts ,y_s)$ if $s\in S'$. 

Let $D_i\subset D$, $i\in I$, be the components of $D$.
We have
$\tau^\ast \bar D_s=\sum_{i\in I} a_{si}D_i$ 
for non-negative integers $a_{si}$.
In local charts, the function $\ts $ then pulls back to the monomial
$\tau^\ast \ts =\prod_i u_i^{a_{si}}$, where locally $D_i=\{u_i=0\}$. 
This shows that $Y \to B$ is $D$-quasi-nodal.

For $(s,i) \in S \times I$, we define $\Omega_{s,i}$ as follows. If $a_{si} = 0$ we let $\Omega_{s,i} = \emptyset$. 
If $a_{si}>0$,  we put
$$
\Omega_{s,i} \coloneqq \bar \Omega_s\times\{i\}. 
$$
For $(\alpha,i) \in \Omega_{s,i}$, we define the effective Weil divisor  $E_{\alpha,i}$  as the reduction of the intersection 
$$
(Y\times _{\bar Y} \bar E_\alpha )\cap (Y\times_{B} D_i) .
$$ 
From this it is clear that the divisors indexed by $\Omega\coloneqq \bigsqcup_{(s,i)\in S \times I}\Omega_{s,i}$ have the properties stated in item \eqref{item:S-coloring:1} in Definition \ref{def:S-coloring}.
Similarly, one checks that item \eqref{item:S-coloring:2} holds true, which 
concludes the proof of the lemma.
\end{proof}

\subsubsection{Bicoloring} \label{subsec:bicolor}

\begin{definition} \label{def:I-coloring}
Let $B$ be a smooth variety and let $D\subset B$ be an snc divisor with components $D_i$, $i\in I$. 
Let $Y\to B$ be a morphism which is $D$-nodal, nearly $D$-nodal or $D$-semistable.
Then for each $p\in B$, the $1$-skeleton $\Gamma^1(Y_p)$ carries an $I$-coloring given by the condition that an edge $e$ has color $i$ if the corresponding node is a specialization of a node over the generic point of $D_i$. 
\end{definition}

Let $Y\to B$ be an $S$-colored $D$-quasi-nodal morphism.
By Theorem \ref{thm:admissible-modification}, there is an admissible modification $Y'\to Y$ such that $Y'$ is strict $D$-semistable.
It follows that the graphs $\Gamma^1(Y'_p)$ carry an $I$-coloring.
This induces a canonical $I$-coloring on the refinement $\hat \Gamma^1(Y'_p)$.  
Hence, $\hat \Gamma^1(Y'_p)$ carries both an $I$-coloring and an $S$-coloring.
We will say that an edge $e$ has \emph{bicolor $(i,s)$} if it has color $i$ in the given $I$-coloring and color $s$ in the given $S$-coloring. 

We are now able to formulate the following consequence of the proof of Theorem \ref{thm:admissible-modification}, that we will need.

\begin{corollary} \label{cor:item:3-admissible}
Let $(\bar B,\bar D)$ be a pair of a smooth variety $\bar B$ and an snc divisor $\bar D$, $\bar D=\bigcup_{s\in S}\bar D_s$.
Let $\bar Y\to \bar B$ be a strict $\bar D$-nodal morphism and let $\tau\colon B\to \bar B$ be an alteration with $B$ regular such that $D\coloneqq \tau^{-1}(\bar D)_{\rm red}$ is an snc divisor with components $D_i$, $i\in I$. 
Consider the base change $Y \coloneqq \bar Y\times_BB$, which is naturally $S$-colored $D$-quasi-nodal over $B$ by Lemma \ref{lem:base-change-of-D-nodal-degenerations}.
Assume that for each $s\in S$ there is a component $D_{i_s}$ of $D$ with $\tau(D_{i_s})=\bar D_s$ and such that $\tau\colon B\to \bar B$ is \'etale at the generic point of $D_{i_s}$. 
Choose a total ordering on $\Omega$ such that any element of $\bigsqcup_{s\in S}\Omega_{s,i_s}$ is smaller than any element outside this subset. 
Let $(g\colon Y'\to Y,\varphi)$ be the admissible resolution associated to this ordering as in Theorem \ref{thm:admissible-modification}.

Then, for $s \in S$ and $p\in D_{i_s}$, the following holds:
\begin{enumerate}
    \item \label{item:cor:item:3-admissible:1} Any edge of bicolor $(i_s,s)$ of the refinement $\hat \Gamma^1(Y'_p)$ of $\Gamma^1(Y'_p)$ is an edge of $\Gamma^1(Y'_p)$.
    \item \label{item:cor:item:3-admissible:1.5} The graph $\hat \Gamma^1(Y'_p)$ does not have any edge of bicolor $(i_s,t)$ with $t\neq s$. 
    \item \label{item:cor:item:3-admissible:2} The map of $S$-colored graphs $\varphi_p\colon \hat \Gamma^1(Y'_p)\to \Gamma^1(Y_p)$ does not contract any edge of bicolor $(i_s,s)$.
    \item \label{item:cor:item:3-admissible:3} Any $s$-colored edge of $\Gamma^1(Y_p)$ is the image of some edge of bicolor $(i_s,s)$. 
\end{enumerate}
\end{corollary}

\begin{proof}
By our construction, $\varphi_p\colon \hat \Gamma^1(Y'_p)\to \Gamma^1(Y_p)$ restricts to a map $\hat \Gamma^1(c'_p)\to \Gamma^1(c_p)$ for each chart $c\in \mathcal C$ as in Definition \ref{def:D-quasi-nodal}.
It thus suffices to prove the corollary in the case where $Y=c$ is a single chart 
such that $u_{i_s}$ appears nontrivially for some $t\in S$, i.e.~$a_{ti_s}\neq 0$ for some $t$. Then $a_{si_s} = 1$ and $a_{ti_s} = 0$ for $t \neq s$ because $\tau \colon B \to \bar B$ is étale at the generic point of $D_{i_s}$. 
The normal form \eqref{eq:normal-form-S-nodal} thus decomposes  into a fiber product $W \times_B Z$, where
         $$ 
         W\coloneqq \left\{
         \textstyle u_{i_s}\prod_{i\in I\setminus \{i_s\}}u_{i}^{a_{si}}=x_{s}y_{s}\right\}  \quad \quad \text{and}\quad \quad Z\coloneqq \left\{
         \textstyle
         \prod_{
         i\in I\setminus \{i_s\}}u_i^{a_{ti}}=x_ty_t \mid t\in S'\setminus \{s\} \right\} ,
         $$    
for some subset $S'\subset S$. 
Moreover, the edges of $\Gamma^1(Z_p)$ have colors contained in $S'\setminus \{s\}$, while the unique edge of $\Gamma^1(W_p)$ has color $s$. 

Recall that in the resolution algorithm from Theorem \ref{thm:admissible-modification} we are first blowing up $E_\alpha$ with $\alpha\in \Omega$ and $o_\alpha(Y) \geq 2$ according to the order of $\Omega$, starting with the largest element.
By our choice of order, this means that in the above chart $c$, we first blow up  $(u_i,x_s)$ or $(u_i,y_s)$ with $i\neq i_s$, and also possibly some $(u_i,x_t)$ or $(u_i, y_t)$ with $t\neq s$, before we blow up $(u_{i_s},x_s)$ or $(u_{i_s},y_s)$. 
If we blow-up $(u_i,x_s)$, then the respective morphism $\varphi_p$ has the property that $s$-colored edges contained in the blow-up chart \eqref{eq:S-nodal-blow-up-u_i-chart} are contracted, while the $s$-colored edges in \eqref{eq:S-nodal-blow-up-x_s-chart} are not.

To prove items \eqref{item:cor:item:3-admissible:1}--\eqref{item:cor:item:3-admissible:2}, since we are only interested in edges of bicolor $(i_s,t)$, we deduce from the local charts in \eqref{eq:S-nodal-blow-up-x_s-chart} and \eqref{eq:S-nodal-blow-up-u_i-chart} that after $a_{si}$-many blow-ups of $(u_i,x_s)$ or $(u_i,y_s)$ for each $i$ with $i\neq i_s$ (and also several blow-ups of $(u_i, x_t)$ or $(u_i, y_t)$ for $t \neq s$), we may without loss of generality reduce to the case of a product $W\times Z$ of the form
         \begin{align}\label{align:WtimesZ}
         W\coloneqq \left\{u_{i_s}=x_{s}y_{s}\right\}  \quad \quad \text{and}\quad \quad Z\coloneqq \left\{ \textstyle
         \prod_{
         i\in I\setminus \{i_s\}
         }u_i^{a_{ti}}=x_ty_t \mid t\in S'\setminus \{s\} \right\} .
         \end{align}
Note that also item \eqref{item:cor:item:3-admissible:3} can be reduced to the case of a product $W \times Z$ with $W$ and $Z$ as in \eqref{align:WtimesZ}, because for each blow-up, $\varphi_p$ is a morphism of $S$-colored graphs which is surjective on edges. 
Our resolution algorithm from Theorem \ref{thm:admissible-modification}, applied to this product, will turn $W\times Z$ into the product $W\times Z'$ for some resolution $Z'\to Z$ such that the $s$-colored edges of $\hat \Gamma^1(W_p\times Z'_p)$ are given by $\Gamma^1(W_p)\times \Gamma^0(Z'_p)$. 
This implies items \eqref{item:cor:item:3-admissible:1} and \eqref{item:cor:item:3-admissible:1.5} in the corollary. 
Moreover, the map $\varphi_p$ respects this product structure (cf.\ Step \ref{step:blow-up-E_alpha-big:4} in Proposition \ref{prop:blow-up-of-S-nodal-degenerations} and Step \ref{step:small-blow-up:4.5} in Proposition \ref{prop:small-blow-up-of-S-nodal-degenerations}).
In other words, each element of $\Gamma^1(W_p)\times  \Gamma^0(Z'_p)$
is of the form $e\times v$ for some vertex $v\in \Gamma^0(Z'_p)$ and $\varphi_p(e\times v)=e\times \varphi_p(v)$, where $\varphi_p$ also denotes the map $\hat \Gamma^1(Z'_p)\to \Gamma^1(Z_p)$. 
This proves items \eqref{item:cor:item:3-admissible:2} and \eqref{item:cor:item:3-admissible:3} and concludes the proof. 
\end{proof}

\section{From algebraicity to quadratic splittings of matroids} \label{sec:algebraicity->d-QS}

\subsection{Setup} \label{subsec:set-up} 
Let $(\underline R,S)$ be a regular matroid with integral realization $S\to U^\ast$, $s\mapsto y_s$, for some free $\Z$-module $U$ of rank $g={\rm rank} (\underline R)$.
Let $B$ be a smooth affine variety with distinguished point $0\in B$.
Let $\ts $, $s\in S$, be regular functions on $B$ such that $(\ts )_{s\in S}\colon B\to \mathbb A^S$ is \'etale and such that $\ts $ restricts to the $s$-th coordinate function on a disc $\Delta^S\subset B$ centered around $0$.
Up to shrinking $B$ if necessary, we may assume that the divisors $H_s\coloneq\{\ts =0\}\subset B$ are irreducible.
We refer to $H_s$ as the \emph{$s$-th coordinate hyperplane} on $B$ and denote by $H\coloneq \{\prod_{s\in S}\ts =0\}$ the union of the coordinate hyperplanes $H_s$ and by $B^\star\coloneq B\setminus H$ the complement of $H$ in $B$.

\begin{definition} \label{def:matroidal-family} 
Let $\pi^\star \colon X^\star\to B^\star$ be a smooth projective family of $g$-dimensional principally polarized abelian varieties.
We say that $\pi^\star$ is a {\it matroidal family} associated to the regular 
matroid $(\underline R,S)$ with integral realization $S\to U^\ast$, $s\mapsto y_s$, if the following holds, where $\bzero \in (\Delta^\star)^S$ is a base point:
\begin{enumerate}  
    \item \label{item:def:matroidal-family:1}
    The family $X^\star_{(\Delta^\star)^S}\to (\Delta^\star)^S$ over the punctured polydisc $(\Delta^\star)^S\subset B^\star$ has unipotent monodromies about the coordinate hyperplanes.
    \item \label{item:def:matroidal-family:2}
    There exists an isomorphism
    $$U\xlongrightarrow{\cong} {\rm gr}^W_0 H_1(X_{\bzero},\Z),$$
    where $W_\bullet$ denotes the weight filtration of the limit mixed Hodge structure associated to the family $X_{(\Delta^\star)^S}\to (\Delta^\star)^S$, under which 
\item \label{item:def:matroidal-family:3} for all $s\in S$ 
    there is a positive integer $d_s$ such that the bilinear form $d_s y^2_s$ on $U$
    is identified with the monodromy bilinear form (cf.~Definition   \ref{def:monodromy-bilinear-form})
    on ${\rm gr}^W_0 H_1(X_{\bzero},\Z)$ induced by the monodromy about $\ts =0$. 
\end{enumerate}
\end{definition}

Moreover, by Lemma \ref{lem:unique-integer-realization}, the above notion depends only on the matroid $\underline R$ and not on the chosen realization.

\begin{remark}\label{rem:matroidal-family:maximal-degeneration}
Since $B\to \A^S$ is \'etale, the cardinality of $S$ coincides with the dimension of $B^\star$.
Note moreover that we ask in Definition \ref{def:matroidal-family} that the dimension $g$ of the fiber $X_{t_0}$ agrees with the rank of the matroid $\underline R$.
In other words, the corresponding degeneration of abelian varieties is maximal.
In the proof of our main results, such as Theorems \ref{thm:matroidal-intro} and \ref{thm:reduction-to-combinatorics-intro}, we use this condition only when invoking \cite[Theorem 7.1]{survey} in Proposition \ref{prop:filling-of-matroidal-family} below; we are confident that the latter results also hold in the general case where the fiber dimension may be larger than the dimension of ${\rm gr}^W_0 H_1(X_{\bzero},\Z)$, but we did not try to prove this and concentrated on the case of maximal degenerations for simplicity.
\end{remark}

\begin{remark} \label{rem:matroidal-family:existence}
By \cite[Proposition 4.10]{survey}, for any regular matroid $(\underline R, S)$ with integral realization $S \to U^\ast$, there exists a family of principally polarized abelian varieties $\pi^\star \colon X^\star\to B^\star$ which is a matroidal family associated to $\uR$.
Furthermore, $X^\star_{(\Delta^\star)^S}\to (\Delta^\star)^S$ 
is uniquely determined by $\uR$ and
$(d_s)_{s\in S}$ up to an analytic
deformation, see \cite[Remark 2.31]{survey}.
\end{remark}

\begin{remark} \label{rem:d_s=d-for-all-s}
Let $d$ be a positive integer with $d_s\mid d$ for all $s$, where $d_s$ are the positive integers from Definition \ref{def:matroidal-family} above.
Up to possibly shrinking $B$ and performing a base change that restricts to $(\ts )_{s\in S}\mapsto (\ts^{d/d_s})_{s\in S}$ on the polydisc $\Delta^S$, we may assume that $d_s=d$ for all $s$.
\end{remark}

\begin{proposition} \label{prop:filling-of-matroidal-family}
Let $\pi^\star \colon X^\star\to B^\star$ be a matroidal family of principally polarized abelian varieties, associated to $(\underline R,S)$.  
Then, up to replacing $B$ by an \'etale local neighborhood of $0$, there is a flat projective morphism $\pi\colon X\to B$ of quasi-projective varieties, such that:
\begin{enumerate}
    \item $X\times_BB^\star=X^\star$;
    \item $\pi\colon X\to B$ is an $H$-nodal morphism, 
    where $H=\left\{\prod_{s\in S}\ts =0\right\}$; if $d_s\geq 2$ for all $s$ in item \eqref{item:def:matroidal-family:3}, then $\pi$ is in fact strict $H$-nodal.
\end{enumerate}
\end{proposition}
\begin{proof}
Note that flatness is a formal consequence of 
being $H$-nodal.
The result follows 
by applying a 
suitable Mumford construction, 
see \cite[Theorem 7.1]{survey}.
\end{proof}

\subsection{Quadratic splittings of matroids}

Let $G$ be an $S$-colored graph with edge set $E=\bigsqcup_{s\in S}E_s$. 
The choice of an orientation on each edge of $E$ induces an inclusion $H_1(G,\Lambda)\subset \Lambda^E$.
For $e\in E$, we denote by $x_e\colon \Lambda^E\to \Lambda$ the $e$-th coordinate function.
According to our conventions, we denote the associated bilinear form by $x_e^2$ and we denote by 
the same symbol its restriction
to $ H_1(G,\Lambda)$.
These bilinear forms do not depend on the choice 
of orientation, because changing the orientation 
of $e$ changes $x_e$ by a sign.
 
We have the following variant of Definition \ref{def:d-Lambda-splitting-general} from the introduction.

\begin{definition} \label{def:d-Lambda-splitting-graph}
Let $(\underline R,S)$ be a regular matroid with integral realization $S\to U^\ast$, $s\mapsto y_s$. 
Let $\Lambda$ be a ring and let $d$ be a positive integer.
    A {\it quadratic $\Lambda$-splitting of level $d$} of $(\underline R,S)$ in a graph $G$ is an $S$-coloring $E=\sqcup_{s\in S} E_s$ of the edges of $G$ together with an embedding $U_{\Lambda}\hookrightarrow H_1(G,\Lambda)$ which induces a decomposition
        \begin{align} \label{eq:def-H_1(G)=U+U'}
        H_1(G,\Lambda)=U_\Lambda\oplus U' , 
        \end{align} 
        for some $U'\subset H_1(G,\Lambda)$,
        such that for all $s\in S$, the following holds for the bilinear form $Q_s\coloneq \sum_{e\in E_s}x_e^2$ on $H_1(G,\Lambda)$:
        \begin{enumerate}
            \item the decomposition \eqref{eq:def-H_1(G)=U+U'} is orthogonal with respect to $Q_s$;
            \item the restriction of $Q_s$ to $U_\Lambda$ agrees with $d\cdot y_s^2$.
        \end{enumerate} 
\end{definition}

\begin{remark} \label{rem:Lambda-solutions-well-defined}
It follows from Lemma \ref{lem:unique-integer-realization} that the existence of a quadratic $\Lambda$-splitting of $(\underline R,S)$ in a graph $G$ does not depend on the choice of integral realization.
\end{remark}

\begin{remark} \label{rem:Lambda-splitting-graph-versus-general}
Definition \ref{def:d-Lambda-splitting-graph} is a special case of a quadratic $\Lambda$-splitting of level $d$ of $(\underline R,S)$ in the cographic matroid associated to $G$ (see Definition \ref{def:d-Lambda-splitting-general}); it corresponds to the special case where the $\sum_{s}a_{se}=1$ for all edges $e\in E$.
The following lemma shows that in fact both notions are equivalent.
This implies for instance that the notion passes to deletions, as this is obvious for Definition \ref{def:d-Lambda-splitting-general}. 
\end{remark}

\begin{lemma}\label{lem:Lambda-splitting-graph-versus-general}
Let $(\underline R,S)$ be a regular matroid with integral realization $S\to U^\ast$, $s\mapsto y_s$. 
Then $(\underline R,S)$ admits a quadratic $\Lambda$-splitting of level $d$ in a cographic matroid $M^\ast(G)$ associated to a graph $G$ (see Definition \ref{def:d-Lambda-splitting-general}) if and only if it admits a quadratic $\Lambda$-splitting of level $d$ in some graph (see Definition \ref{def:d-Lambda-splitting-graph}). 
\end{lemma}
\begin{proof}
One direction is obvious, see Remark \ref{rem:Lambda-splitting-graph-versus-general}.
For the converse, consider the non-negative integers $a_{se}$ from Definition \ref{def:d-Lambda-splitting-general}.
We may first divide each edge $e\in E$ with $\sum_s a_{se}\geq 1$ into a chain of $\sum_s a_{se}$ colored edges, such that precisely $a_{se}$-many of them have color $s$.
This reduces us to the situation where $\sum_{s\in S}a_{se}\in \{0,1\}$ for all $e\in E$.
We may then define a partial $S$-coloring of $G$ by declaring an edge $e\in E$ to be of color $s$ if and only if $a_{se}=1$.
We thus obtain all the conditions from Definition \ref{def:d-Lambda-splitting-graph} apart from the fact that the $S$-coloring of $G$ is only a partial $S$-coloring.
We refer to this condition as a {\it weak quadratic $\Lambda$-splitting of level $d$ into the graph $G$}.

To prove the lemma, it suffices to reduce to the case where the aforementioned partial $S$-coloring is in fact an $S$-coloring, i.e.\ to the case where $\sum_{s\in S}a_{se}=1$ for all $e\in E$.
We do so by induction on the number of colorless edges of $G$.
To this end, let $e\in E$ be a colorless edge and let $G\to G'$ be the contraction of $e$.
If $e$ is not a loop, then it can be extended to a spanning forest of $G$.
It follows that $H_1(G,\Lambda)\cong H_1(G',\Lambda)$ and we obtain an induced weak $\Lambda$-splitting of $\underline R$ in $G'$ which has fewer colorless edges than the one into $G$.
If $e$ is a loop, then $e\in U'$ because the linear forms $y_s$, $s\in S$, form a basis of $U^\ast$.
It is then again easy to see that we obtain an induced weak $\Lambda$-splitting of $\underline R$ in $G'$ which has fewer colorless edges than the one into $G$.
This concludes the proof. 
\end{proof}

The main result of this section is the following, where $B$ and $B^\star=B\setminus H$ are as in Section \ref{subsec:set-up}.

\begin{theorem} \label{thm:algebraic->d-QE} 
Let $\pi^\star\colon X^\star\to B^\star$ be a matroidal family of principally polarized abelian varieties associated to a regular matroid $(\underline R,S)$ with integral realization $S\to U^\ast$, $s\mapsto y_s$, see Definition \ref{def:matroidal-family}. 
Let $\ell$ be a prime and assume that an $\ell$-prime multiple of the minimal curve class of the geometric generic fiber of $\pi$ is represented by an algebraic curve.
Let $\Lambda$ be a $\Z_{(\ell)}$-algebra.
Then there is a positive integer $d$ and an $S$-colored graph $G$, such that the matroid $(\underline R,S)$ admits a quadratic $\Lambda$-splitting of level $d$ in $G$.
\end{theorem}

\begin{remark} \label{rem:algebraic->d-QE-dependence-on-d}
The proof will show that the integer $d$ in the above theorem does not depend on $\Lambda$.
Another way of seeing this fact is to note that the theorem follows formally from the case $\Lambda=\Z_{(\ell)}$ after applying $\otimes_{\Z_{(\ell)}} \Lambda$, and so the value of $d$ that works for $\Z_{(\ell)}$ works in fact for any $\Z_{(\ell)}$-algebra $\Lambda$.
\end{remark}

The remainder of this section is devoted to a proof of the above theorem.
To this end, we will in this section denote by $\Lambda$ an algebra over the localization $\Z_{(\ell)}$ of $\Z$.
We further use Remark \ref{rem:d_s=d-for-all-s} to assume without loss of generality that  
\begin{align} \label{eq:d=d_s}
\text{there is an integer $d\geq 2$ with $ d_s=d$ for all $s\in S$ in Definition \ref{def:matroidal-family}.} 
\end{align}
That is, the $s$-th monodromy bilinear form associated to $\pi^\star\colon X^\star\to B^\star$ identifies to $dy_s^2$ for all $s\in S$.
Since $d\geq 2$, we may by Proposition \ref{prop:filling-of-matroidal-family} choose a strict $H$-nodal extension $\pi\colon X\to B$ of $\pi^\star$.

\subsection{Base change and resolution}

By the assumption in Theorem \ref{thm:algebraic->d-QE}, the geometric generic fiber of $\pi$ contains a curve whose cohomology class is $m$ times the minimal class for some integer $m$ that is coprime to $\ell$.
While Proposition \ref{prop:filling-of-matroidal-family} allows us to freely perform base changes of $B$ that are only ramified along $H$, 
such base changes will in general not suffice to descend this curve to the generic fiber of $X\to B$.
Nonetheless, we are able to make the following reductions.

\begin{lemma}\label{lem:reductions/set-up}
In order to prove Theorem \ref{thm:algebraic->d-QE}, we may assume that there is an alteration 
$$
\tau\colon B'\longrightarrow B
$$
with base change $X'=X\times_BB'\to B'$ and the following properties:
\begin{enumerate}
\item \label{item:lem:reduction:1} The generic fiber $X'_\eta$ of $X'\to B'$ contains a finite number of integral curves $C'_j\subset X'_\eta$,  each with a rational point in its smooth locus, such that $\sum_j [C'_j]\in H^{2g-2}(X'_{\bar \eta},\Z_{\ell})$ 
is an $\ell$-prime multiple of 
the minimal class.
\item \label{item:lem:reduction:3} $B'$ is a regular quasi-projective variety and the reduction $D\coloneq \tau^{-1}(H)_{{\rm red}}$, where $H=\{\prod_{s\in S}\ts =0\}$, is an snc divisor with components $D_i$, $i\in I$;
the reduced preimage $\tau^{-1}(0)_{\rm red}$ 
is an snc divisor as well.
\item \label{item:lem:reduction:2} For each $s\in S$, there is a component $D_{i_s}$ of $D$ such that $\tau(D_{i_s})=H_s$ and $\tau$ is \'etale at the generic point of $D_{i_s}$.
\item \label{item:lem:reduction:5} The morphism $X' \to B'$ is $D$-quasi-nodal and admits a natural $S$-coloring via Lemma \ref{lem:base-change-of-D-nodal-degenerations}.
There is an admissible resolution $(g\colon X''\to X',\varphi)$ as in Theorem \ref{thm:admissible-modification} 
such that the additional properties from Corollary \ref{cor:item:3-admissible} hold true.
That is, if  $D_{i_s}$ is as in \eqref{item:lem:reduction:2} and 
$p\in D_{i_s}$, then the following hold:
\begin{enumerate}
    \item Any edge of bicolor $(i_s,s)$ of the refinement $\hat \Gamma^1(X''_p)$ of $\Gamma^1(X''_p)$ is an edge of $\Gamma^1(X''_p)$, cf.\ Definition \ref{def:admissible}. 
    \item The graph $\hat \Gamma^1(X''_p)$ does not have any edge of bicolor $(i_s,t)$ with $t\neq s$. 
    \item The map of $S$-colored graphs $\varphi_p\colon \hat \Gamma^1(X''_p)\to \Gamma^1(X'_p)$ 
    does not contract any edge of bicolor $(i_s,s)$. 
    \item Any edge of $\Gamma^1(X'_p)$ of color $s$ is in the image of some edge of bicolor $(i_s,s)$. 
\end{enumerate} 
\end{enumerate}
\end{lemma}
\begin{proof} 
In the notation of Theorem \ref{thm:algebraic->d-QE}, there is a finite field extension $L/\C(B)$, such that the base change $X_L$ contains an effective curve whose cohomology class in the geometric generic fiber $X_{\bar L}$ is an 
$\ell$-prime multiple of the minimal class.
Up to enlarging $L$, we may assume that each component of the reduction of this curve admits an $L$-rational point in its smooth locus. A straightforward (and well-known) computation with ramification indices shows that we can find a sequence of alterations of normal quasi-projective varieties
$$
B'\longrightarrow \widetilde B\longrightarrow B
$$
and some integer $d'\geq 1$, such that $L\subset \C(B')$ and the following hold:
\begin{itemize}
    \item $\widetilde B\to B$ is finite and restricts to the cover $(\ts )_{s\in S}\mapsto (\ts ^{d'})_{s\in S}$ over $\Delta^S\subset B$; 
    \item $B'\to \widetilde B$ is \'etale over the generic points of the components of $\tilde{H}_s$, where $\tilde{H}_s\subset \widetilde B$ denotes the reduction of the preimage of $H_s \subset B$ in $\widetilde B$. 
\end{itemize}
Up to replacing $d$ by $d d'$ in
\eqref{eq:d=d_s}, we may, by
Proposition \ref{prop:filling-of-matroidal-family},
assume without loss of generality
that $B=\widetilde B$.
By resolution of singularities, we thus may 
further assume that there is an alteration 
$$
\tau\colon B'\longrightarrow B
$$
which satisfies items \eqref{item:lem:reduction:1}--\eqref{item:lem:reduction:2} in the lemma. 

To prove item \eqref{item:lem:reduction:5}, we recall that $\pi\colon X\to B$ is strict $H$-nodal.
By Lemma \ref{lem:base-change-of-D-nodal-degenerations} (applied to $\bar B=B$, $\bar D=H$, and $\bar Y=X$), the base change $X'=X\times_BB'$ is thus canonically $S$-colored $D$-quasi-nodal, where $D=\tau^{-1}(H)_{\red}$, see Definition \ref{def:S-coloring}.
In particular, we get an index set $\Omega=\bigsqcup_{s,i}\Omega_{s,i}$ and a collection of Weil divisors $E_\alpha$ for all $\alpha\in \Omega$.
By item \eqref{item:lem:reduction:3}, for each $s\in S$ there is a component $D_{i_s}$ of $D$ such that $\tau(D_{i_s})=H_s$ and $\tau$ is \'etale at the generic point of $D_{i_s}$.
We then pick a total ordering of $\Omega$ such that any element of $\bigsqcup_{s\in S}\Omega_{s,i_s}$ is smaller than any element outside this subset.
We apply  Theorem \ref{thm:admissible-modification} to this total ordering and get an admissible resolution $(g\colon X''\to X',\varphi)$ of $X'\to B'$, which, by our choice of ordering, satisfies the conclusions of Corollary \ref{cor:item:3-admissible}. 
\end{proof}

The remainder of this section is devoted to a proof of Theorem \ref{thm:algebraic->d-QE}.
To this end, we will from now on assume that \eqref{eq:d=d_s} as well as items 
\eqref{item:lem:reduction:1}--\eqref{item:lem:reduction:5} in Lemma \ref{lem:reductions/set-up} hold true. 
In particular, $X''$ is regular and the morphism $X'' \to B'$ is strict $D'$-semistable and projective.

We summarize the morphisms from Lemma \ref{lem:reductions/set-up} in the following diagram:
$$
\xymatrix{X''\ar[r]^{g} \ar[rd]& \ar[d] \ar[r]X'\ar[r]& X\ar[d] \ar@{}[dl] | \square \\
& B'\ar[r]^{\tau} &B.}
$$

\subsection{Bringing the curve into good position}

\begin{proposition}\label{prop:mathcal-E}
There is a vector bundle $\mathcal E$ on $X''$, such that for $b\in B'$ general, the $(g-1)$-st Segre class $s_{g-1}(\mathcal E|_{X''_{b}})$ is an $\ell$-prime multiple of the minimal class of $X''_b=X_{\tau(b)}$. 
\end{proposition}  

\begin{proof}
As a consequence of item \eqref{item:lem:reduction:1} in Lemma \ref{lem:reductions/set-up}, there is an open subset $U'\subset B'$ with $X''_{U'}= X'_{U'} $ and a smooth projective family of curves $C'\to U'$, whose fibers may be disconnected but such that each component admits a section, together with a proper morphism $h\colon C'\to X''_{U'}$, such that $h_\ast [C'_b]$ is 
$m$ times the minimal class of $X''_{b}=X'_b=X_{\tau(b)}$ for all $b\in U'$, where $m$ is coprime to $\ell$.
The existence of the sections determines an Abel--Jacobi mapping $C'\to JC'$ and hence a morphism of abelian schemes $JC'\to X''_{U'}$, where $JC'$ denotes the relative Jacobian of $C'$ over $U'$.
We apply Lemma \ref{lem:minimal-class-algebraic} to the generic fiber of $C'\to U'$ and spread the result out to a neighborhood of the generic point of $U'$.
Then, by that lemma, up to shrinking $U'$, we can assume that there is a vector bundle $\ca F_{U'}$ on $X''_{U'}$ whose 
$i$-th Segre class restricts on a general fiber $X''_b$ to  $m^i\cdot \Theta^i/i!$ for all $i$.
Here, $\Theta$ denotes the theta divisor class on $X_b''$. 
By \cite[II, Exercise 5.15]{hartshorne}, there exists a coherent sheaf $\mathcal F$ on $X''$ that restricts to the vector bundle $\ca F_{U'}$ on $X''_{U'}\subset X''$.
This coherent sheaf has then the property that, for general $b\in B'$,
\begin{align}\label{align:emTheta}
s(\mathcal F|_{X''_{b}})=e^{m\Theta}.\end{align} 
Note that $X''$ is quasi-projective as the morphism $X'' \to B'$ is projective and $B'$ is quasi-projective.
We pick a sufficiently ample line bundle $L$ on $X''$, such that $c_1(L)$ is $\ell$-divisible.
Consider the natural map
$$
 \ca E^0 \coloneqq H^0(X'',\mathcal F\otimes L^{n})\otimes L^{-n}\longrightarrow \mathcal F ,
$$
which is surjective for $n\gg 0$. 
Applying the same construction to the kernel of the above map and repeating the process, we produce 
by Hilbert's syzygy theorem (which uses that $X''$ is regular)  a finite locally free resolution
$$
\mathcal E^\bullet\coloneqq   0\to \mathcal E^N\to \mathcal E^{N-1}\to \dots \to \mathcal E^2\to \mathcal E^1\to \mathcal E^0\to \mathcal F\to 0,
$$
of $\mathcal F$, where $N=\dim X''$ and $\mathcal E^i$ is, for $i<N$, a direct sum of some tensor powers of $L$.
Then the total Segre class of $\mathcal F$ satisfies
$$
s(\mathcal F)=\prod_{i }s(\mathcal E^{2i})\cdot  c(\mathcal E^{2i-1}) .
$$
We let $\mathcal E\coloneqq \mathcal E^N$.
Since $c_1(L)$ is $\ell$-divisible and $\mathcal E^i$ is, for $i<N$, a direct sum of some tensor powers of $L$, we deduce:
$$
c(\mathcal E^i)\equiv s(\mathcal E^i)\equiv 1\mod \ell \quad \quad \text{for $i<N$.}
$$
Hence,
$$
s(\mathcal F)\equiv \begin{cases}
    s(\mathcal E) \mod \ell \ \ \ &\text{if $N$ is even};\\
    c(\mathcal E)\mod \ell \ \ \ &\text{if $N$ is odd}.
\end{cases}
$$
By \eqref{align:emTheta}, and the fact that  
the total Chern class is the inverse of the total Segre class, we find that
$$
s(\mathcal E|_{X''_b})\equiv e^{\pm m\Theta} \mod \ell .
$$
Since $m$ is coprime to $\ell$, it follows that $s_{g-1}(\mathcal E|_{X''_{b}})$ is an $\ell$-prime multiple of the minimal class, for $b\in B''$ general, 
as desired. 
\end{proof}
Let $\mathcal E$ be the vector bundle from Proposition \ref{prop:mathcal-E}.
We define
$$
Y\coloneqq \mathbb P(\mathcal E) .
$$
Note that $Y\to X''$ is a smooth morphism. 
In particular, $Y\to B'$ is strict $D$-semistable (hence flat), because the same holds for $X''\to B'$.
Let further $L$ be a very ample line bundle on $\mathbb P(\mathcal E)$, given by the tensor product of $\mathcal O_{\mathbb P(\mathcal E)}(1)$ with an $\ell$-divisible and sufficiently positive bundle, and consider a complete intersection
$$
C\subset Y
$$
of $r+g-1$ many general elements of $|L|$, where $r={\rm rk}(\mathcal E)-1$. 
Note that $C$ is regular by Bertini's theorem.
It is not hard to see that $C\to B'$ is a flat family of curves, but we will only use the following more precise statement for a restriction of this morphism to a suitable open subset $B^\circ\subset B'$.

\begin{lemma} \label{lem:C-on-X''-in-good-position} 
Let $P\subset B'$ be a finite set of closed points.
Then there is a dense open subset $B^\circ\subset B'$ which contains $P$, such that the following holds: 
        \begin{enumerate}  
        \item \label{item:lem:C-on-X''-in-good-position:1}
       The base change $C_{B^\circ}\to B^\circ$ is strict $D^\circ$-semistable, where $D^\circ=D\cap B^\circ$.
       For $b\in B^\circ$, $C_b\subset Y_b$ meets each double locus of $Y_b$ transversely, but misses any deeper stratum of $Y_b$, and induces a bijection on irreducible components.
       Moreover, the singularities of the morphism 
       $C_{B^\circ}\to B^\circ$ are $D^\circ$-nodal, 
       given by transverse slices of the 
       singularities of $Y_{B^\circ}\to B^\circ$. 
       \item \label{item:lem:C-on-X''-in-good-position:1.5} 
       If $g\geq 3$, then for all $p\in B^\circ$, the natural map $\iota\colon C_b\to X''_b$ is a closed embedding. 
        \item \label{item:lem:C-on-X''-in-good-position:2}
        For  general $b\in B^\circ$,  $\iota_\ast[C_b]$ represents $m$ times the minimal class of $X''_b=X_{\tau(b)}$, for some positive integer $m$ that is coprime to $\ell$. 
    \end{enumerate} 
\end{lemma}
\begin{proof}  
Item \eqref{item:lem:C-on-X''-in-good-position:1} follows by applying the relative Bertini theorem in the form of Lemma \ref{lem:Bertini}.
Item \eqref{item:lem:C-on-X''-in-good-position:1.5} follows up to shrinking $B^\circ$ around $P$ from the assumption that $g\geq 3$ and the well-known fact that a general complete intersection curve embeds via a smooth morphism whose target has dimension at least $3$ into the target of the morphism. 
Finally, item \eqref{item:lem:C-on-X''-in-good-position:2} follows from the fact that the $(g-1)$-st Segre class of $\mathcal E$ restricts to an $\ell$-prime multiple of the minimal class on the general fiber of $X''\to B'$.
\end{proof}

To fix notation, we make now the following 

\begin{definition} \label{def:P-and-B^circ}
We fix a finite set $P\subset B^\circ\cap \tau^{-1}(0)$ of closed points which contains a point on each connected component of each non-empty open stratum $D_J^\circ$ of $D$ which is contained in $\tau^{-1}(0)$. 
We then denote  by $B^\circ\subset B'$ the open subset from Lemma \ref{lem:C-on-X''-in-good-position} which contains $P$ and such that items \eqref{item:lem:C-on-X''-in-good-position:1}--\eqref{item:lem:C-on-X''-in-good-position:2} in the lemma hold true.
\end{definition}
  
Since $(g\colon X''\to X',\varphi)$ 
is admissible, for each $p\in B'$ there is a 
canonical $S$-colored graph $\hat \Gamma^1(X''_p)$, 
given by a refinement of $\Gamma^1(X''_p)$, see 
Definition \ref{def:admissible}.
Since $Y=\mathbb P(\mathcal E)\to X''$ is smooth, 
we further have for all $p\in B'$ a canonical 
identification of dual complexes 
$\Gamma(Y_p)=\Gamma(X''_p) $ and associated graphs
$$
\Gamma^1(Y_p)=\Gamma^1(X''_p) .
$$

\begin{proposition}  \label{prop:C-on-X''-good-position-Gamma} 
The following holds for all $p\in B^{\circ}$.  
        \begin{enumerate}
        \item \label{item:prop:C-on-X''-good-position-Gamma:1} 
        The natural map of graphs $\Gamma(C_p)\to  \Gamma^1(Y_p)= \Gamma^1(X''_p)$ is an isomorphism on vertices, surjective on edges and does not contract any edge; in other words, the map identifies certain multiple edges to a single edge.
        \item \label{item:prop:C-on-X''-good-position-Gamma:2} 
        Let $\hat \Gamma(C_p)$ be the refinement of $\Gamma(C_p)$ induced via $\Gamma(C_p)\to \Gamma^1(X''_p)$ from the refinement $\hat \Gamma^1(X''_p)$ of $\Gamma^1(X''_p)$.
        Then $\hat \Gamma(C_p)$ carries a canonical $(I,S)$-bicoloring such that the natural map  
            \begin{align} \label{eq:hat-Gamma-c-to-hat-Gamma-X''_p}
            \hat \Gamma(C_p)\longrightarrow \hat \Gamma^1(X''_p) 
            \end{align}
        is a map of $(I,S)$-bicolored graphs, which is an isomorphism on the set of vertices, surjective on the set of edges and does not contract any edge.
        \item \label{item:prop:C-on-X''-good-position-Gamma:3} 
        For $s\in S$, let $i_s\in I$ be the index with $\tau(D_{i_s})=H_s$ from item \eqref{item:lem:reduction:2}  in Lemma \ref{lem:reductions/set-up}. 
        Then the following holds for $p\in B^\circ \cap D_{i_s}$: 
        \begin{enumerate}
            \item Any edge $e$ of $\hat \Gamma(C_p)$ of bicolor $(i_s,s)$ is an edge of $\Gamma(C_p)$ and, moreover, $C\to B^\circ$ is locally at the corresponding node of $C_p$ given by the product of the normal form $\{u_{i_s}=x_{s}y_{s}\}$ with a smooth fibration, where $u_{i_s}$ is a local equation for $D_{i_s}$.\label{item:prop:C-on-X''-good-position-Gamma:3a} 
            \item The graph $\hat \Gamma(C_p)$ has no edge of bicolor $(i_s,t)$ with $t\neq s$. \label{item:prop:C-on-X''-good-position-Gamma:3b} 
            \item The 
            composition $\hat \Gamma(C_p)\to \hat \Gamma^1(X''_p)\stackrel{\varphi_p}\to \Gamma^1(X'_p)$ is a map of $S$-colored graphs which does not contract any edge of bicolor $(i_s,s)$ and such that any edge of $\Gamma^1(X'_p)$ of color $s$ is the image of some edge of $\hat \Gamma(C_p)$ of bicolor $(i_s,s)$.\label{item:prop:C-on-X''-good-position-Gamma:3c} 
        \end{enumerate} 
\end{enumerate} 
\end{proposition}
\begin{proof}
Item \eqref{item:prop:C-on-X''-good-position-Gamma:1} follows from item \eqref{item:lem:C-on-X''-in-good-position:1} in Lemma \ref{lem:C-on-X''-in-good-position}, because $\Gamma^1(Y_p)=\Gamma^1(X''_p)$ for all $p$, since $Y\to X''$ is a smooth morphism (it is a projective bundle).
Item \eqref{item:prop:C-on-X''-good-position-Gamma:2} is a straightforward consequence of item \eqref{item:prop:C-on-X''-good-position-Gamma:1}.

It remains to prove   item \eqref{item:prop:C-on-X''-good-position-Gamma:3}.
By item  \eqref{item:lem:C-on-X''-in-good-position:1} in Lemma \ref{lem:C-on-X''-in-good-position}, $C_{B^\circ}\subset Y_{B^\circ}$ is a local complete intersection such that the singularities of $C_{B^\circ}\to B^\circ$ are given by transverse intersections of the singularities of $Y_{B^\circ}\to B^\circ$.
Moreover, this intersection is in general position in the sense that the fibers of $C_{B^\circ}\to B^\circ$ meets the double locus of $Y_{B^\circ}\to B^\circ$ transversely but misses any deeper stratum, cf.\ Lemma \ref{lem:Bertini}.
Using this, the claim in item \eqref{item:prop:C-on-X''-good-position-Gamma:3} follows from  item \eqref{item:lem:reduction:5} in Lemma \ref{lem:reductions/set-up},  together with the fact that $Y\to X''$ is a smooth morphism and so $\Gamma^1(Y_b)=\Gamma^1(X''_b)$ for all $b\in B^\circ$.  
\end{proof}

\subsection{A tree of compatible paths} \label{subsec:tree}
Consider $P\subset B^\circ$ from Definition \ref{def:P-and-B^circ}.
Recall that $P\subset \tau^{-1}(0)$ and $\tau^{-1}(0)_{\mathrm{red}}$ is an snc divisor on $B$ by assumptions.   
Let 
\begin{align}
 \psi_0 \colon  \mathbb T\longhookrightarrow \tau^{-1}(0)
\end{align} 
be a topological embedding of a tree (i.e.\ of a finite connected graph with trivial first homology), 
such that:
\begin{itemize}
    \item $\psi_0$ induces a bijection between the vertices of $\mathbb T$ and the set $P$;
    \item $\im(\psi_0)\subset B^\circ$;
    \item the interior of each edge $e$ of $\mathbb T$ is embedded into some open stratum of $D$, i.e.\ $\psi_0(e\setminus \partial e)\subset D_J^\circ$ for some $J\subset I$.
\end{itemize} 
Note that the normal bundle of $D$ restricted to the graph $\psi_0(\mathbb T)$ is trivial, as the latter is contractible.
It follows that we can extend $\psi_0$ to a topological embedding
$$
\psi\colon\mathbb T\times [0,1]\longhookrightarrow B^\circ,
$$
such that
\begin{itemize}
\item $\psi_0$ agrees with the restriction of $\psi$ to $\mathbb T\times \{0\}$;
    \item $\psi(\mathbb T\times (0,1])\subset B^\circ\setminus D$;
    \item for each $x\in \mathbb T$, $\psi(x\times [0,1])$ can be covered by an embedded polydisc $\Delta^S\subset B^\circ$ centered at $\psi_0(x)$.
\end{itemize}
We denote the restriction of $\psi$ to $\mathbb T\times \{1\}$ by  
\begin{align}
   \psi_{\nb}\colon \mathbb T \longhookrightarrow B^\circ\setminus D .
\end{align} 
For $x\in \mathbb T$ with $b:=\psi_0(x)$, we call 
$$
t_b:=\psi_\nb(x)
$$
the {\it point nearby} to $b$.
Moreover, $\psi(x,-)$ yields a path from $\psi_\nb(x)$ to $\psi_0(x)$.

Pick further a base point $\bprimezero\in B'$ which lies above the base point $\bzero \in (\Delta^\star)^S\subset B^\star$ and a path $\gamma$ which connects $t $ to a vertex $v_0$ of $\psi_{\nb}(\mathbb T)$.
 For any point $b\in \im(\psi)$, we then pick a path that connects $t$ to $b$ by first traveling along $\gamma$ to the point $v_0\in \psi_{\nb}(\mathbb T)$, then traveling from $v_0$ to $b$ along any path in $\im(\psi)$, which hits $\im(\psi_0)$ at most at the end point $b$ if $b\in \im(\psi_0)$. 
Note that any two such choices of paths are homotopic to each other, because $\mathbb T$ is contractible.
For this reason we may and will in what follows choose convenient paths as above whenever necessary.

For $p\in P$, let $v_p\in \mathbb T$ be the vertex with $\psi_0(v_p)=p$.
We let further $t_p\coloneq \psi_{\nb}(v_p)$ be the point nearby to $p$.
With the above set-up we then get for all $p\in P$ canonical identifications
\begin{align} \label{eq:iso:tree:X''}
\operatorname{gr}^{W^p}_0H_1(X''_{\bprimezero},\Z) \stackrel{\cong} \longrightarrow \operatorname{gr}^{W^p}_0H_1(X''_{t_p},\Z)\cong H_1(\Gamma(X''_p),\Z),
\end{align}
where $W^p_\bullet $ denotes the weight filtration on $H_1(X''_t,\Z)$ (resp.\ $H_1(X''_{t_p},\Z)$) induced by the local monodromies
at $p$ and the given path from $t$ to $p$ (resp.\ $t_p$ to $p$).
The second isomorphism is detailed for instance in \cite[Proposition 5.12]{survey}; this uses the aforementioned embedded disc $\Delta^S\subset B$, centered at $p$ and containing the point $t_p$. 
Similarly, we get canonical isomorphisms 
\begin{align} \label{eq:iso:tree:C}
\operatorname{gr}^{W^p}_0H_1(C_{\bprimezero},\Z) \stackrel{\cong} \longrightarrow \operatorname{gr}^{W^p}_0H_1(C_{t_p},\Z) \cong H_1(\Gamma(C_p),\Z).
\end{align}
Let $e$ be an edge of 
$\psi_0(\mathbb T)$ with end points $p,q$.
There is some subset $J\subset I$ such that 
the interior of $e$ and one of its end points, 
say $p$, lies in the open stratum $D_{J}^\circ$.
Then $p$ specializes to $q$ 
along the path given by the edge $e$ of 
$\psi_0(\mathbb T)$, as in Definition \ref{def:specialization}.

\begin{lemma} \label{lem:diagram-induced-by-tree-commutes}
Let $p,q\in P$ such that $p$ specializes to $q$ in the above sense.
Then there are natural commutative diagrams
\begin{align} \label{diag:gr^W=gr^W-specialization-C}
\begin{split}
\xymatrix{ 
\operatorname{gr}^{W^q}_0H_1(C_{\bprimezero},\Z)\ar[r]^-{\cong} \ar[d]  & H_1(\Gamma(C_q),\Z)\ar[d]^{\operatorname{sp}_\ast} \\
\operatorname{gr}^{W^p}_0H_1(C_{\bprimezero},\Z)\ar[r]^-{\cong}  & H_1(\Gamma(C_p),\Z) 
}
\end{split}
\end{align} 
and
\begin{align} \label{diag:gr^W=gr^W-specialization-X''}
\begin{split}
\xymatrix{ 
\operatorname{gr}^{W^q}_0H_1(X''_{\bprimezero},\Z)\ar[r]^-{\cong} \ar[d]  & H_1(\Gamma(X''_q),\Z)=H_1(\Gamma(X_0),\Z) \ar[d]^{=} \\
\operatorname{gr}^{W^p}_0H_1(X''_{\bprimezero},\Z)\ar[r]^-{\cong} & H_1(\Gamma(X''_p),\Z)= H_1(\Gamma(X_0),\Z)
} 
\end{split}
\end{align}
where all horizontal maps are isomorphisms. 
The vertical arrows on the left are the natural quotient maps, induced by the inclusion $W^q_{-1} \subset W^p_{-1}$.
The right most vertical arrow in \eqref{diag:gr^W=gr^W-specialization-X''} 
is induced by the fact that $\Gamma(X''_p)$ and $\Gamma(X''_q)$ are refinements of $\Gamma(X'_p)=\Gamma(X_0)$ and $\Gamma(X'_q)=\Gamma(X_0)$, respectively, see item \eqref{item:def-admissible:1} in Definition \ref{def:admissible}. 

Moreover, the natural map $C_{B^\circ}\to X''_{B^\circ}$ (see Lemma \ref{lem:C-on-X''-in-good-position}) induces canonical maps from \eqref{diag:gr^W=gr^W-specialization-C} to \eqref{diag:gr^W=gr^W-specialization-X''} such that the resulting cube of maps commutes.
\end{lemma}
\begin{proof} 
The horizontal isomorphisms in diagrams \eqref{diag:gr^W=gr^W-specialization-C} and \eqref{diag:gr^W=gr^W-specialization-X''} are the isomorphisms from \eqref{eq:iso:tree:C} and \eqref{eq:iso:tree:X''}, respectively, 
cf.~\cite[Proposition 5.12]{survey}. 
Since $p,q\in \tau^{-1}(0)$, we have $X'_p=X_0$ and $X'_q=X_0$; all maps in \eqref{diag:gr^W=gr^W-specialization-X''} are isomorphisms and the commutativity is clear.
To prove the commutativity of \eqref{diag:gr^W=gr^W-specialization-C}, recall from \cite[Propositions 5.11 and 5.12]{survey} that the second isomorphism in \eqref{diag:gr^W=gr^W-specialization-C} is induced by a natural map $H_1(C_{t_p},\Z)\to  H_1(\Gamma(C_{p}),\Z)$, which contracts all cycles that are invariant by the local monodromies at $p$.
In other words, the vanishing cycles as well as the cycles that specialize  in $H_1(C_p,\Z)$ to cycles supported on the smooth locus of $C_p$ are contracted.
Using this description together with the description of ${\rm sp}$ in Section \ref{subsec:specialization} and our choice of compatible paths, the commutativity of \eqref{diag:gr^W=gr^W-specialization-C} is easily verified.

Finally, the naturality of the maps in question implies that the 
canonical maps from \eqref{diag:gr^W=gr^W-specialization-C} to \eqref{diag:gr^W=gr^W-specialization-X''} that are induced by the natural map $C_{B^\circ}\to X''_{B^\circ}$ make the resulting cube of maps commutative.  
\end{proof}

\subsection{1-skeletal splitting} 
We have a natural surjective map 
\begin{align} \label{eq:1-skeletal-splitting-X_0-map}
H_1(\Gamma^1(X_0),\Lambda)\longrightarrow H_1(\Gamma(X_0),\Lambda),
\end{align}
which is induced by the inclusion of the $1$-skeleton $\Gamma^1(X_0)$ into the entire dual complex $\Gamma(X_0)$.
We will show in this subsection that our set-up induces a canonical splitting of the above map, which we refer to as {\it $1$-skeletal splitting}.

Recall from Definition \ref{def:matroidal-family} that we have specified an isomorphism
$$
U\cong {\rm gr}^W_0H_1(X_{\bzero },\Z)={\rm gr}^{W^p}_0H_1(X''_{\bprimezero},\Z) ,
$$
where $\bzero \in B$ is the base point in $(\Delta^\star)^S\subset B^\star$, $\bprimezero\in B'$ is a point with $\tau(\bprimezero)=\bzero $ and $W_\bullet^p$ denotes the weight filtration with respect to a point  $p\in P$ and ${\rm gr}^{W^p}_0H_1(X''_{\bprimezero},\Z)$ is independent of the choice of $p$ by Lemma \ref{lem:diagram-induced-by-tree-commutes}. 
By Lemma \ref{lem:diagram-induced-by-tree-commutes}, we also get a natural isomorphism 
\begin{align} \label{eq:UcongH1Gamma}
U\cong H_1(\Gamma(X_0),\Z) .
\end{align}
Consider the natural map $JC_{B^\circ}\to X''_{B^\circ}$, where $JC_{B^\circ}$ denotes the relative Jacobian of $C_{B^\circ}$ over $B^\circ$ and recall that $C_b\subset X''_b$ represents an $\ell$-prime multiple of the minimal class for $b\in B^\circ$ general, see Lemma \ref{lem:C-on-X''-in-good-position}.
Since $\ell$-prime integers are invertible in $\Lambda$, Lemma \ref{lem:orthogonal-Qs-preliminaries} implies that for each $p\in P$, the natural map
\begin{align}\label{splitting:C_t-surjection}
{\rm gr}^{W^p}_0H_1(C_{t},\Lambda) \longrightarrow {\rm gr}^{W^p}_0H_1(X''_{t},\Lambda),
\end{align}
induced by the map $C_t\to X''_t$, admits a canonical splitting, induced by the  map $X''_t\to JC_t$ that is dual to $JC_t\to X''_t$ with respect to the given principal polarizations (note that $X''_{\bprimezero}=X_{\bzero }$ is an abelian variety because $\bzero \notin H$ and $g$ is an isomorphism away from $D=\tau^{-1}(H)_{\red}$, see Definition \ref{def:admissible}).
We denote this canonical splitting of \eqref{splitting:C_t-surjection} by
\begin{align}\label{splitting:C_t-phi}
\phi^{W^p}_{C_{\bprimezero}} \colon U_\Lambda\longrightarrow {\rm gr}^{W^p}_0H_1(C_{\bprimezero},\Lambda) .
\end{align} 
Via the isomorphisms in Lemma \ref{lem:diagram-induced-by-tree-commutes}, this yields 
a map
\begin{align}\label{splitting:C_p-phi}
\phi_{C_p}\colon U_\Lambda\longrightarrow  H_1(\Gamma(C_{p}),\Lambda) ,
\end{align} 
uniquely determined by the topological embedding $\psi$, 
which splits the natural map
\begin{align}\label{splitting:C_p-surjection}
H_1(\Gamma(C_{p}),\Lambda) \longrightarrow H_1(\Gamma(X''_{p}),\Lambda)= H_1(\Gamma(X_{0}),\Lambda) = U_\Lambda.
\end{align}

\begin{lemma} \label{lem:monodromy->quadratic-forms}
Let $p\in P$ and $i\in I$ such that $p\in D_i$. 
Let $\hat Q_i$ be the monodromy bilinear form on $H_1(\Gamma(C_{p}),\Lambda)={\rm gr}^{W^p}_0H_1(C_{\bprimezero},\Lambda)$ associated to the monodromy about $D_i$ locally at $p$.
Then the following holds:
\begin{enumerate}
    \item the pullback $\phi_{C_p}^\ast \hat Q_i$ is the bilinear form on 
    $$
    U_\Lambda= {\rm gr}^{W}_0H_1(X_0,\Lambda)= {\rm gr}^{W}_0H_1(X'_p,\Lambda)={\rm gr}^{W^p}_0H_1(X''_{\bprimezero},\Lambda),
    $$ 
    which corresponds to $m$ times the monodromy bilinear form on $H_1(X''_{\bprimezero},\Lambda)$ associated to the monodromy about $D_i$ locally at $p$ (see Definition \ref{def:monodromy-bilinear-form}).
    (Here, $m$ is the integer such that $[C_t]=m\cdot [\Theta]^{g-1}/(g-1)!$, see item \eqref{item:lem:C-on-X''-in-good-position:2} in Lemma \ref{lem:C-on-X''-in-good-position}.) 
    \item the decomposition 
    $$
    H_1(\Gamma(C_{p}),\Lambda) = \phi_{C_p}(U_\Lambda)\oplus \ker(H_1(\Gamma(C_{p}),\Lambda) \to  H_1(\Gamma(X''_{p}),\Lambda))$$
    is orthogonal with respect to $\hat Q_i$.
\end{enumerate}
\end{lemma}
\begin{proof}
This follows from Lemma \ref{lem:orthogonal-Qs-preliminaries} by applying $- \otimes_{\Z_\ell} \Lambda$. 
\end{proof}

 The main result of this subsection is the following

\begin{proposition} \label{prop:1-skeletal-splitting}
The choice of the topological embedding $\psi$ from Section \ref{subsec:tree} induces a  canonical map
$$
\phi\colon U_\Lambda\longrightarrow H_1(\Gamma^1(X_0),\Lambda)
$$
which splits \eqref{eq:1-skeletal-splitting-X_0-map}, 
(cf.\ \eqref{eq:UcongH1Gamma}).
Moreover, for each $p\in P$, $\phi$ agrees with the following composition
\begin{align} \label{eq:1-skeletal-splitting-composition}
U_\Lambda\stackrel{\phi_{C_p}}\longrightarrow  H_1(\Gamma(C_{p}),\Lambda) \stackrel{\iota_\ast} \longrightarrow H_1(\Gamma^1(X''_{p}),\Lambda)  \longrightarrow H_1(\Gamma^1(X_0),\Lambda),
\end{align}
where $\iota_\ast$ is induced by the natural map $\iota \colon C_p\to X''_p$ from Lemma \ref{lem:C-on-X''-in-good-position} and where the last arrow is the canonical map
$$
H_1(\Gamma^1(X''_{p}),\Lambda) =H_1(\hat\Gamma^1(X''_{p}),\Lambda)   \stackrel{{\varphi_p}_\ast }\longrightarrow
H_1(\Gamma^1(X'_p),\Lambda)=H_1(\Gamma^1(X_0),\Lambda),
$$
induced by the data $\varphi$
of the admissible modification 
$X''\to X'$  (see Definition \ref{def:admissible}). 
\end{proposition}
\begin{proof}
For $p\in P$, let 
$$
\bar \phi_{C_p}\colon U_\Lambda\longrightarrow H_1(\Gamma^1(X_0),\Lambda)
$$ 
be the composition of the maps in \eqref{eq:1-skeletal-splitting-composition}.
To prove the proposition, we will show that this map defines a splitting which is independent of the choice of $p\in P$.

Let us first show that we get a splitting.
To prove this, consider the following diagram, in which homology is taken with coefficients in $\Lambda$:
$$
\xymatrix{ 
H_1(\Gamma(X_0))\ar[d]^-{\phi_{C_p}} \ar[r]^-{\phi_{C_p}}  
& H_1(\Gamma(C_p)) \ar[r]^-{\iota_\ast} & H_1(\Gamma^1(X''_p)) \ar[d] \ar[r]^{{\varphi_p}_\ast } & H_1(\Gamma^1(X'_p)) \ar[r]^= \ar[d] & H_1(\Gamma^1(X_0)) \ar[d] \\
H_1(\Gamma(C_p))\ar[r]^-{\iota_\ast}&H_1(\Gamma^1(X''_p)) \ar[r] & H_1(\Gamma(X''_p)) & H_1(\Gamma(X'_p)) \ar[r]^-= \ar[l]_-\cong & H_1(\Gamma(X_0)). 
}
$$
The two squares on the left and right commute for obvious reasons, and the square in the middle commutes because $X'' \to X'$ is an admissible modification, see item \eqref{item:def-admissible:1} in Definition \ref{def:admissible}. 
Hence the outer square commutes. 
Moreover, the composition $H_1(\Gamma(X_0)) \to H_1(\Gamma^1(X_0))$ on the top row of this diagram is, by definition, the map $\bar \phi_{C_p}$, and the composition $H_1(\Gamma(C_p)) \to H_1(\Gamma(X_0))$ on the bottom row is the map \eqref{splitting:C_p-surjection}. 
Since $\phi_{C_p}$ splits the map \eqref{splitting:C_p-surjection}, precomposing the bottom row with $\phi_{C_p}$ yields the identity. Hence, precomposing the map $H_1(\Gamma^1(X_0)) \to H_1(\Gamma(X_0))$ with the top row is the identity, so that $\overline{\phi}_{C_p}$ splits \eqref{eq:1-skeletal-splitting-X_0-map}, as claimed. 

It remains to show that this splitting is independent from the choice of $p\in P$.
Recall from Section \ref{subsec:tree} the construction of the tree $\psi_0(\mathbb T)\subset B^\circ \cap \tau^{-1}(0)$ whose set of vertices is $P$, and the choice of a tree of nearby points $\psi_{\nb}(\mathbb T) \subset B'\setminus D$.
We further fixed a path from $t$ to a vertex $v_0$ of $\psi_{\nb}(\mathbb T)$ and used it to connect $t$ to any other point of $\psi(\mathbb T \times [0,1])$ via the composition of $\gamma$ with a path in $\psi([0,1]\times \mathbb T)$, which is unique up to homotopy because the latter is contractible. 
We have used the corresponding paths to identify via Lemma \ref{lem:diagram-induced-by-tree-commutes} certain weight graded quotients of the homology above $t$ with the homology of dual complexes above the points $p\in P$.
Our choices have been made compatibly, in the sense that if $p,q\in P$ such that $p$ specializes to $q$, then the maps induced by specializing from $p$ to $q$ are compatible, see Lemma \ref{lem:diagram-induced-by-tree-commutes}.

Since $\mathbb T$ is a tree and $P$ is its set of vertices, it then suffices to prove that if we have points $p,q\in P$ such that $p$ specializes to $q$, the splitting induced by $\overline{\phi}_{C_p}$ agrees with the one induced by $\overline{\phi}_{C_q}$. 
Since $X''\to X'$ is admissible,  item \eqref{item:def-admissible:2} in Definition \ref{def:admissible}  implies that the following diagram commutes:
$$
\xymatrix{
H_1(\hat \Gamma^1(X''_p),\Lambda)=H_1(\Gamma^1(X''_p),\Lambda) \ar[d]^{{\varphi_p}_\ast} &\ar[l]_-{{\rm sp}_\ast} H_1(\hat \Gamma^1(X''_q),\Lambda)=H_1(\Gamma^1(X''_q),\Lambda) \ar[d]^{{\varphi_q}_\ast}\\
H_1(\Gamma^1(X'_p),\Lambda)=H_1(\Gamma^1(X_0),\Lambda)& \ar[l]_-{{\rm sp}_\ast}  H_1(\Gamma^1(X'_q),\Lambda)=H_1(\Gamma^1(X_0),\Lambda).
}
$$

Moreover, 
the composition $\mathrm{sp}_\ast \circ \phi_{C_q}$ of
the map $\phi_{C_q} \colon H_1(\Gamma(X_0),\Lambda) \to H_1(\Gamma(C_q),\Lambda)$ with the specialization map $\mathrm{sp}_\ast \colon H_1(\Gamma(C_q),\Lambda) \to H_1(\Gamma(C_p),\Lambda)$ is equal to the map $\phi_{C_p} \colon H_1(\Gamma(X_0),\Lambda) \to H_1(\Gamma(C_p),\Lambda)$. 
Together with the commutativity of the cubical
diagram given by the natural map from 
\eqref{diag:gr^W=gr^W-specialization-C} to \eqref{diag:gr^W=gr^W-specialization-X''} 
(see Lemma \ref{lem:diagram-induced-by-tree-commutes}), 
it follows that each square in the following diagram commutes, where again homology is taken 
with coefficients in $\Lambda$:
\[
\xymatrix{
H_1(\Gamma(X_0)) \ar[r]^-{\phi_{C_q}}\ar@{=}[d]& H_1(\Gamma(C_q)) \ar[d]^-{\mathrm{sp}_\ast}\ar[r] & H_1(\Gamma^1(X''_q))\ar[d]^-{\mathrm{sp}_\ast} \ar[r]^-{{\varphi_q}_\ast} & H_1(\Gamma^1(X'_q)) \ar[r]^-{=} \ar[d]^-{\mathrm{sp}_\ast} & H_1(\Gamma^1(X_0)) \ar@{=}[d] \\
H_1(\Gamma(X_0)) \ar[r]^-{\phi_{C_p}} & H_1(\Gamma(C_p)) \ar[r] & H_1(\Gamma^1(X''_p)) \ar[r]^-{{\varphi_p}_\ast} & H_1(\Gamma^1(X'_p)) \ar[r]^-{=} & H_1(\Gamma^1(X_0)). 
}
\]
As the composition on the top row equals $\bar \phi_{C_q}$ and the composition on the bottom row equals $\bar \phi_{C_p}$, 
we deduce that the splittings $\bar \phi_{C_q}$ and $\bar \phi_{C_p}$ agree with each other.
This concludes the proof of the proposition.
\end{proof}
 
\subsection{Edge multiplication}

For each $s\in S$, recall from item \eqref{item:lem:reduction:2} in Lemma \ref{lem:reductions/set-up} the component $D_{i_s}$ of $D$ with $\tau(D_{i_s})=H_s$ such that $\tau\colon B'\to B$ is \'etale at the generic point of $D_{i_s}$.
Since $P\subset \tau^{-1}(0)$ contains a point of each open stratum of $D$ contained in $\tau^{-1}(0)$, we can, for each $s\in S$, pick a point $p_s\in P$ such that $p_s\in D_{i_{s}}$. 

\begin{definition} \label{def:Q_s}
Let $s\in S$ and consider the refinement $\hat \Gamma(C_{p_s})$ of $\Gamma(C_{p_s})$ from Proposition \ref{prop:C-on-X''-good-position-Gamma}. 
The {\it monodromy bilinear form}  $\hat Q_s$ on $H_1(\hat \Gamma(C_{p_s}),\Lambda)=H_1(\Gamma(C_{p_s}),\Lambda)$ is the bilinear form that is induced by the local monodromy at $p_s$ around the divisor $D_{i_s}$, see Definition \ref{def:monodromy-bilinear-form}.
\end{definition}

\begin{lemma} \label{lem:Q_s}
Consider the refinement $\hat \Gamma(C_{p_s})$ of $\Gamma(C_{p_s})$ and let $E_{i_s,s}$ be the set of edges of $\hat \Gamma^1(C_{p_s})$ that have bicolor $(i_s,s)$ (see item \eqref{item:prop:C-on-X''-good-position-Gamma:2} in Proposition \ref{prop:C-on-X''-good-position-Gamma}).
Then each edge $e\in E_{i_s,s}$ is an edge of $\Gamma(C_{p_s})$ and we have 
$$
\hat Q_s=\sum_{e\in E_{i_s,s}} x_e^2.
$$
\end{lemma}

\begin{proof}
By item \eqref{item:prop:C-on-X''-good-position-Gamma:3a} in Proposition \ref{prop:C-on-X''-good-position-Gamma}, any edge of $\hat \Gamma(C_{p_s})$ of bicolor $(i_s,s)$ is an edge of $\Gamma(C_{p_s})$.  
Moreover, if $u_{i_s}$ denotes a local equation for $D_{i_s}$, the corresponding singularity of the map $C_{B^\circ}\to B^\circ$ is locally given by the product of $\{u_{i_s}=x_sy_s\}$ with a smooth fibration.
So by the Picard--Lefschetz formula, 
each edge $e$ of bicolor $(i_s,s)$ contributes to $Q_s$ with the summand $x_e^2$.
Moreover, by item \eqref{item:prop:C-on-X''-good-position-Gamma:3b} in Proposition \ref{prop:C-on-X''-good-position-Gamma},   $\hat \Gamma^1(C_{p_s})$ has no edges of bicolor $(i_s,t)$ with $t\in S\setminus \{s\}$.  
The claimed description of $\hat Q_s$ then follows from the fact that only
edges of $I$-color $i_s$ contribute to $\hat Q_s$, cf.\ Definition \ref{def:I-coloring}.
 \end{proof}

Consider the map of $S$-colored graphs
$$
\varphi_{s}\colon \hat \Gamma^1(C_{p_s})\longrightarrow \Gamma^1(X'_{p_s})= \Gamma^1(X_0)
$$
given as the composition
$$
 \hat \Gamma^1(C_{p_s})\longrightarrow \hat \Gamma^1(X''_{p_s})\stackrel{\varphi_{p_s}}\longrightarrow 
\Gamma^1(X'_{p_s}) = \Gamma^1(X_0),
$$
where the first map comes from item \eqref{item:prop:C-on-X''-good-position-Gamma:2} in Proposition \ref{prop:C-on-X''-good-position-Gamma}. 
By item \eqref{item:prop:C-on-X''-good-position-Gamma:3c} in Proposition \ref{prop:C-on-X''-good-position-Gamma}, any edge of $\Gamma^1(X_0)$ of color $s$ is the image of an edge of bicolor $(i_s,s)$.
Moreover, no edge of bicolor $(i_s,s)$ is contracted.
For any edge $e$ of $\Gamma^1(X_0)$ which has color $s$, we may then define $F(e)$  as the subset of edges of $E(\hat \Gamma^1(C_{p_s}))$, given by
\begin{align} \label{align:def:F(e)}
F(e)\coloneqq \{e'\in E(\hat\Gamma^1(C_{p_s}))\mid \varphi_s(e')=e\ \  \text{and $e'$ has bicolor $(i_s,s)$}\} .
\end{align}

\begin{remark}
The elements of $F(e)$ correspond bijectively to the nodes of the curve $C_{p_s}$ that do not smooth along the divisor $D_{i_s}$; this description follows from the fact that $\hat \Gamma^1(C_{p_s})$ has no nodes of bicolor $(i_s,t)$ with $t\neq s$ and any edge of bicolor $(i_s,s)$ of the refinement $\hat \Gamma^1(C_{p_s})$ is in fact an edge of $\Gamma^1(C_{p_s})$, see item \eqref{item:prop:C-on-X''-good-position-Gamma:3} in Proposition \ref{prop:C-on-X''-good-position-Gamma}.
\end{remark}

\begin{definition} \label{def:Gamma^1_s}
For $s\in S$, we define the $S$-colored graph $\Gamma^1_s(X_0)$ to be the unique $S$-colored graph with a morphism $\pi_s\colon \Gamma^1_s(X_0)\to \Gamma^1(X_0)$, which is an isomorphism on vertices and an isomorphism on edges of color $t\in S\setminus \{s\}$, such that in addition, for any $s$-colored edge $e$ of $\Gamma^1(X_0)$, we have an identification of
$
\pi_s^{-1}(e)
$
with the set $F(e)$.
\end{definition} 

By construction, the natural map of $S$-colored graphs
$
\varphi_s\colon \hat \Gamma(C_{p_s})\to \Gamma^1(X_0)
$ 
factors naturally through a map of $S$-colored graphs
$$
f_s\colon \hat \Gamma(C_{p_s})\longrightarrow \Gamma^1_s(X_0) ,
$$
such that $\varphi_s=\pi_s\circ f_s$.

\begin{definition} \label{def:diagonal-quadratic-Gamma1s}
For $s\in S$, let $E_s(\Gamma^1_s(X_0)) $ denote the set of edges of $\Gamma^1_s(X_0)$ of color $s$.
Then we define the \emph{diagonal bilinear form}
$$
Q_s\coloneqq \sum_{e\in E_s(\Gamma^1_s(X_0)) } x_e^2
$$
on $C_1(\Gamma^1_s(X_0),\Lambda)$ and we use the same symbol for its restriction to $H_1(\Gamma^1_s(X_0),\Lambda)$.
\end{definition}

\begin{lemma}\label{lem:f_s-iso-on-edges-of-color-(i_s,s)}
    The map $f_s\colon \hat \Gamma(C_{p_s})\to \Gamma^1_s(X_0)$ induces an isomorphism between the edges of bicolor $(i_s,s)$ of $\hat \Gamma(C_{p_s})$ and the set of edges of color $s$ of $\Gamma^1_s(X_0)$. 
    In particular, $Q_s$ pulls back to $\hat Q_s$ under ${f_s}_\ast\colon H_1(\hat \Gamma (C_{p_s}),\Lambda)\to H_1(\Gamma^1_s(X_0),\Lambda)$. 
\end{lemma}
\begin{proof}
    The fact that $f_s$ induces an isomorphism between the edges of bicolor $(i_s,s)$ of $\hat \Gamma(C_{p_s})$ and the set of edges of color $s$ of $\Gamma^1_s(X_0)$ follows directly from the construction.
    The second claim is an immediate consequence of this and the description of $\hat Q_s$ given in Lemma \ref{lem:Q_s}. 
\end{proof}

\begin{lemma} \label{lem:splitting-Gamma_s^1}
Let $s\in S$.
Then the natural composition
\begin{align} \label{eq:splitting-Gamma_s^1-surjection}
H_1(\Gamma_s^1(X_0),\Lambda) \stackrel{{\pi_s}_\ast}\longrightarrow H_1(\Gamma^1(X_0),\Lambda)\longrightarrow H_1(\Gamma(X_0),\Lambda)=U_\Lambda
\end{align}
admits a canonical splitting
$$
\phi_{s}\colon U_\Lambda\longrightarrow H_1(\Gamma_s^1(X_0),\Lambda) ,
$$
given by $\phi_s={f_s}_\ast\circ \phi_{C_{p_s}}$. 
Moreover:
\begin{enumerate}
    \item \label{item:lem:splitting-Gamma_s^1:1} This splitting is orthogonal with respect to the bilinear form $Q_s$ on $ H_1(\Gamma_s^1(X_0),\Lambda) $.
    \item \label{item:lem:splitting-Gamma_s^1:2} $\phi_s^\ast Q_s=m\cdot d\cdot y^2_s$, where $y_s^2$ is the rank one bilinear form associated to the realization $S\to U^\ast$, $s\mapsto y_s$. 
\end{enumerate}
\end{lemma}
\begin{proof}
Recall that $\phi_{C_{p_s}} \colon U_\Lambda\to H_1(\Gamma(C_{p_s}),\Lambda)$ splits the natural map 
$$
H_1(\hat \Gamma(C_{p_s}),\Lambda)=H_1(\Gamma(C_{p_s}),\Lambda)\longrightarrow H_1(\Gamma(X_0),\Lambda) ,
$$
see \eqref{splitting:C_p-phi} and \eqref{splitting:C_p-surjection}.
The latter map factors as the composition of ${f_s}_\ast$ with the composition in \eqref{eq:splitting-Gamma_s^1-surjection}.
The first result of the lemma follows.

Item \eqref{item:lem:splitting-Gamma_s^1:1} then follows from the following facts: we have $({f_s}_\ast)^\ast Q_s=\hat Q_s$  (see Lemma \ref{lem:f_s-iso-on-edges-of-color-(i_s,s)}), the splitting $\phi_{C_{p_s}}$ is orthogonal with respect to $\hat Q_s$ (see Lemma \ref{lem:monodromy->quadratic-forms}), and the summands in the decomposition of $H_1(\hat \Gamma(C_{p_s}), \Lambda)$ surject onto the corresponding summands in the decomposition of $H_1(\Gamma_s^1(X_0),\Lambda)$. 

To prove item \eqref{item:lem:splitting-Gamma_s^1:2}, note that $\phi_s^\ast Q_s=\phi_{C_{p_s}}^\ast \hat Q_s$, and $\phi_{C_{p_s}}^\ast \hat Q_s$ is by Lemma \ref{lem:monodromy->quadratic-forms} $m$ times the bilinear form on 
$$
U_\Lambda={\rm gr}^W_0 H_1(X_0,\Lambda)={\rm gr}^{W^{p_s}}_0 H_1(X_{t},\Lambda),
$$
associated to the local monodromy around the divisor $D_{i_s}$ in $B'$.
Since $\tau\colon B'\to B$ is \'etale at the generic point of $D_{i_s}$, this agrees via the first identity above with the local monodromy about the divisor $H_s$ in $B$. 
Since $X^\star\to B^\star$ is a matroidal family associated to $(\underline R,S)$ with integral realization $S\to U^\ast$, $s\mapsto y_s$,  item \eqref{item:def:matroidal-family:3} in Definition \ref{def:matroidal-family} together with \eqref{eq:d=d_s}
implies that this bilinear form equals $d \cdot y^2_s$. 
Hence, we have $\phi_s^\ast Q_s=\phi_{C_{p_s}}^\ast \hat Q_s = m\cdot d \cdot y^2_s$, which proves item \eqref{item:lem:splitting-Gamma_s^1:2}.
\end{proof}

\subsection{Construction of $G$ and proof of Theorem \ref{thm:algebraic->d-QE}}

\begin{definition} \label{def:G}
We let $G$ be the $S$-colored graph, given as the fiber product of the graphs $\Gamma^1_s(X_0)$, $s\in S$, over the graph $\Gamma^1(X_0)$. 
\end{definition} 

\begin{remark}\label{rem:Gamma^1_s}
The graph $G$ can explicitly be described as the unique $S$-colored graph together with morphisms of $S$-colored graphs 
$$
\pr_s\colon G\to \Gamma^1_s(X_0)\quad \quad \text{and}\quad \quad g_s=\pi_s\circ \pr_s\colon G\to \Gamma^1(X_0),
$$
which are isomorphisms on the set of vertices, such that, in addition, for any edge $e$ of $\Gamma^1(X_0)$ of color $s$, the set $g_s^{-1}(e)$ is canonically identified with the set $F(e)$ defined in \eqref{align:def:F(e)}, consisting of edges of $\hat \Gamma(C_{p_s})$ of bicolor $(i_s,s)$ which map to $e$ via $\varphi_s$.
\end{remark}

For $s\in S$, consider the natural composition
\begin{align} \label{eq:surjection-splittin-G}
H_1(G,\Lambda)\stackrel{{\pr_s}_\ast}\longrightarrow H_1(\Gamma^1_s(X_0),\Lambda)\stackrel{{\pi_s}_\ast}\longrightarrow H_1(\Gamma^1(X_0),\Lambda)\longrightarrow H_1(\Gamma(X_0),\Lambda) 
\end{align}
and note that this composition does not depend on $s$.

\begin{proposition} \label{prop:splitting-phi_G}
There is a canonical splitting
$$
\phi_{G}\colon U_\Lambda\longrightarrow H_1(G,\Lambda)
$$
of \eqref{eq:surjection-splittin-G} such that for all $s\in S$ we have ${\pr_s}_\ast \circ \phi_G=\phi_s$, where $\phi_s$ is the splitting from \eqref{lem:splitting-Gamma_s^1}.
\end{proposition}
\begin{proof}
In order to define $\phi_G$, we will first define a map
$$
\phi_G\colon U_\Lambda\longrightarrow C_1(G,\Lambda)
$$
and then show that the image of this map lies in $H_1(G,\Lambda)$.

Let $\alpha\in U_\Lambda$.
We then define 
$\phi_G(\alpha)\in C_1(G,\Lambda)$ 
as the sum
$$
\phi_G(\alpha)\coloneqq \sum_{s\in S} c_s(\alpha),
$$
where $c_s(\alpha)$ denotes the linear combination of oriented edges of color $s$, given by the $s$-colored part of $\phi_s(\alpha) \in H_1(\Gamma^1_s(X_0),\Lambda)$.  
This uses that there is a canonical isomorphism between the edges of color $s$ of $G$ and those of $\Gamma^1_s(X_0)$.
In order to show that the above sum $\sum_{s\in S} c_s(\alpha)$ is closed, we first note that its pushforward to $C_1(\Gamma^1(X_0),\Lambda)$ equals
\begin{align} \label{align:sum_spi_spr_sc_s}
\sum_{s\in S} {\pi_s}_\ast ({\pr_s}_\ast(c_s(\alpha)))\in C_1(\Gamma^1(X_0),\Lambda). 
\end{align}
We claim that the class \eqref{align:sum_spi_spr_sc_s} agrees with the $1$-cycle $\phi(\alpha)$, where $\phi$ is the splitting from Proposition \ref{prop:1-skeletal-splitting}.
To prove this, it suffices, for each $s\in S$, to check that the $s$-colored part of this $1$-cycle agrees with the $s$-colored part of $\phi(\alpha)$, cf.\ Section \ref{subsec:convention:graphs}.
But the $s$-colored part of the above sum is nothing but ${\pi_s}_\ast ({\pr_s}_\ast(c_s(\alpha)))$. 
Hence, it suffices to show, for the $s$-colored part $c_s(\alpha) \in C_1(\Gamma^1_s(X_0),\Lambda)$ of the element $\phi_s(\alpha) \in H_1(\Gamma^1_s(X_0),\Lambda)$, that ${\pi_s}_\ast(c_s(\alpha)) \in C_1(\Gamma^1(X_0),\Lambda)$ equals the $s$-colored part of $\phi(\alpha) \in H_1(\Gamma^1(X_0),\Lambda)$. 
For this, it suffices in turn to show that ${\pi_s}_\ast(\phi_s(\alpha)) = \phi(\alpha) \in H_1(\Gamma^1(X_0),\Lambda)$, 
which follows from 
$$
{\pi_s}_\ast \circ \phi_s = {\pi_s}_\ast \circ {f_s}_\ast \circ \phi_{C_{p_s}} = {\varphi_s}_\ast \circ \phi_{C_{p_s}} = \phi .
$$  
We conclude that
$$
\sum_{s\in S} {\pi_s}_\ast ({\pr_s}_\ast(c_s(\alpha)))=\phi(\alpha), 
$$
hence $\sum_{s\in S} {\pi_s}_\ast ({\pr_s}_\ast(c_s(\alpha)))$ is closed.
Since $\pi_s\circ \pr_s\colon G\to \Gamma^1(X_0)$ is an isomorphism on vertices, we then deduce that $\sum_{s\in S}c_s(\alpha)$ must be closed, because its pushforward to $\Gamma^1(X_0)$ is closed.
Hence,
$$
\phi_G(\alpha)\coloneqq \sum_{s\in S} c_s(\alpha) \in H_1(G,\Lambda)
$$
is well-defined.
By construction, $c_s(\alpha)$ depends linearly on $\alpha$ and so $\phi_G$ is linear.
The composition of $\phi_G$ with the natural map 
$$ 
H_1(G,\Lambda)\longrightarrow H_1(\Gamma^1(X_0),\Lambda) 
$$
agrees with $\phi$ from Proposition \ref{prop:1-skeletal-splitting}.
Hence, $\phi_G$ is a splitting of \eqref{eq:surjection-splittin-G}, as desired.
The equality ${\pr_s}_\ast \circ \phi_G=\phi_s$ follows easily from the construction. 
\end{proof}

For $s\in S$, let $E_s$ denote the set of edges of $G$ of color $s$.
By construction, the edges of color $s$ of $G$ are canonically isomorphic to the edges of $\Gamma^1_s(X_0)$ of color $s$.
By slight abuse of notation, we then denote by 
$$
Q_s=\sum_{e\in E_s} x_e^2
$$ 
the diagonal bilinear form on $H_1(G,\Lambda)$ and note that it descends to the bilinear form on $H_1(\Gamma^1_s(X_0),\Lambda)$ from Definition \ref{def:Q_s} that we denote by the same symbol.

\begin{proposition}\label{prop:orthogonal-Q_s}
Let $s\in S$.
The splitting $\phi_G$ of \eqref{eq:surjection-splittin-G} is orthogonal with respect to $Q_s$. 
Moreover, the pullback $\phi_G^\ast Q_s$ agrees with $m\cdot d\cdot y_s^2$, where $y_s\in U_\Lambda^\ast $ is the linear form induced by the integral realization $S\to U^\ast$, $s\mapsto y_s$.
\end{proposition}
\begin{proof} 
Let $\alpha\in U_\Lambda$ and 
$\beta\in \ker(H_1(G,\Lambda)\to H_1(\Gamma(X_0),\Lambda))$.
We will show that $Q_s$, viewed as a bilinear form, satisfies $Q_s(\phi_G(\alpha),\beta)=0$.
Since $Q_s$ is supported on edges of color $s$ and the natural map $\pr_s\colon G\to \Gamma^1_s(X_0)$ is a morphism of $S$-colored graphs which is an isomorphism 
on edges of color $s$, we have
$$
Q_s(\phi_G(\alpha),\beta)=Q_s({\pr_s}_\ast(\phi_G(\alpha)),{\pr_s}_\ast(\beta)).
$$
Since $\phi_s=\pr_s\circ \phi_G$ (see Proposition \ref{prop:splitting-phi_G}), we get ${\pr_s}_\ast(\phi_G(\alpha)) \in \phi_s(U_\Lambda) $; as
$$
{\pr_s}_\ast(\beta) \in \ker(H_1(\Gamma^1_s(X_0),\Lambda) \to H_1(\Gamma(X_0),\Lambda)=U_\Lambda),
$$
the above right hand side vanishes by Lemma \ref{lem:splitting-Gamma_s^1}.
Hence, the splitting $\phi_G$ is orthogonal with respect to $Q_s$.

Let now $\alpha,\beta\in U_\Lambda$.
Then, by the same reasoning as above, we have
$$
Q_s(\phi_G(\alpha),\phi_G(\beta))=Q_s({\pr_s}_\ast(\phi_G(\alpha)),{\pr_s}_\ast(\phi_G(\beta)))=Q_s(\phi_s(\alpha),\phi_s(\beta)).
$$
The above right hand side agrees with $m\cdot d\cdot y_s(\alpha)y_s(\beta)$, see Lemma \ref{lem:splitting-Gamma_s^1}.
Hence,  $\phi_G^\ast Q_s=m\cdot d\cdot y^2_s$, as we want.
This concludes the proof of the proposition.
\end{proof}

\begin{proof}[Proof of Theorem \ref{thm:algebraic->d-QE}] 
We replace the positive integer $d$ by the positive multiple $md$.
By Propositions \ref{prop:splitting-phi_G} and \ref{prop:orthogonal-Q_s}, there is a decomposition
$$
H_1(G,\Lambda)=U_\Lambda\oplus U'
$$
which is orthogonal with respect to $Q_s$ and such that $Q_s$ restricts to $d\cdot y^2_s$ on $U_\Lambda$, for all $s\in S$.
This says that there is a quadratic $\Lambda$-splitting of level $d$ of $(\underline R,S)$ into the graph $G$, see Definition \ref{def:d-Lambda-splitting-graph}.
This concludes the proof of the theorem. 
\end{proof}

\section{From quadratic splittings to solutions in Albanese graphs}
\label{sec:splitting->solution}
Throughout this section, $\ell$ denotes a prime number. 

\subsection{Albanese graphs}

\begin{definition} \label{def:color-profile-map}
Let $S$ be a finite set and let $G$ be an $S$-colored oriented graph.
Let $\Lambda$ be a ring.
Then the \emph{color profile map} (with coefficients in $\Lambda$) is the unique linear map 
$$
\lambda \colon C_1(G,\Lambda)\longrightarrow \Lambda^S
$$
that sends an edge of color $s$, viewed as $1$-chain via the given orientation, to the $s$-th basis vector $e_s$ of $\Lambda^S$.

Assume in addition that $G$ is $\Lambda$-weighted, i.e.\ for each edge $e$ of $G$ we are given a weight $a(e)\in \Lambda$.
Then the \emph{weighted color profile map} is the unique linear map 
$$
\lambda^w\colon C_1(G,\Lambda)\longrightarrow \Lambda^S
$$
that sends an oriented edge $e$ of color $s$ to the element $a(e)\cdot e_s\in \Lambda^S$.
\end{definition}

\begin{remark}
For a graph $G$ with set of edges $E$, the choice of an orientation of $G$, i.e.\ an orientation for each edge, determines a canonical isomorphism $C_1(G,\Lambda)=\Lambda^E$.
The color profile map $\lambda$ can then explicitly be described as the composition
$$
C_1(G,\Lambda)=\Lambda^E=\bigoplus_{s\in S}\Lambda^{E_s} \stackrel{\Sigma^S}\longrightarrow \Lambda^S,
$$
where $\Sigma^S$ is induced by the sum maps $\Sigma_s\colon \Lambda^{E_s}\to \Lambda$, which maps a vector to the sum of its coordinates.
\end{remark}

Let $(\underline R,S)$ be a regular matroid with integral realization $S\to U^\ast$, $s\mapsto y_s$.
Dually, there is an inclusion 
$U\hookrightarrow \Z^S$, which identifies $U_{\Lambda}=U\otimes_{\Z} \Lambda$ 
with a subspace of $\Lambda^S$.

A crucial ingredient of the remaining sections 
is the following set of universal graphs 
associated to $(\underline R,S)$.

\begin{definition} \label{def:Alb}
Let $ \uR= (\underline R,S)$ be a regular 
matroid with integral realization $S\to U^\ast$ 
and let $j\leq r$ be non-negative integers. 
The \emph{$(\ell^r,\ell^j)$-Albanese graph} of 
 $\underline R$  is the oriented $S$-colored graph 
$\Alb_{\ell^r,\ell^j}(\underline R)$ whose set of 
vertices is given by 
$$
V=\Z^S/(\ell^{j} U+\ell^r\Z^S) ,
$$
and such that two vertices $[v],[w]\in V$ are joined by an oriented edge of color $s$ pointing from $[v]$ to $[w]$ if and only if $[w]=[v+e_s]$, where $e_s\in \Z^S$ denotes the $s$-th basis vector. 
\end{definition}

We will also write $\Alb_{\ell^r}(\underline R)\coloneq \Alb_{\ell^r,1}(\underline R)$ and in particular $\Alb_\ell(\underline R)\coloneq \Alb_{\ell,1}(\underline R)$.

\begin{remark} \label{rem:Alb_{1,1}}
The special case where $j=r=0$ leads to the graph $\Alb_{1,1}(\underline R)$, which has only one vertex and for each $s\in S$ a loop of color $s$.
That is, $\Alb_{1,1}(\underline R)$ is a wedge of $|S|$-many circles, one for each color $s\in S$.
\end{remark}

\begin{remark} \label{rem:ell=2-multiple-edges}
    For $\ell=2$, the Albanese graph $\Alb_{2}(\underline R)$ may have vertices $[v]$ and $[w]$ that are joined by multiple edges of the same color, but with reversed orientations.
    Such edges occur if $[v]=[v+2e_s]$ and $[w]=[v+e_s]\neq [v]$.
\end{remark}

\begin{lemma} \label{lem:Alb-well-defined}
The $S$-colored graph $\Alb_{\ell^r,\ell^j}(\underline R)$ does not depend on the choice of integral realization $S\to U^\ast$.
Moreover, its orientation is well-defined up to possibly reversing the orientation of all edges of some colors $s\in S'\subset S$ simultaneously.
\end{lemma}
\begin{proof}
Fix a basis of $\underline R$.
The chosen integral realization of $\underline R$ then corresponds to a matrix $(\mathds 1_g|D)\in \mathbb Z^{g\times |S|}$, where $g$ denotes the rank of $\underline R$.
By Lemma \ref{lem:unique-integer-realization}, another choice of integral realization of $\underline R$ leads (with respect to the same basis of $\uR$) to a  matrix $(\mathds 1_g|D')$ which may be obtained from $(\mathds 1_g|D)$ by multiplying some rows and columns by $-1$.
Multiplying a row by $-1$ does not change the $S$-colored graph $\Alb_{\ell ^r,\ell ^j}(\underline R)$ nor its orientation, while multiplying the $s$-th column of $(\mathds 1_g|D)$ by $-1$ does not change the $S$-colored graph, but reverses the orientation of all $s$-colored edges of the Albanese graph. 
The lemma follows.
\end{proof}

Assume now that $\Lambda$ is a $\Z_{(\ell)}$-algebra with $\Lambda/\ell^r=\Z/\ell^r$.
Then the set $V$ of vertices of $\Alb_{\ell^r,\ell^j}(\underline R)$ satisfies:
\begin{align} \label{eq:V}
V=\Z^S/(\ell^jU+\ell^r\Z^S)=\Lambda^S/(\ell^jU_\Lambda+\ell^r\Lambda). 
\end{align} 
The natural commutative diagram
\[
\xymatrix{
C_1(\Alb_{\ell^r,\ell^j}(\underline R), \Lambda)  \ar[d]_{\partial} \ar[r]^-{\lambda} & \Lambda^S \ar@{->>}[d] \\
C_0(\Alb_{\ell^r,\ell^j}(\underline R),\Lambda) \ar[r] & V=\Lambda^S/(\ell^jU_\Lambda+\ell^r\Lambda),
}
\] 
in which the lower horizontal map  sends 
$(a_v)_v \in \Lambda^V=
C_0(\Alb_{\ell^r,\ell^j}(\underline R),\Lambda) $ 
to $\sum_v a_v \cdot v \in V$, 
shows that $\im( \lambda\colon H_1(\Alb_{\ell^r,\ell^j}(\underline R),\Lambda)\to \Lambda^S)\subset \ell^jU_\Lambda+\ell^r \Lambda^S$.
The connected oriented $S$-colored graph $\Alb_{\ell^r,\ell^j}(\underline R)$ is universal for this property:

\begin{proposition}[Universal property] \label{prop:universal-property-Alb}
Let $G$ be a connected, oriented, $S$-colored graph.
Let $\Lambda$ be a $\Z_{(\ell)}$-algebra with $\Lambda/\ell^r=\Z/\ell^r$.
Assume that 
$$
\im( \lambda\colon H_1(G,\Lambda)\to \Lambda^S)\subset \ell^jU_\Lambda+\ell^r \Lambda^S.
$$ 
Then, for any $j\leq r$, the choice of a vertex $v_0$ of $G$ defines a canonical map of oriented $S$-colored graphs
$$
{\rm alb}\colon G\longrightarrow \Alb_{\ell^r,\ell^j}(\underline R) ,
$$
which does not contract any edge.
Explicitly, $\alb$ maps a vertex $v$ of $G$ to the element $[\lambda(\gamma_{v})] \in V$, where $\gamma_v\in C_1(G,\Lambda)$ is a $1$-chain with $\partial \gamma_v= v-v_0$. 
\end{proposition}
\begin{proof}
Let us first show that the map $\alb$ is well-defined on the set of vertices.
Since $G$ is connected, we can find a $1$-chain $\gamma_v$ with $\partial \gamma_v=v-v_0$ for any vertex $v$ of $G$.
If $\gamma_v'$ is another choice of such a $1$-chain, then $\gamma_v-\gamma_{v}'$ is closed and so
$$
\lambda(\gamma_v-\gamma_{v}')\in  \ell^jU_\Lambda + \ell^r \Lambda ^S
$$
by assumptions.
We deduce that 
$$
[\lambda(\gamma_v)]= [\lambda(\gamma_{v}')] \in V,
$$
where we used \eqref{eq:V}.
This proves that the map $\alb$ is well-defined on the set of vertices.

To prove that the map of vertices extends to a map of $S$-colored oriented graphs, let $e$ be an edge of $G$ of color $s$ satisfying $\partial e=v_2-v_1$ (with respect to the given orientation of $e$).
Let further $\gamma_{v_1}$ be a $1$-chain with $\partial \gamma_{v_1}=v_1-v_0$.
Then $\gamma_{v_1}+e$ is a $1$-chain with boundary $v_2-v_0$.
This shows $[\alb (v_1)+e_s]=[\alb (v_2)]$.
Hence, $\alb(v_1)$ and $\alb(v_2)$ are joined by a unique oriented edge of color $s$ which points from $\alb(v_1)$ to $\alb(v_2)$, and we may send $e$ to this edge.
Then $\alb$ respects the given orientations and $S$-colorings and hence defines a canonical map of oriented $S$-colored graphs,
which does not contract any edge, as claimed. 
\end{proof}

\subsection{$\ell^{j+1}$-indivisible $\Lambda$-solutions in $\Alb_{\ell^r,\ell^j}(\underline R)$}
Recall from Section \ref{subsec:convention:graphs}, that a $1$-chain $b\in C_1(G,\Lambda)$ of an $S$-colored graph $G$ is of color $s\in S$, if it is a linear combination of oriented edges of color $s$.
Note that this is strictly stronger than asking that the color profile of $b$ is a multiple of $e_s$.

\begin{definition} \label{def:solution}
Let $(\underline R,S)$ be a regular matroid with integral realization $S\to U^\ast$.
Let $G$ be an $S$-colored oriented graph and let $\Lambda$ be a $\Z_{(\ell)}$-algebra.
A {\it $\Lambda$-solution of $(\underline R,S)$ in $G$} (or simply: a $\Lambda$-solution) is a collection of $1$-chains $(b_s)_{s\in S}$, $b_s\in C_1(G,\Lambda)$, of $G$ with coefficients in $\Lambda$ such that the following hold: 
\begin{enumerate}
\item \label{item:def:solution-0} The $1$-chain $b_s$ has color $s$;
    \item \label{item:def:solution-1} Consider the inclusion  $U_\Lambda\hookrightarrow \Lambda^S$ induced by the dual of the given realization. 
    If  $\sum_{s\in S} c_se_s\in U_\Lambda$ for some $c_s\in \Lambda$, 
    then $\sum_{s\in S} c_sb_s\in H_1( G,\Lambda)$ is closed. 
\end{enumerate}
For $i,r \geq 0$, we say that a $\Lambda$-solution  $(b_s)_{s\in S}$ is 
\begin{itemize}
    \item  {\it $\ell^{i}$-indivisible} if,  for all $s\in S$, the color profile of $b_s$ is not $\ell^i$-divisible: $\ell^{i} \nmid \lambda(b_s)$,
    \item {\it constant modulo $\ell^r$} if  $\lambda(b_s)\equiv \lambda(b_t) \bmod \ell^r$ for all $s,t\in S$.
\end{itemize}
\end{definition}
 
\begin{remark} \label{remark:constancy}
The notion of being constant modulo $\ell^r$ is not needed for the main results of this paper, the reason being roughly that we manage to reduce those results to calculating kernels of matrices over $\Z/2$ attached to relatively small and manageable matroids like $M(K_5)$ and $M(K_{3,3})$, see Section \ref{thm:splitting-in-cographic=cographic-intro} and Proposition \ref{prop:K_5-K_3,3}. However, in our subsequent work \cite{engel2025optimalityprymtyurinconstructionmathcala6} we need to deal with $\ell = 3$ and $\underline R = M(K_7)$, where the computations become involved and the simplification of working only with $3$-indivisible solutions which are constant modulo $3$---whose existence is easier to exclude---is essential.
\end{remark} 

\begin{remark} \label{rem:solutions-well-defined}
Item \eqref{item:def:solution-1} can be reformulated as follows.
Let $E=\bigsqcup_{s\in S}E_s$ be the set of edges of $G$ and consider the linear map $\Psi:\Lambda^S\to \Lambda^E=C_1(G,\Lambda)$, which sends the basis vector $e_s$ to the $1$-chain $b_s$.
Then the restriction of $\Psi$ to $U_\Lambda$ lands in $H_1(G,\Lambda)$ and so there is a unique map $U_\Lambda\to H_1(G,\Lambda)$ which makes the following diagram commutative:
$$
\xymatrix{
U_\Lambda \ar[d]^{\exists !} \ar@{^{(}->}[r]& \Lambda^S \ar[d]^\Psi\\
H_1(G,\Lambda)\ar@{^{(}->}[r]& \Lambda^E.
} 
$$
Using this description, it easily follows from 
Lemma \ref{lem:unique-integer-realization} that the existence of 
an $\ell ^i$-indivisible $\Lambda$-solution of 
$(\underline R,S)$ in an $S$-colored oriented graph $G$ does not depend on the 
chosen integral realization. 
Indeed, the choice of a basis of $\underline R$ 
yields an isomorphism $U_\Lambda\cong \Lambda^g$,
where $g$ denotes the rank of $\underline R$, and passing to another integral realization changes the matrix that represents the top horizontal map in the given basis by multiplying some rows and columns by $-1$.
Likewise, the existence of $\ell^i$-indivisible $\Lambda$-solutions which are constant modulo $\ell^r$, of $(\underline R,S)$ in a graph $G$, does not depend on the realization.
\end{remark}

The main result of this section is the following theorem.

\begin{theorem} \label{thm:d-QE->solutions}
Let $(\underline R,S)$ be a regular loopless matroid with integral realization $S\to U^\ast$.
Let $\ell$ be a prime, $r$ a positive integer and $\Lambda$ be a $\Z_{(\ell)}$-algebra with $\Lambda\subset \R$ and $\Lambda/\ell^r=\Z/\ell^r$.
Assume that there is a positive integer $d$ and a graph $G$, such that $(\underline R,S)$ admits a quadratic $\Lambda$-splitting of level $d$ in $G$ (see Definition \ref{def:d-Lambda-splitting-graph}).
Write $d=\ell^jd'$ for an integer $d'$ coprime to $\ell$.
Then, for all $j+1\leq r$, the $(\ell^{r},\ell^j)$-Albanese graph $\Alb_{\ell^{r},\ell^j}(\underline R)$ admits an $\ell^{j+1}$-indivisible $\Lambda$-solution of $(\underline R,S)$ which is constant modulo $\ell^r$ (see Definitions \ref{def:Alb} and \ref{def:solution}). 
\end{theorem}

The rest of this section is devoted to a proof of the above theorem; to this end, $\Lambda$ will, for the remainder of this section, denote a $\Z_{(\ell)}$-algebra with $\Lambda\subset \R$ and $\Lambda/\ell^r=\Z/\ell^r$.
(There is no harm in assuming $\Lambda=\Z_{(\ell)}$, which suffices for our applications, but we prefer to make the precise assumptions, needed in our argument, transparent.)
Recall also that the data of a quadratic $\Lambda$-splitting of $\underline R$ in $G$ comes in particular with an $S$-coloring of $G$ and so we may from now on view $G$ as an $S$-colored graph.

As a first reduction step, we reduce to the case where $G$ is connected.
Indeed, if it is not, then we pick a vertex in each connected component and glue these vertices to get a connected $S$-colored graph $G'$.
We then replace $G$ by $G'$.
This does not change the cohomology of $G$ and hence does not change the fact that $(\underline R,S)$ admits a quadratic $\Lambda$-splitting of level $d$ in $G$.
So we henceforth assume that $G$ is connected.

\subsection{Characteristic $1$-chain of color $s$}   
Denote the set of edges of $G$ by $E$ and the set of edges of color $s$ by $E_s$.
Since $G$ is $S$-colored, $E=\sqcup_{s\in S} E_s$. 
We further fix once and for all an orientation on the set of edges of $G$.
This induces a canonical embedding  
$$
H_1(G,\Lambda)\subset \Lambda^E=C_1(G,\Lambda).
$$

\begin{proposition} \label{prop:multiple-y-x} 
    Consider an edge $e\in E_s$ of color $s$ of $G$ and the corresponding coordinate function $x_e \colon \Lambda^E\to \Lambda$.
    Then the linear form $x_e|_{U_\Lambda}\in U_\Lambda^\ast$, given as the composition
    $$
    U_\Lambda \longhookrightarrow H_1(G,\Lambda)\subset \Lambda^E\longrightarrow \Lambda,\ \ \ u\mapsto x_e(u) ,
    $$
    is a multiple of $y_s$, where $y_s$ are the linear forms of the integral 
    realization $S\to U^\ast$, $s\mapsto y_s$, 
    of $(\underline R,S)$.
\end{proposition} 

\begin{proof}
We can extend the element $s\in S$ to a basis $S'\subset S$ of the matroid $(\underline R,S)$.
Assume without loss of generality that $s=1$ and 
$S'=\{1,\dots ,g\}$. 
Then the linear forms $y_1,\dots ,y_g$ form a basis of $U^\ast$. 
It follows that there are unique elements $a_{ei}\in \Lambda$ with
$$
x_e|_{U_\Lambda}=\sum_{i=1}^g a_{ei}y_i .  
$$
Since $U_\Lambda\hookrightarrow H_1(G,\Lambda)$ forms part of the data of a quadratic splitting of $(\underline R,S)$ in $G$ of level $d$, we have
$$
d\cdot y_1^2=\sum_{e\in E_s} x_e^2|_U ,
$$
see Definition \ref{def:d-Lambda-splitting-graph} 
(and recall that $s=1$).
We thus get 
$$
d\cdot y_1^2=\sum_{e\in E_s} x_e^2|_U=\sum_{e\in E_s}\left(\sum_{i=1}^g a_{ei}y_i \right)^2 .
$$

The above equality is an equality of quadratic forms on $U_\Lambda$; since $\Lambda$ has characteristic zero and $y_1,\dots ,y_g$ is a basis of $U_\Lambda^\ast$, we see that the equality holds in fact in the polynomial ring $\Lambda[y_1,\dots ,y_g]$.
Since $\Lambda\subset \R$, we then deduce from the elementary Lemma \ref{lem:elementary} below that $a_{ei}=0$ for all $i\geq 2$ and all $e\in E_s$. 
This proves that $x_e|_U$ is a multiple of $y_s$.
\end{proof}

\begin{lemma} \label{lem:elementary}
Let $\Lambda$ be a subring of $\R$ and consider the polynomial ring $\Lambda[y_1,\dots ,y_g]$ in $g$ variables.
Let $(a_{ji})\in \Lambda^{m\times g}$ be an $m\times g$ matrix and assume there is some $d\in \Lambda$ with
$$
d\cdot y_1^2=\sum_{j=1}^m \left( \sum_{i=1}^g a_{ji}y_i \right)^2.
$$
Then $a_{ji}=0$ for all $i\geq 2$ and $j \geq 1$.
\end{lemma}
\begin{proof}
We have
$$
\sum_{j=1}^m \left( \sum_{i=1}^g a_{ji}y_i \right)^2= \sum_{i=1}^g \left(\sum_{j=1}^m a_{ji}^2\right)y_i^2 +\sum_{i_1<i_2} c_{i_1,i_2}y_{i_1}y_{i_2}
$$
for some $c_{i_1,i_2}\in \Lambda$.
Since the above expression agrees with $d\cdot y_1^2$ by assumption, we find that the mixed terms $c_{i_1,i_2}y_{i_1}y_{i_2}$, $i_1<i_2$, vanish and, in addition, $\sum_{j=1}^m a_{ji}^2=0$ for all $i\geq 2$.
Since $\Lambda\subset \R$, this is only possible if $a_{ji}=0$ for all $i\geq 2$, as claimed.
\end{proof}

By Proposition \ref{prop:multiple-y-x}, 
for each $s\in S$ and each edge $i\in E_s$
of $G$ of color $s$, 
there is a unique element $a(i)\in \Lambda$ such that
\begin{align} \label{def:a(i)}
x_i|_{U_\Lambda}=a(i)y_s\in U^\ast_\Lambda.
\end{align}
The coefficients $a(i)$ play an important 
role and lead in particular to the following notion:

\begin{definition} \label{def:b_s-characteristic-1-chain}
The {\em characteristic} $1$-chain of color $s$ is the $1$-chain
$$
b_s\coloneqq \sum_{i\in E_s}a(i)\cdot i\in C_1(G,\Lambda) .
$$
\end{definition}

Associated to the characteristic $1$-chains $b_s$ of color $s$, there is the linear map
$$
\Psi:\Lambda^S \longrightarrow \Lambda^E=\bigoplus_{s\in S} \Lambda^{E_s},\quad \quad \sum_{s\in S} c_s e_s\mapsto \sum_{s\in S} c_s b_s ,
$$
which identifies to the direct sum of the maps $\Lambda\to \Lambda^{E_s}$, $1\mapsto b_s$.
We may then consider the following diagram
\begin{align} \label{diag:existence-of-solutions}
\begin{split}
\xymatrix{
U_\Lambda\ar@{^{(}->}[d] \ar@{^{(}->}[r] & \Lambda^S \ar[d]^{\Psi}  \\
H_1(G,\Lambda ) \ar@{^{(}->}[r] & \Lambda^E=\bigoplus_{s\in S} \Lambda^{E_s}
} 
\end{split}
\end{align} 
where the top horizontal arrow is induced by the dual of the realization $S\to U^\ast$ of $\underline R$, the lower horizontal arrow is induced by the choice of orientation of $G$, and the left vertical map is induced by the fact that there is a $\Lambda$-splitting of $(\underline R,S)$ in $G$ of level $d$. 

\begin{proposition} \label{prop:existence-of-solutions}
The diagram \eqref{diag:existence-of-solutions} commutes.
In other words, if $u=\sum_s c_se_s\in U_\Lambda \subset \Lambda^S$, then the $1$-chain $\sum_s c_s b_s\in C_1(G,\Lambda) = \Lambda^E$ is closed and agrees in fact with $u\in U_\Lambda$ via the embedding $U_\Lambda\subset H_1(G,\Lambda)$.
\end{proposition}

\begin{proof}
Let $u\in U_\Lambda$.
We first think about $U_\Lambda$ as a subspace of $\Lambda^S$ and write 
$$
u=\sum_s c_se_s\in U_\Lambda  \subset \Lambda^S 
$$
for some $c_s\in \Lambda$.
This means that $y_s(u)=c_s$ for all $s\in S$.   

Now view $U_\Lambda$ as a subspace of $H_1(G,\Lambda)$.
By \eqref{def:a(i)}, we have for all $i\in E_s$ the identity
    $$
    x_i|_{U_\Lambda}=a(i)y_s \in U^\ast_\Lambda .
    $$
Hence, for all $s\in S$ and $i\in E_s$,  we find
$$
x_i(u)= a(i)\cdot y_s(u)= a(i)\cdot c_s \in \Lambda .
$$
This in turn precisely means that
$$
u=\sum_{s\in S}\sum_{i\in E_s} c_s\cdot a(i)\cdot  i  = \sum_{s\in S} c_s\cdot b_s \in U_\Lambda .
$$
This proves the proposition.
\end{proof}

\subsection{Weighted color profile}
In this section, we use the elements $a(i)\in \Lambda$ from \eqref{def:a(i)} to define a $\Lambda$-weighting on the oriented $S$-colored graph $G$ as follows.

\begin{definition} \label{def:G-weight}
We define the \emph{weight} of an edge 
$i \in E_s$ to be the element $a(i)\in \Lambda$, 
where $a(i)$ is as in \eqref{def:a(i)}. 
\end{definition}

Associated to this weighting of $G$, there is a weighted color profile map
\begin{align} \label{def:lambda^w}
\lambda^w\colon C_1(G,\Lambda)\longrightarrow \Lambda^S ,
\end{align}
which is the unique $\Lambda$-linear map which sends an edge $i\in E_s$ of color $s$ to $a(i)\cdot e_s$ (see also Definition \ref{def:color-profile-map}).
The $s$-th component of this map is denoted by
$$
\lambda_s^w\colon C_1(G,\Lambda)\longrightarrow \Lambda .
$$
We then have $\lambda^w(x)=\sum_s \lambda^w_s(x)e_s$ for $x \in C_1(G,\Lambda)$.

Recall, moreover, the $Q_s$-orthogonal decomposition
$$
H_1(G,\Lambda)=U_\Lambda \oplus U' ,
$$
induced by the quadratic $\Lambda$-splitting 
of $(\underline R,S)$ in $G$, 
see Definition \ref{def:d-Lambda-splitting-graph}.

\begin{proposition} \label{prop:U'=colorless}
The weighted color profile map $\lambda^w$ from \eqref{def:lambda^w} has the following properties:
\begin{enumerate}
    \item \label{item:prop:U'=colorless:1} If $u'\in U'$, then $\lambda^w(u')=0$.
    \item \label{item:prop:U'=colorless:2} The characteristic $1$-cycle $b_s$ of color $s$ (see Definition \ref{def:b_s-characteristic-1-chain}) satisfies $\lambda_s^w(b_s)=d$. 
    \item \label{item:prop:U'=colorless:3} The image of the restriction of $\lambda^w$ to $H_1(G,\Lambda)$ agrees with $d\cdot U_\Lambda\subset \Lambda^S$.
\end{enumerate} 
\end{proposition}
\begin{proof} 
By assumption in Theorem \ref{thm:d-QE->solutions}, $(\underline R,S)$ is loopless.
By Lemma \ref{lem:integral-realization}, this means that for each $s\in S$, there  exists some $u_s\in U_\Lambda$ and
elements $c_t\in \Lambda$, $t\in S\setminus s$, with
\begin{align} \label{def:u_s}
u_s= e_s + \sum_{t\in S\setminus s} c_te_t\in U_\Lambda.
\end{align} 
If we view $u_s$ as an element of $U_\Lambda\subset H_1(G,\Lambda)$, then this implies by Proposition \ref{prop:existence-of-solutions} that the $s$-colored part of the closed $1$-chain $u_s$ agrees with the characteristic $1$-chain $b_s$, see Definition \ref{def:b_s-characteristic-1-chain}. 

Consider the bilinear form 
$Q_s=\sum_{i\in E_s}x_i^2$ 
on $C_1(G,\Lambda)$.
We then have
$$
Q_s(u_s,\alpha )=Q_s(b_s,\alpha) \quad \quad \text{for all $\alpha\in C_1(G,\Lambda)$,}
$$
because $Q_s$ is supported on the edges of color $s$.
Moreover, if we write $\alpha=\sum_{i\in E} \alpha_i \cdot i$ as a linear combination of edges, then, by linearity, and because $Q_s=\sum_{i\in E_s}x_i^2$, we get
$$
Q_s(b_s,\alpha)=\sum _{i\in E_s}\alpha_i Q_s(b_s,i)=\sum _{i\in E_s}\alpha_i Q_s(a(i)i,i)=\sum_{i\in E_s}a(i)\cdot \alpha_i.
$$
The above right hand side agrees with $\lambda_s^w(\alpha)$.
We have thus shown that
\begin{align}\label{eq:Q_s=lambda^w}
Q_s(u_s,\alpha )=\lambda_s^w(\alpha )\quad \quad \text{for all $\alpha \in C_1(G,\Lambda)$} .
\end{align}

Since $u_s\in U_\Lambda$ and the decomposition $H_1(G,\Lambda)=U_\Lambda\oplus U'$ is $Q_s$-orthogonal, \eqref{eq:Q_s=lambda^w} implies that 
$$
\lambda^w(u')=Q_s(u_s,u')=0\quad \quad \text{for all $u'\in U'$.}
$$
This proves item \eqref{item:prop:U'=colorless:1} in the proposition.

To prove item \eqref{item:prop:U'=colorless:2}, 
note that $ \lambda^w_s(b_s)=\lambda^w_s(u_s) $, because the $s$-colored part of the closed $1$-chain $u_s$ is given by $b_s$.
Moreover, $\lambda^w_s(u_s)=Q_s(u_s,u_s)$, by \eqref{eq:Q_s=lambda^w}.
Finally, $Q_s(u_s,u_s)=d \cdot y_s(u_s)^2$, because  $Q_s|_{U_\Lambda}=d\cdot y_s^2$.
By \eqref{def:u_s}, $y_s(u_s)=1$ and so 
$$
\lambda^w_s(b_s)=d\cdot y_s(u_s)^2=d.
$$
 This proves item \eqref{item:prop:U'=colorless:2}.

Finally, to prove item \eqref{item:prop:U'=colorless:3}, recall the commutative diagram \eqref{diag:existence-of-solutions}, see Proposition \ref{prop:existence-of-solutions}.
By item \eqref{item:prop:U'=colorless:1}, proven above, and the decomposition $H_1(G,\Lambda)=U_\Lambda\oplus U'$, the image of $\lambda^w$ agrees with the image of its restriction to $U_\Lambda$.
Let now $u\in U_\Lambda$, view $U_\Lambda$ as a subspace of $\Lambda^S$, and write $u=\sum c'_se_s$ for some $c'_s\in \Lambda$.
By the commutative diagram \eqref{diag:existence-of-solutions}, we then have
$$
\lambda^w(u)=\lambda^w\left(\sum_{s\in S} c'_sb_s\right)=\sum_{s\in S}c'_s\lambda_s^w(b_s)\cdot e_s .
$$
By item \eqref{item:prop:U'=colorless:2}, proven above, $\lambda_s^w(b_s)=d$ and we get
$$
\lambda^w(u)=\sum_{s\in S}c'_sd\cdot e_s=d\sum_{s\in S} c'_se_s=d\cdot u\in \Lambda^S.
$$
This proves item \eqref{item:prop:U'=colorless:3} and hence concludes the proof of the proposition.
\end{proof}

\subsection{Refinement of $G$ and its Albanese image}
In this section we finish the proof of Theorem \ref{thm:d-QE->solutions}.
To this end, we write $d=\ell^jd'$ for an integer $d'$ coprime to $\ell$ and we fix an integer $r\geq j+1$.
Our goal is to show that the $(\ell^{r},\ell^j)$-Albanese graph $\Alb_{\ell^{r},\ell^j}(\underline R)$ admits an $\ell^{j+1}$-indivisible $\Lambda$-solution (which is constant modulo $\ell^r$), see Definitions \ref{def:Alb} and \ref{def:solution}.

To begin with, for each edge $i$ of $G$, 
we let $\hat a(i)\in \Z$ denote the smallest positive integer with the property that
$$
\hat a(i)\equiv a(i) \bmod \ell^r,
$$
where $a(i)\in \Lambda$ are the elements from \eqref{def:a(i)}. 
Note that $\hat a(i)$ exists because $\Lambda/\ell^r=\Z/\ell^r$ by assumption.

\begin{definition}
We denote by $\hat G$ the oriented $S$-colored graph that is given by the refinement of $G$, where each edge $i\in E_s$ of color $s$ of $G$ is replaced by a chain of $\hat a(i)$-many consecutive edges of the same color and same orientation.
\end{definition}

We define 
$$
\Xi:C_1(G,\Lambda)\longrightarrow C_1(\hat G,\Lambda)
$$
as the unique linear map which sends an edge $i\in E_s$ of $G$ of color $s$ to the chain of $\hat a(i)$-many edges of $\hat G$ of the same color and same orientation that corresponds to $i$ in the refinement $\hat G$ of $G$. 
Note that $\Xi$ induces the canonical isomorphism
$$
\Xi|_{H_1(G,\Lambda)}\colon H_1(G,\Lambda)\stackrel{\cong}\longrightarrow H_1(\hat G,\Lambda)
$$
induced by the homeomorphism $G\stackrel{\approx}\to \hat G$ given by the fact that $\hat G$ is a refinement of $G$.

Note that $\hat G$ is, by definition, unweighted.
In fact, the refinement is chosen in such a way that the weighted color profile of $G$ compares to the (unweighted) color profile map
$$
\lambda:C_1(\hat G,\Lambda)\longrightarrow \Lambda^S
$$
of $\hat G$, as follows:

\begin{lemma} \label{lem:lambda^w-vs-lambda}
The following diagram commutes modulo $\ell^r$:
$$
\xymatrix{
C_1(G,\Lambda)\ar[d]_{\Xi} \ar[rd]^-{\lambda^w}&\\
C_1(\hat G,\Lambda)\ar[r]_-{\lambda}&\Lambda^S.
}
$$ 
\end{lemma}
\begin{proof}
By linearity, it suffices to  check that $(\lambda \circ \Xi)(i) \equiv \lambda^w(i) \bmod \ell^r$ for  a single edge $i\in E_s$ of color $s$ of $G$.
Then $\Xi(i)$ is a chain of $\hat a(i)$-many edges of the same color and same orientation.
Hence, $\lambda(\Xi(i))=\hat a(i)\cdot e_s$.
On the other hand, $\lambda^w(i)=a(i)\cdot e_s$.
The result thus follows from the fact that  $\hat a(i)\equiv a(i) \bmod \ell^r$.
\end{proof}

\begin{proof}[Proof of Theorem \ref{thm:d-QE->solutions}]
By Lemma \ref{lem:lambda^w-vs-lambda}, item \eqref{item:prop:U'=colorless:3} in Proposition \ref{prop:U'=colorless} implies that 
$$
\im( \lambda \colon H_1(\hat G,\Lambda)\longrightarrow \Lambda^S)\equiv \ell^j U_\Lambda \mod \ell^r \Lambda^S,
$$
where we used that $d=\ell^jd'$, where $d'$ is coprime to $\ell$ and hence invertible in $\Lambda$. 
Thus,
\begin{align}\label{align:im-lambda-Ghat}
\im( \lambda \colon H_1(\hat G,\Lambda)\longrightarrow \Lambda^S)\subset \ell^j U_\Lambda + \ell^r \Lambda^S.
\end{align}
Recall that $G$ is connected by the first reduction step in the proof of Theorem \ref{thm:d-QE->solutions} (see the paragraph after the theorem).
Hence, $\hat G$ is connected as well.  
By Proposition \ref{prop:universal-property-Alb}, \eqref{align:im-lambda-Ghat} implies that the choice of a vertex of $\hat G$ defines a canonical morphism of oriented $S$-colored graphs
$$
\alb \colon \hat G\longrightarrow \Alb_{\ell^r,\ell^j}(\underline R) .
$$ 
Recall the characteristic $1$-chain $b_s\in C_1(G,\Lambda)$ from Definition \ref{def:b_s-characteristic-1-chain} and define 
$$
\hat b_s\coloneqq \Xi(b_s)\in C_1(\hat G,\lambda) .
$$ 
(By the definition of $\Xi$, $\hat b_s$ is the $1$-chain on $\hat G$ that is obtained from $b_s$ by replacing each edge $i\in E_s$ of $G$ by the corresponding chain of $\hat a(i)$-many edges of $\hat G$ of the same color and same orientation.)

Let now $c_s\in \Lambda$, for $s \in S$, be such that $\sum_{s\in S} c_se_s\in U_\Lambda\subset \Lambda^S$.
By Proposition \ref{prop:existence-of-solutions}, $\sum_{s\in S} c_sb_s\in C_1(G,\Lambda)$ is closed.
Since $\Xi$ respects closedness, we find that
$$\textstyle 
\Xi(\sum_{s\in S}c_sb_s)=\sum_{s\in S}c_s\hat b_s\in H_1(\hat G,\Lambda) .
$$
This shows that $(\hat b_s)_{s \in S}$ is a $\Lambda$-solution of $(\underline R,S)$ in $\hat G$, see Definition \ref{def:solution}.
By Lemma \ref{lem:lambda^w-vs-lambda}, we have
$
\lambda(\hat b_s)\equiv \lambda^w(b_s) \bmod \ell^r
$.
By item \eqref{item:prop:U'=colorless:2} in Proposition \ref{prop:U'=colorless}, we deduce that 
\begin{align} \label{eq:lambda(hat b_s)=constant}
\lambda(\hat b_s)\equiv de_s \bmod \ell^r.
\end{align}
Since $d=\ell^jd'$ with $d'$ coprime to $\ell$, and since $r \geq j+1$ by assumption, we see that the $\Lambda$-solution $(\hat b_s)_{s\in S}$ of $(\underline R,S)$ in $\hat G$ is in fact $\ell^{j+1}$-indivisible. Because the Albanese morphism $\alb$ is a morphism of oriented $S$-colored graphs, 
the collection of $1$-chains 
$$
\alb_\ast \hat b_s\in C_1(\Alb_{\ell^r,\ell^j}(\underline R),\Lambda) 
$$ 
with $s\in S$ is then a $\Lambda$-solution of $(\underline R,S)$ in $\Alb_{\ell^r,\ell^j}(\underline R)$ which 
is still $\ell^{j+1}$-indivisible, because $\alb$ does not contract any edge. 
For the same reason, \eqref{eq:lambda(hat b_s)=constant} implies that the solution is constant modulo $\ell^r$ in the sense of Definition \ref{def:solution}.
This concludes the proof of Theorem \ref{thm:d-QE->solutions}.
\end{proof}

\section{From $\ell^{j}$-indivisible solutions to $\ell$-indivisible solutions} 
\label{sec:l^j-indivisible->l-indivisible}

The main result of this section is the following $\ell$-power descent
theorem for solutions in Albanese graphs, improving 
Theorem \ref{thm:d-QE->solutions} from the previous section.

\begin{theorem}\label{thm:reduction-to-2-solution}
Let $(\underline R,S)$ be a regular loopless matroid with integral realization $S\to U^\ast$.
Let $j\leq i\leq r$ be non-negative integers, $\ell$ a prime and let $\Lambda$ be an $\Z_{(\ell)}$-algebra.
Assume that $(\underline R,S)$ has an $\ell^i$-indivisible $\Lambda$-solution $(b_s)_{s\in S}$ in $\Alb_{\ell^{r},\ell ^j}(\underline R)$. 
Then $(\underline R,S)$ has an $\ell^{i-j}$-indivisible $\Lambda$-solution $(b'_s)_{s\in S}$ in $\Alb_{\ell^{r-j},1}(\underline R)$. 
If, moreover, $(b_s)_{s\in S}$ is constant modulo $\ell^r$, then we can arrange that $(b'_s)_{s\in S}$ is constant modulo $\ell^{r-j}$.
\end{theorem}

\subsection{Albanese tori}

\begin{definition} \label{def:Albanese-torus}
Let $(\underline R,S)$ be a regular matroid with integral realization $S\to U^\ast$, and let $j\leq r$ be non-negative integers.
Then the \emph{$(\ell^r,\ell^j)$-Albanese torus} of $(\underline R,S)$ is the real 
$|S|$-dimensional torus
$$
\mathcal T_{\ell^r,\ell^j}(\underline R)\coloneqq \R^S/(\ell^jU+\ell^r\Z^S)  .
$$ 
\end{definition}

\begin{remark} \label{rem:Alb-torus-well-defined}
It follows from Lemma \ref{lem:unique-integer-realization} that $\mathcal T_{\ell^r,\ell^j}(\underline R)$ depends only on $(\underline R,S)$ and not on the chosen integral realization $S\to U^\ast$, cf.\ proof of Lemma \ref{lem:Alb-well-defined}.
\end{remark}

\begin{remark} \label{rem:Alb-torus-polyhedral}
It will frequently be convenient to endow $\mathcal T_{\ell^r,\ell^j}(\underline R)$ with the structure of a polyhedral complex, induced by the unit cube tiling of $\R^S$.
For instance, using this polyhedral structure, the Albanese graph $\Alb_{\ell^r,\ell^j}(\underline R)$ from Definition \ref{def:Alb} is easily identified to the $1$-skeleton of $\mathcal T_{\ell^r,\ell^j}(\underline R)$. 
\end{remark}

The remainder of Section \ref{sec:l^j-indivisible->l-indivisible} is devoted to a proof of Theorem \ref{thm:reduction-to-2-solution}.
To this end we fix a regular loopless matroid $(\underline R,S)$  with integral realization $S\to U^\ast$. 
We then write for simplicity 
$$
\Alb_{\ell^r,\ell^j}\coloneqq \Alb_{\ell^r,\ell^j}(\underline R)\ \ \text{and}\ \ \mathcal T_{\ell^r,\ell^j}\coloneqq \mathcal T_{\ell^r,\ell^j}(\underline R).
$$
By Remark \ref{rem:Alb-torus-polyhedral}, there is a natural inclusion of polyhedral complexes 
$$
\Alb_{\ell^r,\ell^j} \longhookrightarrow \mathcal T_{\ell^r,\ell^j} .
$$
We consider the canonical topological covering map
$$
\xi:\mathcal T_{\ell^r,\ell^j}=\R^S/(\ell^jU+\ell^r\Z^S)\longrightarrow \mathcal T_{\ell^{j},\ell^j}=\R^S/\ell^j\Z^S ,
$$
given by quotiening out $\ell^j\Z^S/\ell^jU + \ell^r\Z^r$. 
This map induces topological coverings of the respective Albanese graphs, that we denote by the same symbol:
$$
\xi \colon \Alb_{\ell^r,\ell^j}\longrightarrow \Alb_{\ell^j,\ell^j}.
$$

\subsection{A refinement of the Albanese graph and the Albanese torus}
From now on we fix the non-negative integers $j\leq r$ from Theorem \ref{thm:reduction-to-2-solution}.

\begin{definition} \label{def:hat-Alb}
For non-negative integers $a\leq b $, let $\widehat \Alb_{\ell^{b},\ell^a}$ be the $S$-colored graph obtained from $\Alb_{\ell^{b},\ell^a}$ by replacing each edge by a chain of $\ell^j$ consecutive edges of the same color and the same orientation.

Similarly, we denote by $\widehat {\mathcal T}_{\ell^b,\ell^a}$ the polyhedral structure on $\mathbb R^S/(\ell^a U+\ell^b \mathbb Z^S)$, induced by the tiling of $\R^S$ with cubes of side length $1/\ell^j$. 
\end{definition}

In view of the polyhedral structure on  $\mathcal T_{\ell^b,\ell^a}=\R^S/(\ell^a U+\ell^b \mathbb Z^S)$ whose 1-skeleton identifies canonically to  $\Alb_{\ell^{b},\ell^a}$, see Remark \ref{rem:Alb-torus-polyhedral}, we see that the 1-skeleton of $\widehat {\mathcal T}_{\ell^b,\ell^a}$ contains canonically the graph $\widehat \Alb_{\ell^{b},\ell^a}$.

\begin{lemma} \label{lem:Hat-Alb-Hat-mathcalT}
Let $a\leq b$ be non-negative integers.
Then multiplication by $\ell^j$ induces a canonical identification of polyhedral complexes
$$
\widehat {\mathcal T}_{\ell^b,\ell^a} \stackrel{\cong} \longrightarrow  \mathcal T_{\ell^{b+j},\ell^{a+j}}.
$$
Precomposing this map with the inclusion of $\widehat \Alb_{\ell^b,\ell^a}$ into the $1$-skeleton of $\widehat {\mathcal T}_{\ell^b,\ell^a} $ induces a natural embedding of oriented $S$-colored graphs
$$
\iota \colon \widehat \Alb_{\ell^b,\ell^a}\longhookrightarrow \Alb_{\ell^{b+j},\ell^{a+j}}
$$
\end{lemma}
\begin{proof}
The first claim is clear. The second one follows because the $1$-skeleton of $\mathcal T_{\ell^{b+j},\ell^{a+j}}$ is canonically isomorphic to $\Alb_{\ell^{b+j},\ell^{a+j}}$, see Remark \ref{rem:Alb-torus-polyhedral}. 
\end{proof}

\begin{lemma}\label{lem:Hat-Alb-solutions}
Assume that $(\underline R,S)$ has an $\ell^i$-indivisible $\Lambda$-solution in 
$\widehat{\Alb}_{\ell^{r-j},1}$ (which is constant modulo $\ell^r$).
Then  $(\underline R,S)$ has an $\ell^{i-j}$-indivisible $\Lambda$-solution in 
$\Alb_{\ell^{r-j},1}$ (which is constant modulo $\ell^{r-j}$).
\end{lemma} 

\begin{proof}
Recall that $\widehat{\Alb}_{\ell^{r-j},1}$ is obtained from $\Alb_{\ell^{r-j},1}$ by dividing each edge $e$ of $\Alb_{\ell^{r-j},1}$ into a chain of $\ell^j$ consecutive edges of the same orientation and the same color.
Since $(\underline R,S)$ is loopless, 
the natural linear map $C_1(\Alb_{\ell^{r-j},1},\Lambda) \to C_1(\widehat \Alb_{\ell^{r-j},1},\Lambda)$, that sends an edge to the sum of the consecutive edges that divide it, induces a bijection between the spaces of $\Lambda$-solutions.
The lemma follows from this. 
\end{proof}

In order to prove Theorem \ref{thm:reduction-to-2-solution}, our goal is to produce solutions of $\Alb_{\ell ^r,\ell ^j}$ that are supported on the image of $\widehat{\Alb}_{\ell^{r-j},1}$ and to apply Lemma \ref{lem:Hat-Alb-solutions} to conclude.
The idea is to push solutions from $\Alb_{\ell ^r,\ell ^j}$ to $\widehat{\Alb}_{\ell^{r-j},1}$, viewed as a subgraph of  $\Alb_{\ell ^r,\ell ^j}$ via Lemma \ref{lem:Hat-Alb-Hat-mathcalT}, along a suitable map of the ambient Albanese torus, which we discuss next.

\subsection{Homotopy}
For a real number $x$, 
we denote by $\lceil x \rceil $ the round-up of $x$.
For a vector $x\in \R^S$ we denote by $\lceil x \rceil $ the vector, given by applying $\lceil - \rceil $ componentwise.
We then consider the fundamental domain $[0,\ell ^j]^S$ of $\mathcal T_{\ell^j,\ell^j}=\R^S/\ell ^j\Z^S$ together with the continuous self-map
\begin{align} \label{eq:tilde-h}
\tilde h\colon [0,\ell ^j]^S\longrightarrow [0,\ell ^j]^S,
\end{align}
given by
$$
\tilde h(x)\coloneqq \begin{cases} 
 \ell^j\cdot x  &\text{if $x\in [0,1]^S$}\\
  \ell^j \lceil x/\ell ^j\rceil  &\text{otherwise}.
\end{cases}
$$ 

\begin{lemma} \label{lem:r}
The map $\tilde h$ descends to a continuous self-map  
$$
h\colon \mathcal T_{\ell^j,\ell^j}=\R^S/\ell^j\Z^S\longrightarrow \mathcal T_{\ell^j,\ell^j}=\R^S/\ell^j\Z^S
$$
which is homotopic to the identity.
\end{lemma}
\begin{proof} 
Since $\tilde h$ is defined componentwise, the lemma reduces to the claim that the self-map of the interval $[0,\ell^j]$, which maps $x\in [0,1]$ to $\ell^jx$ and maps $x\in [1,\ell ^j]$ to $\ell^j \lceil x/\ell ^j\rceil=\ell^j$, is continuous and descends to a self-map of the circle $\R/\ell ^j\Z$ which is homotopic to the identity.
This claim is clear and so the lemma follows.
\end{proof}

\begin{remark} \label{rem:h}
We can describe the map $h$ as follows: 
The map $\tilde h$ is induced by first 
retracting the cube $[0,\ell ^j]^S$ of side length 
$\ell^j$ suitably to the unit cube $[0,1]^S$  
and then expanding the unit cube to $[0,\ell^j]^S$ by multiplication by $\ell ^j$. See Figure \ref{fig:map-h}.
\end{remark}

\begin{figure}[h]
    \centering
    \includegraphics[width=0.5\linewidth]{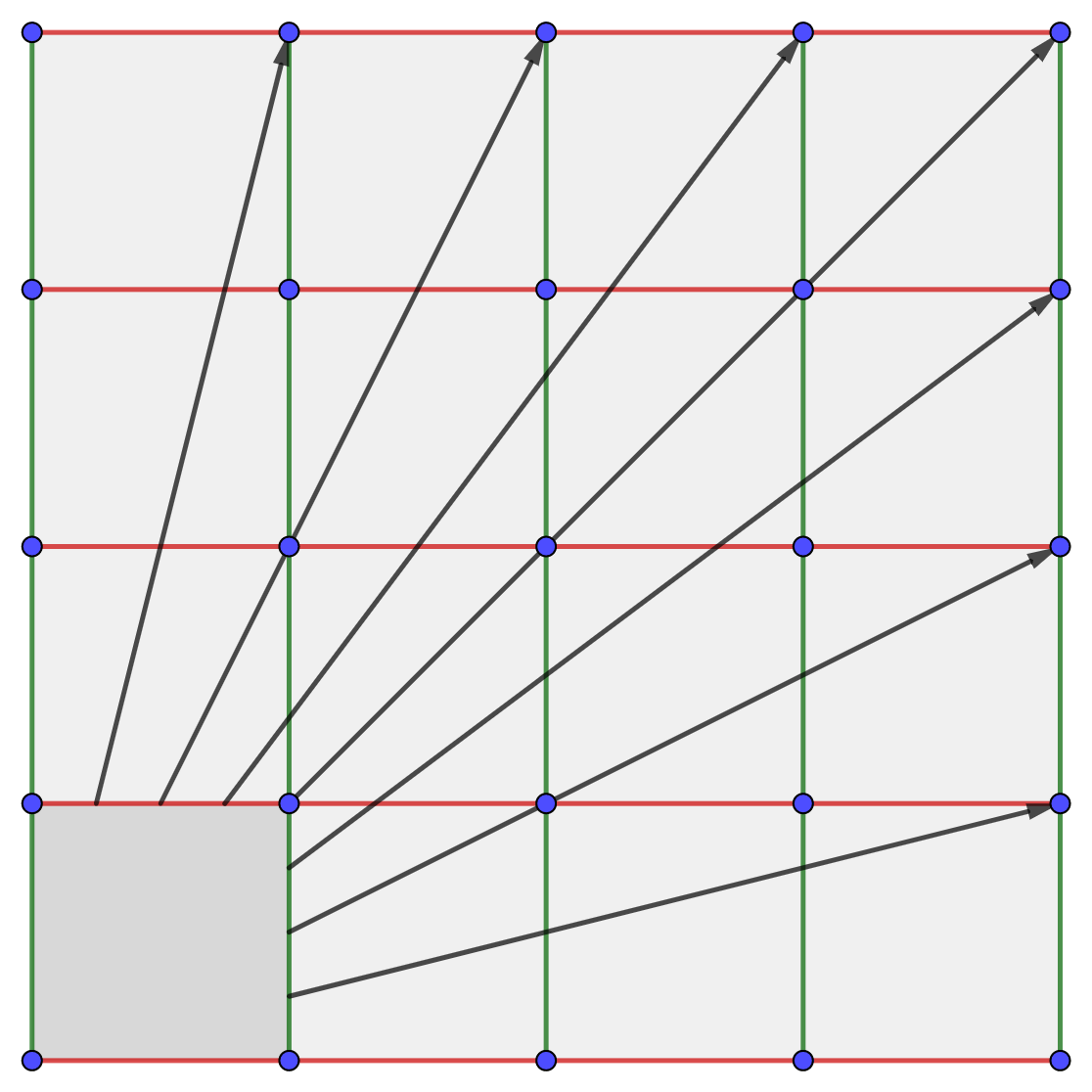}
    \caption{Depiction of the map
    $h\colon [0,2^2]^2\to [0,2^2]^2$ (i.e., $\ell=j=|S|=2$).}
    \label{fig:map-h}
\end{figure}

By Lemma \ref{lem:r}, the map $\tilde h$ from \eqref{eq:tilde-h} extends to a continuous map 
\begin{align} \label{eq:tilde-h-extension}
\tilde h\colon \R^S\longrightarrow \R^S ,
\end{align}
which is equivariant with respect to the natural $\ell^j \Z^S$-action on source and target that is given by translation; we denote this map by slight abuse of notation by the same letter.

Recall the graph $\Alb_{1,1}$, which is a graph with one vertex and one oriented edge for each $s\in S$; i.e.\ it is a wedge of $|S|$-many circles.
We consider the refinement $\widehat{\Alb}_{1,1}$ of $\Alb_{1,1}$ from Definition \ref{def:hat-Alb}.
By Lemma \ref{lem:Hat-Alb-Hat-mathcalT}, there is a canonical inclusion of oriented $S$-colored graphs
\begin{align} \label{eq:hat-Alb_1,1}
\widehat{\Alb}_{1,1}\longhookrightarrow \Alb_{\ell^j,\ell^j} 
\end{align}
which we use to identify $\widehat{\Alb}_{1,1}$ with a subgraph of $ \Alb_{\ell^j,\ell^j}$.
Via the inclusion $\Alb_{\ell^j,\ell^j}\hookrightarrow \mathcal T_{\ell^j,\ell^j}$, we thus get a canonical embedding of 
topological spaces
\begin{align} \label{eq:hat-Alb-11-hat-T}
\widehat{\Alb}_{1,1}\longhookrightarrow \mathcal T_{\ell^j,\ell^j}.
\end{align}
As before, we identify the refinement $\widehat \Alb_{\ell^j,\ell^j}$ of $\Alb_{\ell^,\ell^j}$ 
canonically with a subgraph of the $1$-skeleton of $\widehat{\mathcal T}_{\ell^j,\ell^j}$.
This yields a canonical embedding of topological spaces:
\begin{align} \label{eq:hat-Alb-lj-hat-T}
\widehat \Alb_{\ell^j,\ell^j}\longhookrightarrow \widehat{\mathcal T}_{\ell^j,\ell^j}\approx {\mathcal T}_{\ell^j,\ell^j} .
\end{align}

\begin{lemma} \label{lem:wegde-of-circles-image-h} 
The continuous map $h\colon {\mathcal T}_{\ell^j,\ell^j}\to {\mathcal T}_{\ell^j,\ell^j}$ restricts to a map of oriented $S$-colored graphs
$$
h\colon \widehat{\Alb}_{\ell^j,\ell ^j}\longrightarrow \widehat{\Alb}_{1,1} ,
$$
where both sides are viewed as subspaces of $\mathcal T_{\ell ^j,\ell ^j}$ via the inclusions \eqref{eq:hat-Alb-lj-hat-T} and \eqref{eq:hat-Alb-11-hat-T}, respectively. 
\end{lemma}
\begin{proof} 
This follows easily from the construction of $\tilde h$ and $h$.  
\end{proof}

\begin{lemma} \label{lem:covering-wedge-circles}
The following diagram is Cartesian:
$$
\xymatrix{
\widehat \Alb_{\ell^{r-j},1} \ar@{^{(}->}[r]^{\iota} \ar[d]&\Alb_{\ell^r,\ell ^j}\ar[d]^{\xi}\\
\widehat \Alb_{1,1} \ar@{^{(}->}[r]^{\iota}&\Alb_{\ell ^j,\ell ^j},
}
$$
where the horizontal maps are the embeddings from Lemma \ref{lem:Hat-Alb-Hat-mathcalT} and the vertical maps are the natural covering maps, respectively.
\end{lemma}
\begin{proof}
By Lemma \ref{lem:Hat-Alb-Hat-mathcalT} we have a canonical identification of polyhedral complexes $ \widehat{\mathcal T}_{\ell^{r-j},1}\cong \mathcal T_{\ell^r,\ell^j} $ and $ \widehat{\mathcal T}_{1,1}\cong \mathcal T_{\ell^j,\ell ^j} $.
Using this, we get a natural commutative diagram
$$
\xymatrix{
\widehat \Alb_{\ell^{r-j},1}\ar@{^{(}->}[r]\ar[d] & \ar@{}[dl] | {\Box} \widehat{\mathcal T}_{\ell^{r-j},1}\cong \mathcal T_{\ell^r,\ell^j}\ar[d]^{\xi} & \Alb_{\ell^r,\ell ^j} \ar@{_{(}->}[l] \ar[d]^{\xi} \ar@{}[dl] | {\Box} \\
\widehat {\Alb}_{1,1}\ar@{^{(}->}[r] & \widehat{\mathcal T}_{1,1}\cong \mathcal T_{\ell^j,\ell ^j} & \ar@{_{(}->}[l]\Alb_{\ell^{j},\ell^j} \\
}
$$
where the horizontal maps are the natural embeddings (which are maps of polyhedral complexes) and the vertical maps are induced by the covering map $\xi$.
Note moreover that the squares on the left and the right are Cartesian.
The claim in the lemma then follows because the image of the left lower horizontal map is contained in the image of the right lower horizontal one.
\end{proof}

\begin{definition} \label{def_H}
    We define $H\colon \mathcal T_{\ell^r,\ell ^j}\to \mathcal T_{\ell^r,\ell ^j}$ to be the map of topological spaces 
    that is induced by the equivariant map 
    $\tilde h\colon \R^S\to \R^S$ from \eqref{eq:tilde-h-extension}.
\end{definition}

\begin{lemma} \label{lem:def:H}
The map $H$ is homotopic to the identity and makes following diagram commutative:
$$
\xymatrix{
\mathcal T_{\ell^r,\ell ^j}\ar[r]^{H}\ar[d]^{\xi}&\mathcal T_{\ell^r,\ell ^j}\ar[d]^{\xi}\\
\mathcal T_{\ell ^j,\ell ^j}\ar[r]^{h}& \mathcal T_{\ell ^j,\ell ^j}.
}
$$
\end{lemma}
\begin{proof}
Commutativity of the diagram is clear, because $h$ is induced by $\tilde h$ as well.
The fact that $H$ is homotopic to the identity follows from the homotopy lifting property together with the fact that $h$ is homotopic to the identity, see Lemma \ref{lem:r}. 
\end{proof}

It follows from Lemmas \ref{lem:wegde-of-circles-image-h}, \ref{lem:covering-wedge-circles} and \ref{lem:def:H} that $H$ restricts to a 
morphism of oriented $S$-colored graphs
$$
H \colon \widehat \Alb_{\ell^r,\ell ^j}\longrightarrow \widehat \Alb_{\ell^{r-j},1} ,
$$
that we denote, by slight abuse of notation, by the same symbol.
This map has the following important property,
where we recall that $(\underline R,S)$ 
is loopless by the assumption in this section:

\begin{proposition} \label{prop:homotopy-solution} 
Let $(b_s)_{s\in S}$ be a $\Lambda$-solution of $(\underline R,S)$ in $\Alb_{\ell^r,\ell ^j}$.
Let $\hat b_s\in C_1(\widehat \Alb_{\ell^r,\ell ^j},\Lambda)$ be the corresponding $1$-chain obtained via refinement.
Then:
\begin{enumerate}
    \item \label{item:prop:homotopy-solution:1} The collection of $1$-chains
$$
H_\ast \hat b_s\in C_1(\widehat \Alb_{\ell^{r-j},1},\Lambda)
$$
with $s\in S$ is a $\Lambda$-solution of $(\underline R,S)$ in $\widehat \Alb_{\ell^{r-j},1}$.
    \item \label{item:prop:homotopy-solution:2} 
    The color profiles of $b_s$ and $H_\ast \hat b_s$ coincide: $\lambda(b_s)=\lambda(H_\ast \hat b_s )$.
\end{enumerate}
\end{proposition}
\begin{proof}
Since 
$$
H\colon \widehat \Alb_{\ell^r,\ell ^j}\longrightarrow \widehat \Alb_{\ell^{r-j},1} 
$$
is a map of $S$-colored graphs, and $(b_s)_{s\in S}$ is a $\Lambda$-solution, we find that for any $c_s\in \Lambda$, $s\in S$, with $\sum_s c_se_s\in U_\Lambda$, we have
$\sum_s c_sb_s\in H_1(\Alb_{\ell^r,\ell ^j},\Lambda)$.
Hence, 
$$
\sum_s c_s\hat b_s\in H_1(\widehat \Alb_{\ell^r,\ell ^j},\Lambda) .
$$
We thus conclude
$$
H_\ast \sum_s c_s\hat b_s=\sum_s c_s H_\ast \hat b_s\in H_1(\widehat \Alb_{\ell^{r-j},1},\Lambda ) .
$$ 
This proves that the collection of $1$-chains $H_\ast \hat b_s$, $s\in S$, is a $\Lambda$-solution of $(\underline R,S)$ in $\widehat \Alb_{\ell^{r-j},1}$, as we want in item \eqref{item:prop:homotopy-solution:1} of the proposition.

It remains to prove item \eqref{item:prop:homotopy-solution:2}.
Note that the map $H \colon \widehat \Alb_{\ell^r,\ell ^j}\to \widehat \Alb_{\ell^{r-j},1}$ contracts edges and so it is not directly suitable to compare color profiles.
Instead, we argue as follows.
Let $s\in S$.
Since $(\underline R,S)$ is loopless, there is an element
$$
u_s=\sum_{t\in S }c_te_t\in U_\Lambda \quad \quad \text{with $c_s=1$}, 
$$
see  Lemma \ref{lem:integral-realization}. 
Since $(b_s)_{s\in S}$ is a solution, we conclude that
$$
\alpha\coloneqq \sum_{t\in S} c_tb_t=b_s+\sum_{t\in S\setminus \{s\}}c_tb_t\in H_1(\Alb_{\ell^r,\ell ^j},\Lambda)
$$
is closed.
The $s$-color profile $\lambda_s(\alpha)=\lambda_s(b_s)$ can then be computed as follows.
 
Consider the composition
$$
f \colon \Alb_{\ell^r,\ell ^j}\longhookrightarrow \mathcal T_{\ell^r,\ell ^j}\stackrel{\xi}\longrightarrow  \mathcal T_{\ell^j,\ell ^j} =\R^S/\ell^j\Z^S \stackrel{\pr_s}\longrightarrow (\R/\ell ^j\Z)\cdot e_s .
$$ 
This identifies to a map of $S$-colored oriented graphs, where $(\R/\ell ^j\Z)\cdot e_s$ denotes the graph that consists of a chain of $\ell^j$ edges of color $s$ and of the same orientation, where the starting and end point of the chain are glued to give a circle.
Moreover, the map $f$ does not contract any edge of color $s$.
Hence,
$$
\lambda_s(\alpha)=\lambda_s(f_\ast \alpha)=\ell ^j \mu ,
$$
where $\mu\in \Lambda$ is the unique element such that
$$
f_\ast \alpha=\mu\cdot [(\R/\ell ^j\Z)\cdot e_s]\in H_1((\R/\ell ^j\Z)\cdot e_s,\Lambda),
$$
where $[(\R/\ell ^j\Z)\cdot e_s]$ denotes the positively oriented  generator of $H_1((\R/\ell ^j\Z)\cdot e_s,\Lambda)$.

Let $f' \colon \widehat \Alb_{\ell^r,\ell ^j} \to (\R/\ell^j\Z)\cdot e_s$ (resp.\ $f'' \colon \widehat \Alb_{\ell^{r-j},1} \to (\R/\ell^j\Z)\cdot e_s$) be the composition of the inclusion $\widehat \Alb_{\ell^r,\ell ^j} \hookrightarrow \widehat{\mathcal T}_{\ell^{r},\ell^j} \approx \mathcal T_{\ell^r, \ell^j}$ (resp.\ $\widehat \Alb_{\ell^{r-j},1} \hookrightarrow  \widehat{\mathcal T}_{\ell^{r-j},1}\cong \mathcal T_{\ell^r,\ell^j}$) with the map $\pr_s \circ \xi \colon \mathcal T_{\ell^r, \ell^j} \to (\R/\ell^j\Z) \cdot e_s$.
Note that
\begin{align} \label{eq:f-ast-alpha}
f_\ast \alpha= f'_\ast \sum_{t\in S} c_t\hat b_t =f''_\ast H_\ast \sum_{t\in S} c_t\hat b_t \in  H_1((\R/\ell ^j\Z)\cdot e_s,\Lambda),
\end{align}
where the first equality uses that $\alpha= \sum_{t\in S} c_t b_t $ and $\sum_{t\in S} c_t\hat b_t$ yield the same class in singular homology of $\mathcal T_{\ell ^r,\ell ^j}$, and the second equality uses that $H$, when viewed as a map $ T_{\ell^r,\ell ^j}\to \mathcal T_{\ell^r,\ell ^j}$, is homotopic to the identity. 
More precisely, the second equality follows from the commutative diagram 
\[
\xymatrix{
\widehat \Alb_{\ell^r, \ell^j} 
\ar[r]\ar[d]^H & \mathcal T_{\ell^r, \ell^j}\ar[d]^H \ar[r] & \mathcal T_{\ell^j, \ell^j} \ar[d]^h \ar[r] & (\R / \ell^j\Z) \cdot e_s \ar[d]^{h_s} 
\\
\widehat \Alb_{\ell^{r-j},1}
\ar[r] & \mathcal T_{\ell^r, \ell^j} \ar[r]& \mathcal T_{\ell^j, \ell^j} \ar[r] & (\R / \ell^j\Z) \cdot e_s,
}
\]
where the rightmost vertical arrow $h_s$
is induced by $\tilde h$ from \eqref{eq:tilde-h} 
(for $S = \set{s}$) hence homotopic to the identity,
and the two horizontal compositions coincide with $f'$ and $f''$ respectively. 
By \eqref{eq:f-ast-alpha} and  the above description of $\lambda_s$, we then deduce 
that 
$$\textstyle 
\lambda_s\left( \alpha \right)=
\lambda_s\left( H_\ast \sum_{t\in S} c_t\hat b_t\right).
$$ 
Hence,
$$\textstyle 
\lambda_s(b_s)=\lambda_s(\alpha)=
\lambda_s\left( H_\ast \sum_{t\in S} c_t\hat b_t\right) 
=\lambda_s\left( H_\ast \hat b_s\right),
$$
where the last equality uses that $H_\ast b_t$ is a $1$-chain of color $t$, because $H\colon \widehat{\Alb}_{\ell^r,\ell ^j}\to \widehat{\Alb}_{\ell^{r-j},1}$ is a map of $S$-colored graphs.
This proves item \eqref{item:prop:homotopy-solution:2} in the proposition and hence concludes the proof.
\end{proof}

We are now in the position to finish the proof of Theorem \ref{thm:reduction-to-2-solution}.

\begin{proof}[Proof of Theorem \ref{thm:reduction-to-2-solution}]
By Lemma \ref{lem:Hat-Alb-solutions}, it suffices to produce  an $\ell^{i}$-indivisible $\Lambda$-solution in $\widehat{\Alb}_{\ell^{r-j},1}$ (which is constant modulo $\ell^r$).
This is achieved by Proposition \ref{prop:homotopy-solution}. 
\end{proof}

\section{From $2$-indivisible solutions to cographic matroids} \label{sec:2-dindivisible->cographic}

In this section we prove the following, where we recall that $\Alb_{2}(\underline R)\coloneq \Alb_{2,1}(\underline R)$:

\begin{theorem} \label{thm:2-indivisible-solution=>cographic}
Let $(\underline R,S)$ be a regular matroid with integral realization $S\to U^\ast$.
Assume that $(\underline R,S)$ admits a $2$-indivisible $\Z/2$-solution in $\Alb_{2}(\underline R)$.
Then $(\underline R,S)$ is cographic.
\end{theorem}

\subsection{Closed under taking minors}

In this section we show that having $\ell^i$-indivisible $\Lambda$-solutions in $\Alb_{\ell^r,\ell ^j}$ 
is a condition that is closed under taking minors. 

\begin{proposition} \label{prop:closed-under-minors} 
Let $(\underline R',S')$ be a minor of a regular matroid $(\underline R,S)$.
Let $\ell$ be a prime number and let $0\leq j\leq i \leq r$ be integers.
Let $\Lambda$ be a $\Z_{(\ell)}$-algebra.
If $(\underline R,S)$ admits an $\ell^i$-indivisible $\Lambda$-solution 
$(b_s)_{s\in S}$ 
in $\Alb_{\ell^r,\ell^j}(\underline R)$, then $(\underline R',S')$ admits an $\ell^i$-indivisible $\Lambda$-solution $(b_s')_{s\in S'}$ 
in  $\Alb_{\ell^r,\ell^j}(\underline R')$. If, moreover, $(b_s)_{s\in S}$ is constant modulo $\ell^r$, then we can arrange that $(b'_s)_{s\in S'}$ is also constant modulo $\ell^{r}$.
\end{proposition}

\begin{proof}
Note first that the question of whether $(\underline R,S)$ has an $\ell^i$-indivisible $\Lambda$-solution in $\Alb_{\ell^r,\ell^j}(\underline R)$ is by Lemma \ref{lem:unique-integer-realization} independent of the chosen realization, see Lemma \ref{lem:Alb-well-defined} and Remark \ref{rem:solutions-well-defined}. 
It thus suffices to prove the proposition for the integral realization $S'\to (U')^\ast$ of  $(\underline R',S')$  that is induced by a given integral realization $S\to U^\ast$, $s\mapsto y_s$, of $(\underline R,S)$.

By the definition of a minor, it suffices to prove the proposition in the case where  $(\underline R',S')$  is a deletion or a contraction of $(\underline R,S)$.

\begin{case} \label{case:closed-minors:1}
$S'\subset S$ and  $(\underline R',S')$ is the deletion of $\underline R$ to $S'$.
\end{case}

In this case, let $(U')^\ast\subset U^\ast$ be the span of the linear forms $y_s$ with $s\in S'$.
Then the induced integral realization of $(\underline R',S')$ is given by $S'\to(U')^\ast$, $s\mapsto y_s$.
The dual of this yields an embedding $U'\hookrightarrow \Z^{S'}$.
With respect to this realization, $\Alb_{\ell^r,\ell^j}(\underline R')$ can be realized as the 1-skeleton 
$$
\Alb_{\ell^r,\ell^j}(\underline R')\subset \R^{S'}/(\ell^jU'+\ell^r\Z^{S'})
$$
of $\R^{S'}/(\ell^jU'+\ell^r\Z^{S'})$, see Remark \ref{rem:Alb-torus-polyhedral}.
A similar description holds for
$$
\Alb_{\ell^r,\ell^j}(\underline R)\subset \R^{S}/(\ell^jU+\ell^r\Z^{S}) .
$$
This description shows that the natural projection map $\Z^S\to \Z^{S'}$, which maps $U$ into $U'$,
 induces a map of oriented $S$-colored graphs
$$
\pi\colon \Alb_{\ell^r,\ell^j}(\underline R)\longrightarrow \Alb_{\ell^r,\ell^j}(\underline R'),
$$
where we view the $S'$-colored graph $\Alb_{\ell^r,\ell^j}(\underline R')$ as an $S$-colored graph (without any edge of color $s\in S\setminus S'$). 
This map does not contract any edge of color $s\in S'$ and so $\lambda(\gamma )=\lambda(\pi_\ast \gamma)$ for any $1$-chain $\gamma$ of color $s\in S'$.
It follows that if $(b_s)_{s\in S}$ is an $\ell^i$-indivisible $\Lambda$-solution of $\Alb_{\ell^r,\ell^j}(\underline R)$ (which is constant modulo $\ell^r$), then $(\pi_\ast b_s)_{s\in S'}$ is an $\ell^i$-indivisible $\Lambda$-solution of $\Alb_{\ell^r,\ell^j}(\underline R')$ (which is constant modulo $\ell^r$).
This concludes the proof in Case \ref{case:closed-minors:1}.

\begin{case} \label{case:closed-minors:2}
$S'=S\setminus T$ for an independent set $T$ of 
$(\underline R,S)$ and $(\underline R',S')$ 
is the contraction of 
$\underline R$ by $T$.
\end{case}

Consider the realization $S\to U^\ast$ of $\underline R$ and let $\langle T\rangle \subset U^\ast$ be the span of $T$.
The induced realization of $\underline R'$ is then given by $S'\to {U'}^\ast$, where ${U'}^\ast=U^\ast/\langle T\rangle $ is a quotient of $U^\ast$.
Dualizing this quotient map, we get an inclusion $U'\hookrightarrow U$, where
$$
U'=U\cap \Z^{S'}
$$
is the intersection of $U$ with the subspace $\Z^{S'}\subset \Z^S$.
This description shows that there is a natural embedding of Albanese tori
$$
\R^{S'}/(\ell^jU'+\ell^r\Z^{S'})\longhookrightarrow \R^{S}/(\ell^jU+\ell^r\Z^{S})
$$
with associated embedding on Albanese graphs
\begin{align} \label{eq:iota:subgraph-minors}
\iota \colon \Alb_{\ell ^r,\ell ^j}(\underline R')\longhookrightarrow \Alb_{\ell ^r,\ell ^j}(\underline R) .
\end{align}
Each element $[x]\in V\coloneq  \Z^S/(\ell^jU+\ell^r \Z^S)$ acts via translation on $\R^{S}/(\ell^jU+\ell^r\Z^{S})$.
The corresponding action restricts to a self-map of the oriented $S$-colored graph $\Alb_{\ell ^r,\ell ^j}(\underline R)$.
We may thus consider the translates of the image of $\iota$:
$$
G\coloneq \bigcup_{[x]\in V}\Alb_{\ell ^r,\ell ^j}(\underline R')+[x] \longhookrightarrow \Alb_{\ell ^r,\ell ^j}(\underline R) .
$$
We identify $G$ with a subgraph of $\Alb_{\ell ^r,\ell ^j}(\underline R)$ and note that any edge of $\Alb_{\ell ^r,\ell ^j}(\underline R)$ of color $s\in S'$ is contained in $G$. 
Note moreover that there is a natural map of oriented $S$-colored graphs
$$
\pi\colon G\longrightarrow \Alb_{\ell ^r,\ell ^j}(\underline R'),
$$
whose restriction to $\Alb_{\ell ^r,\ell ^j}(\underline R')+[x]$ is given by translation by $[-x]$.

Let now $(b_s)_{s\in S}$ be an $\ell^i$-indivisible $\Lambda$-solution of $\Alb_{\ell ^r,\ell ^j}(\underline R)$ (which is constant modulo $\ell^r$).
We claim that the subcollection $(b_{s})_{s\in S'}$ is an  $\ell^i$-indivisible $\Lambda$-solution of $\underline R'$ in the $S'$-colored graph $\Alb_{\ell ^r,\ell ^j}(\underline R')$ (which is constant modulo $\ell^r$).
To prove this, let $\sum_{s\in S'} c_se_s\in U'_\Lambda$.
We then have
$$
\sum_{s\in S'} c_se_s=\sum_{s\in S'} c_se_s+\sum_{t\in T} 0\cdot e_t \in U_\Lambda .
$$
Since $(b_s)_{s\in S}$ is a $\Lambda$-solution of $\Alb_{\ell ^r,\ell ^j}(\underline R)$, we find that the $1$-chain
$$
\sum_{s\in S'} c_sb_s \in C_1(\Alb_{\ell ^r,\ell ^j}(\underline R),\Lambda )
$$
is closed.
This is a sum of $1$-chains of colors contained in $S'$, hence it is supported on the subgraph $G$ above.
We may then consider the map $\pi\colon G\to \Alb_{\ell ^r,\ell ^j}(\underline R')$ from above and we find that
 the collection of $1$-chains $(\pi_\ast b_{s})_{s\in S'}$ on $\Alb_{\ell ^r,\ell ^j}(\underline R')$ is a $\Lambda$-solution of $\underline R'$ in the $S'$-colored graph $\Alb_{\ell ^r,\ell ^j}(\underline R')$.
This solution is $\ell^i$-indivisible (and constant modulo $\ell^r$) because the same holds for   $(b_s)_{s\in S}$ by assumption and because $\pi$ is a map of oriented $S$-colored graphs that does not contract any edge.
This concludes the proof of the proposition. 
\end{proof}

\subsection{Reduced Albanese graph and excluded minors}

Recall from Remark \ref{rem:ell=2-multiple-edges} that the oriented $S$-colored graph $\Alb_{2}(\underline R)=\Alb_{2,1}(\underline R)$ has the property that for each edge of color $s$ which points from a vertex $v$ to another vertex $w$, there is also an edge of the same color between the same vertices which points in the other direction.
For computations with $\bF_2=\Z/2$-homology, orientations do not play any role and so it will be convenient to contract these multiple edges, as follows.

\begin{definition} \label{def:reduced-Alb}
Let $(\underline R,S)$ be a regular matroid with integral realization $S\to U^\ast$.
The \emph{reduced Albanese graph} $\Alb^{\red}_{2}(\underline R)$ is the unique unoriented $S$-colored graph that admits a morphism of $S$-colored graphs $\pi\colon \Alb_{2}(\underline R)\to \Alb^{\red}_{2}(\underline R)$ which is an isomorphism on edges and which identifies parallel edges of the same color to a single edge.
\end{definition}

Since orientations play no role when working with $\F_2$-coefficients, we may for any $S$-colored graph $G$, such as $\Alb^{\red}_{2}(\underline R)$, define color profile maps $\lambda\colon C_1(G,\F_2)\to \F_2^S$ as well as $\F_2$-solutions in $G$ (and $2$-indivisible $\bF_2$-solutions) analogously to Definitions \ref{def:color-profile-map} and \ref{def:solution}.

\begin{lemma} \label{lem:solutions-Alb->Alb-red}
Let $(\underline R,S)$ be a regular matroid with integral realization $S\to U^\ast$.
If $\underline R$ admits a $2$-indivisible $\bF_2$-solution in $\Alb_{2}(\underline R)$, then it also admits a  $2$-indivisible $\bF_2$-solution in $\Alb_{2}^\red(\underline R)$.
\end{lemma}
\begin{proof}
This follows from the fact that 
we may pushforward any  $2$-indivisible $\bF_2$-solution in $\Alb_{2}(\underline R)$ via the map of $S$-colored graphs $\pi\colon \Alb_{2}(\underline R)\to \Alb^{\red}_{2}(\underline R)$ 
(which does not contract any edge) to get a 
$2$-indivisible $\bF_2$-solution in $\Alb_{2}^\red(\underline R)$.
\end{proof}

Recall the complete graph $K_5$ on 5 vertices with $\binom{5}{2}=10$ edges, and the utility graph $K_{3,3}$ with $6$ vertices and $3\cdot 3=9$ edges.
In what follows, we say that an $\F_2$-solution $(b_s)_{s\in S}$ in an $S$-colored graph $G$ is {\em $2$-divisible}, if $\lambda(b_s)=0\in \F_2^S$ for all $s\in S$.

\begin{proposition} \label{prop:K_5-K_3,3}
Let $(\underline R,S)$ be the graphic matroid $M(K_5)$ or $M(K_{3,3})$.
Then $(\underline R,S)$ does not admit a $2$-indivisible $\bF_2$-solution in $\Alb_{2}(\underline R)$ or $\Alb_{2}^{\red}(\underline R)$.
In fact, any $\bF_2$-solution $(b_s)_{s\in S}$ in $\Alb^{\red}_{2}(\underline R)$ is $2$-divisible.
\end{proposition}
\begin{proof}
By Lemma \ref{lem:Alb-well-defined}, 
the space of $\ell$-indivisible $\bF_\ell$-solutions 
in $\Alb_{\ell}(\underline R)$ is independent 
of the chosen realization $S\to U^\ast$. 
For a regular matroid of rank $g$ on $n$ elements,
the graph ${\rm Alb}_{\ell}$ has $\ell^{n-g}$ vertices 
and $n\cdot \ell^{n-g}$ edges. Ranging over a $g$ element
basis of $U$, condition
\eqref{item:def:solution-1} in Definition \ref{def:solution}
amounts to $g\cdot \ell^{n-g}$ conditions on the 
$\bF_\ell$-coefficients of the edges.
Thus, the space of $\bF_\ell$-solutions of $(\uR, S)$
in its Albanese graph is the right kernel of a
$g\cdot \ell^{n-g}\times n\cdot \ell^{n-g}$ 
matrix over $\F_\ell$.
Those solutions which are $\ell$-divisible, 
i.e.~lie in $\ker(\lambda)$, modulo $\ell$, 
are the right kernel of an augmented,
$(g\cdot \ell^{n-g}+n)\times 
n\cdot \ell^{n-g}$ matrix. 

Thus, to prove that $(\uR, S)$ does not admit
an $\ell$-indivisible $\bF_\ell$-solution in 
${\rm Alb}_{\ell}(\uR)$, it suffices to verify that the
matrix and its augmentation
have the same rank. For ${\rm Alb}_{2}$, 
Lemma \ref{lem:solutions-Alb->Alb-red} allows us to
further reduce the computational difficulty by
considering ${\rm Alb}_{2}^{\rm red}$ whose solutions
correspond to a submatrix with half the columns 
(i.e.~edges) as ${\rm Alb}_{2}$.

$M(K_{3,3})$ and $M(K_5)$ are, respectively, matroids
of rank $g=5$ on $n=9$ elements, and rank $g=4$ 
on $n=10$ elements.
Thus, the space of solutions for 
${\rm Alb}_{2}^{\rm red}$
are kernels of $80\times 72$ 
and $256\times 320$ matrices
over $\bF_2$ respectively, while the 
$2$-divisible solutions
are kernels of matrices of size
$89\times 72$ and $266\times 320$, respectively.

Computing the kernel in SAGE \cite{sagemath, code}, 
one finds that indeed, all solutions are 
$2$-divisible---these 
kernels are, respectively, of 
rank $15$ and $103$ over $\bF_2$
for $M(K_{3,3})$ and $M(K_5)$. 
This proves the proposition.
\end{proof}

\begin{proposition}\label{prop:R10-even-sol} 
The $\underline{R}_{10}$ matroid does not admit 
$2$-indivisible $\mathbb F_2$-solutions 
in $\Alb^{\red}_{2}(\underline R_{10})$.
In fact, any $\bF_2$-solution is $2$-divisible.
\end{proposition}

\begin{proof}  
The first claim follows from Proposition \ref{prop:closed-under-minors},
the fact that any $1$-element deletion of $\uR_{10}$ is 
$M(K_{3,3})$, and Proposition \ref{prop:K_5-K_3,3}.
One may also directly compute that all $\F_2$-solutions
in ${\rm Alb}_{2}^{\rm red}(\uR_{10})$ are 
$2$-divisible (the rank of the solution space 
over $\bF_2$ is $35$).
\end{proof}

\subsection{Proof of Theorem \ref{thm:2-indivisible-solution=>cographic}}

\begin{proof}[Proof of Theorem \ref{thm:2-indivisible-solution=>cographic}]
By a theorem of Tutte, see \cite[p.\ 441, Corollary 13.3.4]{oxley-matroids}, a regular matroid is cographic if and only if it does not have the graphic matroid associated to $K_5$ or $K_{3,3}$ as a minor.
The theorem thus follows from Propositions \ref{prop:closed-under-minors} and \ref{prop:K_5-K_3,3}.
\end{proof}

\section{Conclusions, applications, and discussion}  
\label{sec:application}

\subsection{Quadratic splittings in cographic matroids}

\begin{proof}[Proof of Theorem \ref{thm:reduction-to-combinatorics-intro}]
Let $B$ and $B^\star=B\setminus H$ be as in Section \ref{subsec:set-up} and let $\pi^\star \colon X^\star \to B^\star$ be a matroidal family of principally polarized abelian varieties associated to a regular matroid $(\underline R,S)$ with integral realization $S\to U^\ast$, see Definition \ref{def:matroidal-family}.
Let $\iota\colon C_t\to X_t$ be a non-constant morphism from a curve $C_t$ to a very general fiber $X_t$ of $\pi^\star$ with $\iota_\ast [C_t]=m[\Theta]^{g-1}/(g-1)!$.
Let $\ell$ be a prime that is coprime to $m$.
Note that there is an isomorphism of fields $ \overline{\C(B)}\cong \C$ under which the geometric generic fiber of $\pi^\star$ is identified with $X_t$. 
Hence, the geometric generic fiber of $\pi^\star$ contains a curve whose cohomology class is an $\ell$-prime multiple of the minimal class.
(Alternatively, this conclusion can also be derived from a standard Hilbert scheme argument.)
It thus follows from Theorem \ref{thm:algebraic->d-QE} and Remark \ref{rem:Lambda-splitting-graph-versus-general} that there is a positive integer $d$, such that the regular matroid $\underline R$ admits a quadratic $\Z_{(\ell)}$-splitting of level $d$ into a cographic matroid.
\end{proof}

\begin{proof}[Proof of Theorem \ref{thm:splitting-in-cographic=cographic-intro}]
Clearly, any cographic matroid admits, 
for any ring $\Lambda$ and any positive integer $d$,
a $\Lambda$-splitting of level $d$ into a cographic matroid.
Conversely, let $(\underline R,S)$ be a regular matroid which admits a $\Z_{(2)}$-splitting of some level $d$ in a cographic matroid associated to a graph $G$.
Up to performing some deletions, we may assume that $(\underline R,S)$ is loopless.
Moreover, by Lemma \ref{lem:Lambda-splitting-graph-versus-general}, we may without loss of generality assume that $(\underline R,S)$ admits a $\Z_{(2)}$-splitting of level $d$ in a graph (see Definition \ref{def:d-Lambda-splitting-graph}).
It thus follows from Theorems \ref{thm:d-QE->solutions}, \ref{thm:reduction-to-2-solution}, and \ref{thm:2-indivisible-solution=>cographic} that $\underline R$ is cographic. 
Here, we used that any $2$-indivisible $\Z_{(2)}$-solution in $\Alb_{2}(\underline R)$ naturally gives rise to a $2$-indivisible $\Z/2$-solution in $\Alb_{2}(\underline R)$.
This concludes the proof of the theorem.
\end{proof}

\begin{proof}[Proof of Theorem \ref{thm:splitting-in-cographic=cographic-intro-positive-result}]
Let $(\underline R,S)$ be a regular matroid of rank $g$. 
By \cite[Proposition 4.10]{survey}, we may then construct a family $\pi^\star \colon X^\star\to B^\star$ of principally polarized abelian varieties of dimension $g$ which, with respect to an snc extension $B^\star\subset B$ and a point $0\in B\setminus B^\star$ in the boundary, is a matroidal family associated to $\underline R$, see Definition \ref{def:matroidal-family}.
The very general fiber of $\pi^\star$ has the property that $(g-1)!$ times its minimal class is clearly represented by an algebraic curve (namely by the intersection of $g-1$ generic translates of $\Theta$).
It thus follows from Theorem \ref{thm:reduction-to-combinatorics-intro}
that, for any ring $\Lambda$ in which $(g-1)!$ is invertible, $(\underline R,S)$ admits a $\Lambda$-splitting of some level $d$ into a graph, and hence, by Lemma \ref{lem:Lambda-splitting-graph-versus-general}, into a cographic matroid.
This concludes the proof of the theorem.
\end{proof}

\subsection{Curves on very general fibers of matroidal families}

\begin{proof}[Proof of Theorem \ref{thm:matroidal-intro}]
The first paragraph of Section \ref{subsec:outline}
proves Theorem \ref{thm:matroidal-intro} as a formal consequence of Theorems \ref{thm:reduction-to-combinatorics-intro} and \ref{thm:splitting-in-cographic=cographic-intro}, proven above. 
\end{proof}

\begin{remark}
Recall from Remarks \ref{rem:matroidal-family:maximal-degeneration} and \ref{rem:matroidal-family:existence}, that for any regular matroid $\underline R$ of rank $g$ on an $n$-element ground set $S$, there exists a matroidal family of $g$-dimensional principally polarized abelian varieties over an 
$n$-dimensional base that is associated to $\underline R$, cf.\ \cite[Proposition 4.10]{survey}. 
Theorem \ref{thm:matroidal-intro} then shows that if $\underline R$ is not cographic, then the very general fiber of such a family does not contain a curve whose cohomology class is an odd multiple of the minimal class. 
In particular, the integral Hodge conjecture fails for the very general fiber, see Lemma \ref{lem:alg-min->l-prime-multiple-is-effective}.
\end{remark}

\begin{corollary}\label{cor:matroidal-topological-body}
Let $B$ be a smooth quasi-projective variety and let $H\subset B$ be an snc divisor which restricts to the coordinate hyperplanes on an embedded polydisc $\Delta^S\subset B$, centered at a distinguished point $0\in B$.
Let $B^\star=B\setminus H$ and let $\pi^\star\colon X^\star\to B^\star$ be a matroidal family of principally polarized abelian varieties, associated to a regular matroid $\underline R$. 
Then the restriction $ X^\star_{(\Delta^\star)^S}\to (\Delta^\star)^S$ to the punctured polydisc does not 
analytically deform to a family of Jacobians of curves if and only if $\underline R$ is not cographic, and these conditions imply that the integral Hodge conjecture fails for a very general fiber $X_t$. 
\end{corollary}

\begin{proof} 
By Theorem \ref{thm:matroidal-intro}, it suffices to show that $\underline R$ is cographic if and only if the restriction $ X^\star_{(\Delta^\star)^S}\to (\Delta^\star)^S$ to the punctured polydisc 
analytically deforms to a family of Jacobians of curves. 
If $\uR$ is cographic, then
$ X^\star_{(\Delta^\star)^S}\to (\Delta^\star)^S$ 
deforms to a family of Jacobians,
by \cite[Remark 2.31]{survey} 
and the fact that any cographic matroid can be realized via a matroidal family of Jacobians of curves, see e.g.~\cite[Example 4.13]{survey}.

Conversely, suppose 
$X^\star_{(\Delta^\star)^S}\to (\Delta^\star)^S$ 
deforms to a family of Jacobians.
The monodromy cone $\cC_\uR$ associated to $\uR$, 
see \cite[Definition 4.6]{survey}, is a cone
of the matroidal fan, see \cite[Remark 4.15]{survey}.
As this monodromy cone is unchanged under the deformation, we may assume that the image of
$(\Delta^\star)^S$ under the classifying map is contained in the Schottky
locus. 
The collection of cographic 
cones forms a subfan of the matroidal fan and defines a toroidal extension 
$\cA_g\hookrightarrow \cA_g^{\rm cogr}$.
As the closure of the 
Schottky locus in $\cA_g^{\rm cogr}$ is compact,
coinciding with the image of the Torelli map $\overline{\cM}_g\to \cA_g^{\rm cogr}$, $\cC_\uR$ is contained in the
union of all cographic cones. As
the matroidal fan is a fan (in particular,
the intersection of two cones is a face
of each), $\cC_\uR$ must
be a cographic cone. 
\end{proof}

\subsection{Application to IHC for abelian varieties} \label{subsec:applicationIHC}

\begin{proof}[Proof of Theorem \ref{thm:main:IHC:intro}]
Let $(X,\Theta)$ be a very general principally polarized abelian variety of dimension $g\geq 4$ and let $Z\subset X$ be an equidimensional closed subscheme of codimension $c$ with $2\leq c\leq g-1$.
Since $X$ is very general, the Mumford--Tate group of $X$ is maximal and so there is some non-negative integer $m$ with
 $$
[Z]=m\cdot [\Theta]^{c}/c!\in H^{2c}(X,\Z) .
$$ 
To prove the theorem, it suffices to show that $m$ is even.

Let us first assume that $g=4$ and $c=3$.
Let $(\underline R,S)$ be the graphic matroid $M(K_{5})$ with integral realization $S\to \Z^S/H_1(K_5,\Z)=U^\ast$, where $S$ denotes the set of edges of $K_5$, equipped with some orientation.
Then $S$ has $10$ elements, and ${\rm{rank}}(\underline R) =\dim U=4$.
By \cite[Proposition 4.10]{survey}, there is a matroidal family $\pi^\star\colon X^\star \to B^\star$ of principally polarized abelian fourfolds associated to $\underline R$, cf.\ Definition \ref{def:matroidal-family} and Remarks \ref{rem:matroidal-family:maximal-degeneration} and \ref{rem:matroidal-family:existence}.
Since $K_5$ is not planar, $M(K_5)$ is not cographic.
It thus follows from Theorem \ref{thm:matroidal-intro} that the very general fiber $X_t$ of $\pi^\star$ does not contain a curve whose cohomology class is an odd multiple of the minimal class.
This proves the case $g=4$ and $c=3$, because $X$ specializes to $X_t$.
(In fact, $X_t$ is a very general principally polarized abelian variety of dimension $4$, because $\dim B^\star=10$ and one can check that the moduli map $B^\star\to \mathcal A_4$ is dominant.)
The case $g=4$ and $c=2$ follows from the simple observation that $[\Theta]\cdot [\Theta]^2/2=[\Theta]^3/2$ is an odd multiple of the minimal curve class on $X$. 
This proves the theorem for $g=4$.

Let now $g\geq 5$ and let $(X,\Theta)$ be a very general principally polarized abelian variety of dimension $g$. 
We specialize $(X,\Theta)$ to a product $(Y,\Theta_Y)\times (E,\Theta_E)$ where $E$ is an elliptic curve and $(Y,\Theta_Y)$ is a very general principally polarized abelian variety of dimension $g-1$.
Let $p:Y\times E\to Y$ and $q:Y\times E\to E$ be the projections.
Let further $Z_0\subset Y\times E$ be the specialization of $Z\subset X$.
Then we find that
$$
[Z_0]=\frac{m}{c!}\cdot p^\ast [\Theta_Y]^c +\frac{m}{(c-1)!}\cdot p^\ast [\Theta_Y]^{c-1}\cup q^\ast [\Theta_E] \in H^{2c}(Y\times E,\Z) .
$$ 
Hence,
$$
p_\ast [Z_0]=m[\Theta_Y]^{c-1}/(c-1)!\in H^{2c-2}(Y,\Z)\quad \quad \text{and}\quad \quad \iota^\ast [Z_0]=m [\Theta_Y]^c/c!\in H^{2c}(Y,\Z),
$$
where $\iota \colon Y\hookrightarrow Y\times E$ denotes the inclusion.
Since $2\leq c\leq g-1$ and $(Y,\Theta_Y)$ is very general of dimension $g-1$, we conclude by induction that $m$ is even.
This concludes the proof of Theorem \ref{thm:main:IHC:intro}.
\end{proof}

\begin{proof}[Proof of Corollary \ref{cor:abelian-4-folds-IHC}]
Let $(X,\Theta)$ be a very general principally polarized abelian  
variety of dimension $g \in \set{4,5}$. 
By Theorem \ref{thm:main:IHC:intro}, it suffices to show that $2[\Theta]^c/c!$ is algebraic for $2 \leq c \leq g-1$.
This is clear for $c=2$ and follows for $c = g-1$ from the fact that $(X,\Theta)$ is a Prym variety.
It thus remains to deal with the case $g=5$ and $c=3$.
In this case, we use the Prym curve $C\subset X$ with class $[C]=2\cdot[\Theta]^4/4!$ and 
consider the sum map $C^{(2)}\to X$, whose image has, by a standard Pontryagin product computation, class 
$4 \cdot[\Theta^3]/3!
$.
Since $6\cdot [\Theta]^3/3!$ is algebraic as well, we find that $2\cdot [\Theta]^3/3!$ is algebraic. 
\end{proof}

\subsection{Application to cubic threefolds} \label{subsec:cubic-threefold-application}

\begin{proof}[Proof of Theorem \ref{thm:main:cubics:intro}]
Consider the Segre cubic threefold 
$$
Y_0\coloneq \left\{\sum_{i=0}^5x_i=\sum_{i=0}^5x_i^3=0 \right \}\subset \mathbb P^5,
$$
which is the unique cubic threefold with 10 nodes. 
Let $B$ be a smooth quasi-projective variety with an embedded polydisc $\Delta^{10}\subset B$ and an snc divisor $H\subset B$ that restricts to the coordinate hyperplanes on $\Delta^{10}$.
Choose $B$ such that there is a flat family of cubic threefolds $Y\to B$ which is smooth over $B^\star=B\setminus H$ and such that the fiber over $0$ is the Segre cubic $Y_0$ from above.
Assume moreover that $Y\to B$ has maximal variation in the sense that the restriction $Y_{\Delta^{10}}\to \Delta^{10}$ is a universal deformation of $Y_0$ (note that ${\rm Def}_{Y_0}=\Delta^{10}$ and such deformations $Y$
exist, for instance, over a suitable 
\'etale open neighborhood of $0\in \mathbb A^{10}$).
Let $\pi^\star \colon X^\star\coloneq JY^\star\to B^\star$ be the associated smooth projective family of intermediate Jacobians.

It follows from a result of Gwena \cite{gwena} that $\pi^\star$ is a matroidal family, associated to the matroid $\underline R_{10}$, see also \cite[Examples 2.16 and 4.7]{survey}.
Since $\underline R_{10}$ is not cographic, we conclude from Theorem \ref{thm:matroidal-intro} that the very general fiber $JY_t$ of $JY^\star\to (\Delta^\star)^{10}$ does not contain a curve whose cohomology class is an odd multiple of the minimal class.  
This concludes the proof of the theorem 
because $JY_t$ has Picard rank $1$.
\end{proof}

\begin{proof}[Proof of Corollary \ref{cor:cubic-intro}]
By Theorem \ref{thm:main:cubics:intro}, the very general cubic threefold $Y$ has the property that the minimal curve class of its intermediate Jacobian $JY$ is not algebraic.
This implies by \cite{voisin-universalCHgroup} that $Y$ does not admit a decomposition of the diagonal.
Hence, $Y$ is neither stably nor retract rational, nor $\bA^1$-connected, see for instance \cite{lange-Sch} and the references therein.
\end{proof}

\begin{proof}[Proof of Corollary \ref{cor:isogeny-product}]
Let $X$ and $Y$ be abelian varieties and let $f\colon X\times Y\to \prod_ iJC_i$ be an isogeny to a product of Jacobians of curves.
Assume that $f$ has odd degree.
Then $f_\ast \colon H_2(X\times Y,\Z_{(2)})\to H_2 (\prod_iJC_i,\Z_{(2)})$ induces an isomorphism.
By \cite{beckmann-degaayfortman}, the integral Hodge conjecture holds for curves classes on $\prod_iJC_i$.
Since $H_2(X,\Z_{(2)})$ admits a split embedding in $H_2(X\times Y,\Z_{(2)})$, it follows that any $\Z_{(2)}$-linear combination of Hodge classes in $H_2(X,\Z)$ is a $\Z_{(2)}$-linear combination of algebraic classes.
In the special case where $X$ is a $g$-dimensional principally polarized abelian variety with theta divisor $\Theta$, this implies that an odd multiple of $[\Theta]^{g-1}/(g-1)!$ is algebraic.
The corollary thus follows from Theorems \ref{thm:main:IHC:intro} and \ref{thm:main:cubics:intro}.
\end{proof}

\subsection{Discussions and open problems}

\begin{definition}
Let $\ell$ be a prime number.
We define $\mathcal M_\ell$ as the class of matroids consisting of all regular matroids $(\underline R,S)$ that admit a $\Z/\ell$-solution $(b_s)_{s\in S}$ in $\Alb_{\ell}(\underline R)$ with constant nonzero color profile $\lambda(b_s)=c\in (\Z/\ell)^\ast$ for all $s\in S$.
\end{definition}
 
\begin{remark}
By Proposition \ref{prop:closed-under-minors}, $\mathcal M_\ell$ is closed under taking minors.
\end{remark}

\begin{remark}
In the above definition, we could drop the condition on the constancy of the color profile and only ask that the color profile is not zero, i.e.~$\lambda_s(b_s)\neq 0\in \Z/\ell$ for some $s\in S$.
This gives rise to a class of matroids $\widetilde {\mathcal M_\ell}$, closed under taking minors by Proposition \ref{prop:closed-under-minors}, which satisfies $\mathcal M_\ell\subset \widetilde {\mathcal M}_\ell$.
The results of this paper show that in fact $\mathcal M_2=\widetilde {\mathcal M}_2$ is the class of cographic matroids, because both classes are minor-closed, contain the cographic matroids, but do not contain $M(K_5)$ and $M(K_{3,3})$. 
In general, it is easier to show that a matroid does not lie in $\mathcal M_\ell$ than it is to show that it does not lie in $\widetilde {\mathcal M}_\ell$. While we do not make use of this observation in the current paper, it will be important in the consecutive paper \cite{engel2025optimalityprymtyurinconstructionmathcala6}, where we show $M(K_7)\notin \mathcal M_3$ (while we do not know whether this matroid lies in $\widetilde{\mathcal M_3}$), cf.~Remark \ref{remark:constancy}.
\end{remark}

\begin{remark} \label{rem:class-C_l--l-geq-g}
By Lemma \ref{lem:Lambda-splitting-graph-versus-general} together with Theorems \ref{thm:d-QE->solutions} and \ref{thm:reduction-to-2-solution}, we see that any regular matroid $\underline R$ that admits a $\Z_{(\ell)}$-splitting of some level $d$ in a cographic matroid is contained in $\mathcal M_\ell$.
Theorem \ref{thm:splitting-in-cographic=cographic-intro-positive-result} thus implies that $\underline R\in \mathcal M_\ell$ for all $\ell\geq g$, where $g$ denotes the rank of $\underline R$.
\end{remark}

\begin{remark} \label{rem:K5,K33,R10-ell-odd} 
If $\ell$ is odd, then one can check that $M(K_5)$ and $M(K_{3,3})$ are contained in $\mathcal M_\ell$ and so is, for instance, $\underline R_{10}$. 
Indeed, by Theorem \ref{thm:splitting-in-cographic=cographic-intro} and Remark \ref{rem:class-C_l--l-geq-g}, 
it suffices to prove these claims for the prime $\ell=3$, which can be done  
via a similar computation as in Proposition \ref{prop:K_5-K_3,3} and \ref{prop:R10-even-sol}.
Alternatively, one can observe that the relevant matroidal families of principally polarized abelian varieties can be realized as Prym varieties and so the Prym curve yields a curve whose cohomology class is twice the minimal class.
Using this we can argue as in the proof of Theorem \ref{thm:splitting-in-cographic=cographic-intro} to conclude.
\end{remark}

We have just proven that $\mathcal M_2$ is the class of cographic matroids; it can be characterized by Tutte's theorem via two excluded minors: the graphic matroids associated to $K_5$ and $K_{3,3}$. 
The class $\mathcal M_\ell$  of regular matroids is by Proposition \ref{prop:closed-under-minors}  closed under taking minors. Hence, by the Robertson-Seymour theorem and its conjectural extension to regular matroids, $\mathcal M_\ell$ is expected to be determined by a finite list of excluded minors. 
This leads to the following natural open problem.

\begin{problem}
Let $\ell$ be an odd prime.
Characterize the subclass $\mathcal M_\ell$ of the class of regular matroids via a finite list of excluded minors.
\end{problem} 

\begin{definition} \label{def:distance}
Let $(\underline R,S)$ be a regular matroid.
The {\it radical distance} of $(\underline R,S)$ to the class of cographic matroids is the integer
$$
d(\underline R)\coloneq {\rm lcm}\{\ell \mid \underline R\notin \mathcal M_\ell\}
$$
\end{definition}

By Theorem \ref{thm:2-indivisible-solution=>cographic}, $d(\underline R)=1$ if and only if $\underline R$ is cographic.
By Remark \ref{rem:class-C_l--l-geq-g} (resp.\ Theorem \ref{thm:splitting-in-cographic=cographic-intro-positive-result}), $d(\underline R)$ is always finite and universally bounded by the product of all primes which are smaller than ${\rm rank}(\underline R)$.

Some of the main results of this paper may then be summarized as follows.

\begin{theorem} \label{thm:distance-1}
Let $(\underline R,S)$ be a regular matroid.
Let $\pi^\star \colon X^\star \to B^\star$ be a matroidal family of principally polarized abelian varieties associated to a  $(\underline R,S)$, see Definition \ref{def:matroidal-family} and Remark \ref{rem:matroidal-family:existence}.
Let $C_t\subset X_t$ be a curve on a very general fiber of $\pi^\star$ with
$$
[C_t]=m\cdot [\Theta]^{g-1}/(g-1)! \in H_2(X_t,\Z) .
$$
Then $m$ is divisible by $d(\underline R)$.
\end{theorem}
\begin{proof}
By the Chinese remainder theorem, it suffices to show $m\equiv 0\bmod \ell$ for all primes $\ell$ with $\underline R\not \in \mathcal M_\ell$.
For a contradiction, assume $\underline R \not \in \mathcal M_\ell$ and $m$ is coprime to $\ell$.
Then, Theorems \ref{thm:algebraic->d-QE}, \ref{thm:d-QE->solutions}, and \ref{thm:reduction-to-2-solution} show that $\underline R\in \mathcal M_\ell$, which contradicts our assumptions.
\end{proof}

\begin{proposition}
Let $K_{3,5}$ be the complete bipartite graph on $3+5=8$ vertices. 
Then the graphic matroid
$M(K_{3,5})$ is an excluded minor 
for the class $\cM_3$. In particular, 
for $g \geq 7$,
the minimal multiple of the minimal class which is algebraic on a very general
principally polarized abelian $g$-fold is at least $6$.
\end{proposition}
\begin{proof}
From the 
computations in \cite{code}, 
$M(K_{3,5})\notin \cM_3$ and any
$1$-element deletion or contraction
of $M(K_{3,5})$ lies in $\cM_3$. Thus
$M(K_{3,5})$ is an excluded minor for the class
$\cM_3$. Since $\rank M(K_{3,5})=7$, the proposition 
follows from Remark \ref{rem:matroidal-family:existence}, Lemma \ref{lem:alg-min->l-prime-multiple-is-effective}, Theorem \ref{thm:distance-1} and Proposition \ref{prop:closed-under-minors}; indeed $M(K_{3,5})$
is a minor of $M(K_{3,n})$ (of rank $2+n$),
for all $n\geq 5$. 
(Alternatively, one can use a specialization/induction argument, as in the proof of Theorem \ref{thm:main:IHC:intro}, to reduce to the case $g=7$.)
\end{proof}
 
\begin{remark} In a sequel to this paper, see \cite{engel2025optimalityprymtyurinconstructionmathcala6}, 
we prove the same
result for principally polarized abelian $6$-folds,
where one also has an upper bound of $6$ by \cite{adfio}. \end{remark}

We conclude with the following natural combinatorial problems.

\begin{problem}\label{problem:d(R)-unbounded}
Are there regular matroids $\underline R$ whose radical distance $d(\underline R)$ from the class of cographic matroids is arbitrarily large? 
\end{problem}

\begin{problem}
What is (the asymptotic behavior of) 
the function 
$$
g\mapsto d(g) \coloneqq {\rm lcm}\{d(\underline R)\mid \text{$\underline R$ is a regular matroid of rank $g$}\} ?
$$
\end{problem}

Note that Problem \ref{problem:d(R)-unbounded} asks precisely whether $d(g)$ is unbounded in $g$.
The interest in the function $d(g)$ lies in the following result:

\begin{theorem} \label{thm:distance-2}
The integer
$d(g)$ divides the minimal positive integer multiple of the minimal class which is algebraic on a very general principally polarized abelian variety of dimension $g$. 
\end{theorem}
\begin{proof}
Let $C\subset X$ be a curve on a very general principally polarized abelian variety $(X,\Theta)$ of dimension $g$.
Then $[C]=m\cdot [\Theta]^{g-1}/(g-1)!$ for some integer $m$ and it suffices to show that $m$ is divisible by $d(g)$. This
follows from Theorem \ref{thm:distance-1} together with the existence of matroidal families of $g$-dimensional principally polarized abelian varieties associated to any regular matroid $\underline R$ of rank $g$, see \cite[Proposition 4.10]{survey}.
\end{proof}

\printbibliography

\end{document}